\documentclass{amsart}
\usepackage{graphicx} 

\usepackage{amsmath,graphicx}
\usepackage{adjustbox}
\usepackage{amssymb}

\usepackage{tikz, tikz-cd}
\usepackage{fullpage}
\usepackage[all]{xy}
\usepackage[margin=1in]{geometry}
\usepackage{amsthm}
\usepackage{amsfonts}
\usepackage{bbm}
\usepackage{comment}
\usepackage{amsmath}
\usepackage{amsthm,diagbox}
\usepackage{bbm}
\usepackage{bm}

\usepackage{array} 
\usepackage{color}
\usepackage[utf8]{inputenc}
\usepackage[english]{babel}
\usepackage{hyperref}
\usepackage{graphics} 
\usepackage{subcaption} 
\usepackage{thm-restate}

\newcommand{\C}{{\mathbb{C}}}

\newcommand{\F}{{\mathbb{F}}}

\newcommand{\Q}{{\mathbb{Q}}}

\newcommand{\R}{{\mathbb{R}}}
\newcommand{\Z}{{\mathbb{Z}}}

\newcommand{\im}{\text{im}}

\newcommand{\Sym}{\text{Sym}}

\newcommand{\T}{\bm{T}}

\newcommand{\FDTC}{\text{FDTC}}

\DeclareMathOperator{\HFL}{HFL}
\DeclareMathOperator{\HFK}{HFK}

\DeclareMathOperator{\HF}{HF}

\DeclareMathOperator{\AKh}{AKh}

\DeclareMathOperator{\CFL}{CFL}

\DeclareMathOperator{\rank}{rank}

\newtheorem{theorem}{Theorem}[section]
\newtheorem{lemma}[theorem]{Lemma}

\newtheorem{corollary}[theorem]{Corollary}

\newtheorem{proposition}[theorem]{Proposition}
\theoremstyle{definition}  
\newtheorem{definition} [theorem] {Definition}

\newtheorem{remark} [theorem] {Remark}
\newtheorem{question} [theorem] {Question}

\theoremstyle{definition}

\usepackage[pagewise]{lineno}


\title{Fractional Dehn twist Coefficients and rank bounds for categorified link invariants}
\author{Fraser Binns}
\address{Department of Mathematics, Princeton University, Princeton, NJ, USA}
\email{fb1673@princeton.edu}
\author{Diana Hubbard}
\address{Brooklyn College, New York, NY, USA}
\email{diana.hubbard@brooklyn.cuny.edu}

\date{\today}

\begin{document}

\begin{abstract}
We give two new lower bounds on the rank of categorified link invariants: one on the link Floer homology of fibered links in terms of the fractional Dehn twist coefficients of their monodromies, and another, as a corollary, on the annular Khovanov homology of braid closures in terms of the fractional Dehn twist coefficient of the braid. The most important technical component of the proof is that we determine the behaviour of the link Floer homology of fibered links under adding boundary Dehn twists to their monodromies in the next to top Alexander grading.
\end{abstract}
\maketitle

\section{Introduction}\label{sec:intro}

One of the main problems in the study of knots and links is to understand what topological information a given knot or link invariant encodes. This paper explores that question for two categorified invariants, namely link Floer homology and annular Khovanov homology, by providing rank bounds for these invariants arising from a quantity related to the study of mapping class groups. \emph{Link Floer homology} is a link invariant due to Ozsv\'ath and Szab\'o~\cite{HolomorphicdiskslinkinvariantsandthemultivariableAlexanderpolynomial}, and is a vector space valued invariant of links defined using symplectic topology. \textit{Annular Khovanov homology}, due to Asaeda, Przytycki, and Sikora~\cite{asaeda_categorification_2004}, takes value in the category of triply graded vector spaces. It is a version of \emph{Khovanov homology}, which is due to Khovanov~\cite{khovanov2000categorification}, for links in the thickened annulus. Throughout this paper, we take coefficients in $\F:=\Z/2\Z$.

Throughout this paper we fix $\Sigma$  to be a smooth, connected, oriented surface with non-empty boundary. The mapping class group of $\Sigma$ is its group of orientation-preserving self-diffeomorphisms modulo isotopies that fix every boundary component pointwise. Given a knot or link $L$ in a closed, oriented three-manifold $Y$, there are two ways we can associate an element of a mapping class group to $L$:
\begin{enumerate}
\item If $L$ is fibered in $Y$, meaning that $L$ arises as the binding of an open book decomposition of $Y$, then the monodromy of the open book decomposition $\phi$ is an element of the mapping class group of the page $\Sigma$ of the open book. 
\item If $Y = S^{3}$, then the link $L$ can be represented as the closure of some braid $\beta$ (in fact, there are infinitely many such possible braids). A braid on $n$ strands is an element of the mapping class group of an $n$-punctured disk.
\end{enumerate}

Given an element $\phi$ in the mapping class group of some surface $\Sigma$ with boundary as described above, the \textit{fractional Dehn twist coefficient of $\phi$}, $\FDTC(\phi)$, is a rational number assigned to each boundary component of $\Sigma$ that, informally, measures the amount of twisting $\phi$ effects about that boundary component (see Section~\ref{sec:topprelim} for a formal definition). To denote the FDTC of $\phi$ at a specific boundary component $\partial_{i} \Sigma$, we write $\FDTC(\phi,\partial_i\Sigma$.) 

Our main result for link Floer homology gives a lower bound on the next-to-top term of link Floer homology arising from the fractional Dehn twist coefficient.

\begin{restatable}{theorem}{corHFKrank}\label{corHFKrank}
   Let $\Sigma$ be a  surface with $m$ boundary components that is neither an annulus nor a disk and $\phi:\Sigma\to\Sigma$ be a diffeomorphism which fixes $\partial\Sigma$ pointwise.  Let $Y_{\phi}$ be the $3$-manifold given by the open book decomposition $(\Sigma, \phi)$ with $\partial \Sigma \subset Y_{\phi}$ the fibered link given by the binding. Then:  \begin{align*}
    \rank\big(\widehat{\HFL}(Y_\phi,\partial\Sigma, {[\Sigma]},1-G)\big)\geq \Big\lceil\frac{m}{2}\Big\rceil+\underset{1\leq i\leq m}{\sum}2\max\Big\{\big(\big\lfloor|\FDTC(\phi,\partial_i\Sigma)|\big\rfloor-1\big),0\Big\}.\end{align*}
\end{restatable}

Here $G:= g(\Sigma)+m-1$, the maximal Alexander grading $A$ in which $\widehat{\HFL}(Y_\phi,\partial\Sigma, {[\Sigma]},A)\big)\neq 0$.  $ {[\Sigma]}$ is the Alexander grading on $\widehat{\HFL}(\partial\Sigma)$ induced by $\Sigma$; see Section~\ref{sec:Heegaard} for details. Note that the lower bound in Theorem~\ref{corHFKrank} is in terms of the floor of the fractional Dehn twist coefficients. This could be written equivalently as a lower bounds in terms of the ``floor" of $\phi$, a coarser measure of the twisting of $\phi$ about components of $\partial\Sigma$. See Section~\ref{sec:HFLintro} for details. In particular, the lower bound in Theorem~\ref{corHFKrank} does not account for more granular information about $\phi$ recorded by the non-integer components of the fractional Dehn twist coefficients.  The most substantial new content of the proof of Theorem~\ref{corHFKrank} is two propositions (Proposition~\ref{thm:HFKaddtwists} and Proposition~\ref{thm:HFKaddtwists1}) that together (almost) characterize the behavior of link Floer homology under adding full boundary twists to $\phi$ in the multi-boundary component case. Theorem~\ref{corHFKrank} fails in the case that $\Sigma$ is an annulus; see Proposition~\ref{prop:annuluscase} and Equation~\ref{eq:annulus2} for a complete computation of the link Floer homology of $\partial \Sigma$ in this case.

As noted in Baldwin's blog~\cite{Baldwinsite}, the fact that there is a lower bound on link Floer homology arising from the fractional Dehn twist coefficient in the single boundary case follows from Hedden and Levine's dual surgery formula~\cite{hedddenlevinedual} and~\cite[Theorem 1.3]{baldwin2022floerveering}. In this paper, we give an equivalent proof in this case using Hanselman-Rasmussen-Watson's reinterpration of the bordered Floer homology of manifolds with torus boundary in terms of immersed curves~\cite{hanselman2016bordered,hanselman2018heegaard} and work of Baldwin, Sivek, and Ni~\cite{baldwin2022floerveering}.

For knots, link Floer homology and knot Floer homology are equivalent. For the link case, as we are using it, the link Floer homology of a link $L$ in a $3$-manifold $Y$ is equivalent to the knot Floer homology of $L$ in $Y$ (i.e. the knot Floer homology of the \emph{knotification} of $L$ in the $3$-manifold $Y\#(S^1\times S^2)^{m-1}$, see~\cite{Holomorphicdisksandknotinvariants} for details), so in the statements of our results, we could have replaced $\widehat{\HFL}$ with $\widehat{\HFK}$. However, since we will be working with the Heegaard diagrams used in the setup of link Floer homology, as opposed to knot Floer homology, we use link Floer homology notation.

The majority of this paper is devoted to proving the result in the multiple boundary component case, where the theory of immersed curves has not yet been developed. Our strategy is to make an argument using an explicit computation of a map in a surgery exact triangle.

As a consequence, using the spectral sequence from annular Khovanov homology of an annular link $L$ to the knot Floer homology of the lift of the braid axis to the double branched cover of $L$ due to Roberts~\cite{roberts_knot_2013} and Grigsby and Wehrli~\cite{grigsby_khovanov_2010}, we show:

\begin{restatable}{theorem}{akh}\label{con:akh}
    Suppose that $\beta$ is an $n$-braid with $n>1$.  If $n$ is odd then \begin{align}\rank(\AKh(\widehat{\beta},n-2))\geq\max\Big\{2\Big\lfloor\Big|\frac{\FDTC(\beta)}{2}\Big|\Big\rfloor,2\Big\}.\end{align}

  \noindent  If $n$ is even then \begin{align}\label{ineq:akh2}
        \rank(\AKh(\widehat{\beta},n-2))\geq \max\{4\lfloor|\FDTC(\beta)|\rfloor-2,2\}.\end{align}
\end{restatable}

Here, $\AKh(L,k)$ indicates the annular grading $k$ summand of the annular Khovanov homology of $L$, with $\Z/2$ coefficients. The first author proved a weaker version of this rank bound, namely that $\beta$ is an $n$-braid then ${\rank(\AKh(\widehat{\beta},n-2);\C)\geq 2}$~\cite[Theorem 3.1]{binns2024closures}. The still weaker result that $\rank(\AKh(\widehat{\beta},n-2;\C))\geq 1$ follows directly from the fact that annular Khovanov homology with complex coefficients admits the structure of an $\mathfrak{sl}_2$-representation~\cite{grigsby_annular_2018}.

In Section \ref{sec:examples} we give some examples showing that the rank bounds in Theorems~\ref{corHFKrank} and~\ref{con:akh} can be arbitrarily bad. We do not have any examples where either are tight. Note that work of Gabai, Kazez, and Roberts \cite{gabai1997problems,KazezRobertsFDTC} shows that the monodromy of every fibered knot in $S^{3}$ has FDTC between $-\frac{1}{2}$ and $\frac{1}{2}$, so fibered knots in $S^{3}$ will not give interesting examples. 

Theorems~\ref{corHFKrank} and~\ref{con:akh} fit into a large body of work about which topological information knot and link Floer homology and annular Khovanov homology encodes. For instance, knot Floer homology detects the genus of a knot~\cite{ozsvath2004holomorphicgenusbounds}, whether a knot in $S^{3}$ is fibered~\cite{ghiggini2008knot, ni2007knot}, whether the monodromy of a fibered knot is right-veering~\cite{baldwin2022floerveering}, and the number of fixed points of the monodromy of a fibered knot~\cite{ni2023note}. Annular Khovanov homology distinguishes braids from other tangles~\cite{grigsby_sutured_2014}, distinguishes the trivial braid closure from other braid closures~\cite{BG}, and can be used to obstruct quasipositivity and detect right-veeringness of braid closures~\cite{grigsby2018annular}.  

In~\cite[Theorem 1 and Equation 1]{HeddenMark}, Hedden and Mark show that the rank of the Heegaard Floer homology of a 3-manifold bounds the absolute value of the FDTC of the monodromy of any of its open book decompositions with connected binding. In this paper we give a superficially different proof of a similar (but not identical) bound, Proposition \ref{prop:versionofheddenmark}.  As a corollary to their theorem, Hedden and Mark show that there is a bound relating the FDTC of an odd-stranded braid to the reduced Khovanov homology of its closure, using a spectral sequence. Our proof strategy for Theorem \ref{con:akh} is analogous in its basic structure to theirs, but we are able to draw conclusions for both odd- and even-stranded braids due to having a link Floer homology result for both fibered knots and links.

\subsection*{Outline} In Section~\ref{sec:topprelim} we review relevant background. In Section~\ref{sec:prelim} we review some aspects of Heegaard Floer homology that will be relevant in Sections~\ref{sec:singlecomponent} and~\ref{sec:HFL}, and define some specific Heegaard diagrams we will work with in those sections. In Section~\ref{sec:singlecomponent} we prove the single boundary component case of Theorem~\ref{corHFKrank} using immersed curve techniques. In Section~\ref{sec:HFL} we extend this to the multi-boundary component case, proving Theorem~\ref{corHFKrank} in full generality. In Section~\ref{sec:akh} we prove Theorem~\ref{con:akh}. We end with some examples and questions in Section~\ref{sec:examples}.

\subsection*{Acknowledgments} The first author would like to thank Shunyu Wan for many helpful conversations concerning a related project. He is also grateful for the program~\cite{hunt2015computing}, which he found helpful throughout the course of this project. Finally, he would like to thank Robert Lipshitz, Gheehyun Nahm, Peter Ozsv\'ath, and Zolt\'an Szab\'o for a number of helpful discussions. Both authors would like to thank John Hubbard for useful conversations about annuli. Some figures were made with the help of ChatGPT. 

The first author was supported by the Simons Grant {\em New structures in low-dimensional topology}. The second author was supported by NSF LEAPS-MPS Grant 2213451 and PSC-CUNY grants 66536-00 54, 67582-00 55, and 68461-00 56.

\section{Topological Background}\label{sec:topprelim}
In this section, we review aspects of mapping class group theory that will be of use in subsequent sections. We also use this opportunity to fix notation and terminology.

\subsection{Fibered Links}

Fix $Y$, a closed oriented $3$-manifold. An \emph{open book decomposition} of $Y$ is a decomposition of $Y$ into the mapping torus of some diffeomorphism $\phi$ on an oriented compact surface with boundary $\Sigma$ along with a union of solid tori (see \cite{etnyre2004lectures}, Definition 2.3, for details). We require that $\phi$ restrict to the identity on $\partial \Sigma$. It is well known that every $3$-manifold admits an open book decomposition. The map $\phi$ is often referred to as the \emph{monodromy} of the open book, $\partial \Sigma$ as the \emph{binding} of the open book, and any copy of $\Sigma$ under the flow of $\phi$ as a \emph{page} of the open book. A knot or link in $Y$ is said to be \emph{fibered} in $Y$ if it can be realized as the binding of an open book decomposition for $Y$.  Examples of fibered knots in $S^{3}$ include the trefoils and the figure-eight knot.

Let $L_i$ be a component of a fibered link, $L$, in a $3$-manifold, $Y$, corresponding to an open book decomposition, $(\Sigma,\phi)$. The core of $-1/n$-surgery on $L_i$, together with the components of $L\setminus L_i$ is  a fibered link in $Y_{1/n}(L_i)$. Here the framing is measured with respect to the framing induced on $L_i$ by $\Sigma$. The link in the surgered manifold corresponds to an open book decomposition $(\Sigma,\Delta_i^n\circ\phi)$ thereof.

\subsection{Right-Veering and Left-Veering Diffeomorphisms}

Let $\Sigma$ be a connected oriented surface with compact boundary. Given two properly embedded arcs $a$ and $b$ on $\Sigma$ both originating from a basepoint $x_{0} \in \partial \Sigma$, we say that $a$ is \emph{weakly to the right} of $b$ if either $a$ is isotopic to $b$ rel endpoints or, after isotoping the arcs rel boundary so that they have minimal intersection, $a$ is to the right of $b$ in a neighborhood of $x_{0}$. Equivalently, $b$ in this scenario is to the left of $a$. Notice that, according to this definition, an arc is both weakly to the right and weakly to the left of itself. If we do not want to allow the case where $a$ and $b$ are isotopic rel endpoints, we say that one is \emph{strictly to the right/left} of the other. These definitions naturally descend to homotopy classes of arcs in $\Sigma$ relative to their boundary. The reader dissatisfied with this definition may instead define notions of right and left by considering geodesic representatives of lifts of $a$ and $b$ to the universal cover of $\Sigma$ endowed with its canonical metric; see, for instance, \cite[Section 3]{feller2023dehn}. 

Now, let $\phi$ be a diffeomorphism on $\Sigma$ that fixes $\partial \Sigma$ pointwise. We say that $\phi$ is \emph{(weakly) right-veering} (resp. \emph{ (weakly) left-veering}) if it sends every properly embedded arc in $\Sigma$ (weakly) to the right (resp. left). The notion of weak right-veeringness is important as it gives a characterization of tight contact structures~\cite{honda2007right}.

\subsection{Fractional Dehn Twist Coefficients}\label{sec:FDTC}

Let $\Sigma$ be a connected oriented surface with compact boundary and let $\phi$ be a diffeomorphism on $\Sigma$ that fixes $\partial \Sigma$ pointwise. Fix a connected boundary component $L_i$ of $\Sigma$. The fractional Dehn twist coefficient $\FDTC(\phi,L_i)$ is a real number that, informally, measures how much $\phi$ ``twists" about $L_i$. This number is generally different for each boundary component of $\Sigma$. In the case where we are studying fibered knots or links, the phrase ``fractional Dehn twist coefficient of the fibered knot or link" refers to the fractional Dehn twist coefficient of the monodromy for the corresponding open book. 

The fractional Dehn twist coefficient (or rather, its reciprocal, the \textit{degeneracy slope}) was originally defined by Gabai and Oertel in \cite{gabai1989essential} for fibered knots in their study of essential laminations of $3$-manifolds. The fractional Dehn twist coefficient was later used to study contact structures on $3$-manifolds in~\cite{honda2007right}. It has since been studied by many others and several different definitions for it appear in the literature; see for instance~\cite{KazezRobertsFDTC,ito2018fractional,malyutin2005twist}. 

We give a formal definition of the fractional Dehn twist coefficient due to Feller, the second author, and Turner in \cite{feller2023dehn}. Fix $x_{0}$ a basepoint on $L_i$. Let $\mathcal{J}$ denote the set of homotopy classes of continuous maps $\gamma : [0,1] \to \Sigma$ that start at $x_{0}$, end anywhere on $\partial \Sigma$, and are not boundary parallel. We restrict the homotopies to fix the endpoints of $\gamma$. 

Fix any element $I$ in $\mathcal{J}$. Let $P_{I}$ denote the set of all diffeomorphisms $\phi$ on $\Sigma$ fixing $\partial \Sigma$ pointwise that send $I$ weakly to the right. Now given two diffeomorphisms $\phi$ and $\psi$ on $\Sigma$ fixing $\partial \Sigma$ pointwise, we write that $\phi \leq \psi$ if $\phi^{-1} \psi \in P_{I}$. Let $\Delta_i$ denote the positive Dehn twist about a curve parallel to $L_i$. Let $\big\lfloor \phi \rfloor = \max \{ k : \Delta_{i}^{k} \leq \phi\}$. Then
$$\FDTC(\phi,L_i) = \displaystyle{\lim_{n \to \infty}} \frac{ \big\lfloor \phi^n \rfloor}{n}. $$

In \cite{feller2023dehn} it is shown that this quantity is a real number that is independent of the conjugacy class of the isotopy class of $\phi$ (rel isotopies that fix the boundary), that it is equivalent to other definitions of the fractional Dehn twist coefficient appearing in the literature, and that it is independent of the choice of $I$. 

It is well-known that if $\FDTC(\phi,L_i) > 0$ then $\phi$ is strictly right-veering at $L_i$ (see, for instance, \cite[Corollary 16]{feller2023dehn}). Kazez and Roberts showed in \cite{KazezRobertsFDTC} (and it is clear from the above definition) that if $\phi$ sends any properly embedded arc starting from a basepoint $x_{0}$ on $L_i$ to the right, then $\FDTC(\phi) > 0$, and the equivalent statement for the left. They also give bounds on the fractional Dehn twist coefficient in terms of intersection numbers of arcs. Of use to us will be their \cite[Proposition 2.9]{KazezRobertsFDTC}, which implies that if $\FDTC(\phi,L_i)>2$,  and $a$ is a properly embedded arc originating at $L_i$, the interiors of $a$ and $\phi(a)$ must intersect at least twice in an annular neighborhood of $L_i$.

In Section~\ref{sec:akh}, we will need to move between knot Floer homology and annular Khovanov homology using a double branched cover construction. Given a closed $n$-braid $\beta$ with braid axis the unknot $U$ in $S^{3}$, the double branched cover of $S^{3}$ branched along $\beta$ is a $3$-manifold with the open book decomposition whose monodromy $\phi$ is the lift of $\beta$, whose pages $\Sigma$ are the double branched covers of the $n$-punctured $D_{n}$, and whose binding is the lift of $U$. In particular, if $n$ is odd, $\Sigma$ has one boundary component, and if $n$ is even, $\Sigma$ has two. In \cite{ito2018fractional}, Ito and Kawamuro analyze the behavior of the fractional Dehn twist coefficient under branched coverings. They show, in particular, that when $n$ is odd, $\frac{1}{2} \FDTC(\beta) = \FDTC(\phi)$, and when $n>2$ is even, $\FDTC(\beta)=\FDTC(\phi,c_{i})$ for both boundary components $c_{1}$ and $c_{2}$ of $\Sigma$.

The $n=2$ case must be analyzed separately. The double branched cover of a $2$-braid is a $3$-manifold with equipped with an open book decomposition with annular pages. Up to isotopy rel boundary, the only diffeomorphisms of the annulus are Dehn twists about a core curve (that is, a boundary parallel curve). For any such diffeomorphism, $\phi$, $\FDTC(\phi,L_1) = \FDTC(\phi,L_2)$, where $L_1$ and $L_2$ are the boundary components of the annulus.

\section{Heegaard Floer Preliminaries}\label{sec:prelim}

In this section we review various aspects of Heegaard Floer homology we will use in later sections. As in the previous section, we use this opportunity to fix notation and conventions.

\subsection{Heegaard Diagrams for Fibered Links}\label{sec:Heegaard}
In this section we discuss how to produce Heegaard diagrams for fibered links. In doing so, we set the notation that we will continue to use in Section~\ref{sec:HFL}. 

\begin{definition}
    A \emph{basis of arcs} for an oriented surface, $\Sigma$, with non-empty boundary is a collection of disjoint, properly embedded arcs in $\Sigma$ whose complement is a disk. 
\end{definition}

 Pick an ordering of the connected components of the boundary of $\Sigma$ and denote them by $\partial_1\Sigma,\partial_2\Sigma,\dots\partial_n\Sigma$. We will typically assume that $\partial_1\Sigma$ intersects exactly one arc, $a_1$, as this will make various aspects of the proofs in Section~\ref{sec:HFL} more straightforward. The number of arcs in a basis of arcs for $\Sigma$ is $N:=2g(\Sigma)+|\partial\Sigma|-1$. We will take our arcs to be oriented, and let $\partial_\pm a_i$ indicate the positive and negative components of $\partial a_i$.
 
We first recall a way of obtaining a Heegaard diagram adapted to an open book, see~\cite[Section 5.2]{OzbagciStipsicz} for details.

\begin{definition}\label{def:HDfor3mnfld}
    Let $(\Sigma,\phi)$ be an open book decomposition of a three-manifold $Y$. A \emph{Heegaard diagram} (for $Y$) \emph{adapted to $(\Sigma,\phi)$} is the Heegaard diagram $(S,\bm{\alpha},\bm{\beta},\bm{z},\bm{w})$ formed as follows:

    \begin{enumerate}
        \item Pick a basis of arcs $\{a_i\}_{1\leq i\leq N}$ for $\Sigma$.
        \item Take another copy of $\Sigma$ with the opposite orientation, $-{\Sigma}$. Consider the arcs $\{\overline{a}_i\}_{1\leq i\leq N}\subset -{\Sigma}$. Consider the arcs $\overline{b}_i:=\phi(\overline{a}_i)$ for $1\leq i\leq N$.
        \item $\mathcal{H}$ has Heegaard surface $S:=\Sigma\cup_{\partial \Sigma}-{\Sigma}$. The $\alpha$-curves are given by $a_i\cup \overline{a}_i$ and the $\beta$-curves are given by a perturbation of $a_i\cup \overline{b}_i$ as shown in Figure~\ref{fig:adaptedheegaarddiagram}.
    \end{enumerate}
\end{definition}

    \begin{figure}[ht]
\centering

   \begin{tikzpicture}


\draw[thick] (0,0) ellipse (2 and 3.5);

\draw[thick] (0,0) ellipse (0.25 and 0.5);

\draw[ thick] (0,2) ellipse (0.5 and 0.25);

\draw[thick] (0,-2) ellipse (0.5 and 0.25);

\draw[red, thick] (0,0) ellipse (1.5 and 3);

\draw[blue, thick]  (-1.6,0)  arc(180:360:1.5 and 3);

\draw[red, thick] (0,0) ellipse (0.5 and 0.8);

\draw[blue, thick]  (-0.6,0)  arc(180:360:0.5 and 0.8);

\draw[red,thick] (0,-1.75) 
.. controls (1.5,-1.75) and (1.5,1.75)..
(0.0,1.75);
\draw[red ,thick, dashed] (0,-1.75) 
.. controls (1,-1.75) and (1,1.75)..
(0.0,1.75);

\draw[blue ,thick, dashed] (0.4,-1.9) 
.. controls (1.6,-1.9) and (1.2,0)..
(1.2,0);

\draw[blue ,thick] (0.4,-1.9) 
.. controls (1,-1.75) and (1,0)..
(1,0);

\end{tikzpicture}
    \caption{A Heegaard diagram adapted to an open book $(\Sigma,\phi)$. It is standard to draw the figure so that $-\Sigma$ is the upper half of the diagram and $\Sigma$ is the lower half.}\label{fig:adaptedheegaarddiagram}
   \end{figure}
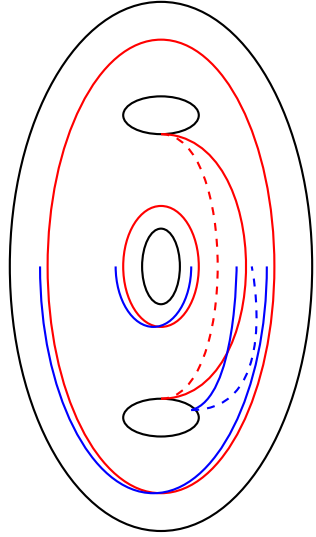

We now describe how to obtain a Heegaard diagram adapted to the fibered link $\partial\Sigma$.

\begin{definition}\label{def:HDadaptredtoL}
    Let $L$ be an $n$ component oriented link that is the binding of an open book $(\Sigma,\phi)$. A \emph{Heegaard diagram adapted to $L$} is a pointed Heegaard diagram which can be obtained from a Heegaard diagram adapted to $(\Sigma,\phi)$ by the following process:
    
    \begin{enumerate}
        \item Pick a point $p_i$ on $\partial_i\Sigma\setminus(\bm{\alpha}\cup\bm{\beta})$ for each $i$. Use a finger move to push each intersection between a $\beta$-curve and $\partial_i\Sigma$ over $p_i$, in the direction dictated by the orientation of $\Sigma$, without introducing any new intersections of the $\beta$-curves with each other.
        \item Place basepoints on the diagram as follows: label each $p_{i}$ with a $w_i$ basepoint. Leaving $w_{i}$ and traveling along $\partial_i\Sigma$ in the direction dictated by the orientation of $\Sigma$, stop and place a $z_i$ basepoint when all the $\beta$-curves that intersect $\partial_i \Sigma$ have been passed but none of the $\alpha$-curves have been passed. 
        \item Add the following $\alpha$ and $\beta$-curves to the diagram. For all  $\partial_i\Sigma$ except for $\partial_2\Sigma$, add an $\alpha$-curve, $\widetilde{\alpha}_i$, that is the boundary of a neighborhood of the arc in $\partial_i\Sigma$ from $w_i$ to $z_i$ that is oriented as the boundary of $\Sigma$, and add a $\beta$-curve denoted $\widetilde{\beta}_i$ that is a neighborhood of the arc in $\partial_i\Sigma$ from $z_i$ to $w_i$ that is oriented as the boundary of $\Sigma$. We will sometimes refer to these curves as \emph{necklace curves}.
    \end{enumerate}

\end{definition}

See Figure~\ref{fig:adaptedheegaarddiagramlink} for an example of a Heegaard diagram adapted to $L$. These Heegaard diagrams are used in~\cite{tovstopyat2018transverse}. The choice of $\partial_2\Sigma$ --- as opposed to any of the other boundary components --- is unnecessary, but will prove convenient at various stages in Section~\ref{sec:HFL}.

   \begin{figure}[ht]
\centering

   \begin{tikzpicture}


\draw[ thick]
(-1,-3.5)
.. controls (3,-3.5) and (3,3.5) ..
(-1,3.5)
.. controls (-7.5,3.5) and (-7.5,-3.5) ..
(-1,-3.5);

\draw[red, thick] (0,0) ellipse (1 and 3.0);

\draw[red, thick] (-3,0) ellipse (1.8 and 1.0);

\draw[ draw=black, thick] (0,2) ellipse (0.5 and 0.3);
\draw[draw=black, thick] (0,-2) ellipse (0.5 and 0.3);
\draw[draw=black, thick] (-3.25,0) ellipse (0.3 and 0.5);

\node at (-4,0) {$w_1$};
\node at (-5.1,0) {$z_1$};
\node at (-1.5,0) {$w_2$};
\node at (-2.5,0) {$z_2$};

\draw[blue, thick] (-4.6,0) ellipse (0.8 and 0.5);

\draw[blue, thick]
(-2,0)
.. controls (-2,-1) and (-1.0,-1.0) ..
(0,-1.7);

\draw[blue, dashed, thick]
(0,-1.7) .. controls (1,-1.5) and (1.0,-1.4) ..
(1.7,-1.5);

\draw[blue, thick]
(1.7,-1.5) .. controls (0,-1) and (-1.8,-0.5) ..
(-1.8,0);

\draw[blue, thick]
(-5.7,0) .. controls (-5.7,-2) and (-2.3,-2) ..
(-2.3,0);


\draw[red, thick]
(-3.5,0.2) .. controls (-4.5,1) and (-4.5,-1) ..
(-3.5,-0.2);

\draw[red, thick,dashed]
(-3.5,0.2) .. controls (-4,1) and (-5,1) ..
(-5.9,0.2);

\draw[red, thick,dashed]
(-3.5,-0.2) .. controls (-4,-1) and (-5,-1) ..
(-5.9,-0.2);

\draw[red, thick]
(-5.9,0.2) .. controls (-4.6,1) and (-4.6,-1) ..
(-5.9,-0.2);


\draw[blue, thick]
(-2.2,0)
.. controls (-2.2,-1.5) and (-0.5,-2.5) ..
(0,-2.5)
.. controls (0.1,-2.5) and (0.5,-2.5) ..
(0.6,-2)
.. controls (0.6,-0.7) and (-1.9,-1) .. (-1.9,0);

\draw[red, thick, dashed]
(0,1.7)
.. controls (0.25,0.5) and (0.25,-0.5)..
(0.0,-1.7);
\draw[red, thick]
(0,1.7)
.. controls (-0.25,0.5) and (-0.25,-0.5)..
(0.0,-1.7);

\end{tikzpicture}
    \caption{A Heegaard diagram adapted to the binding of an open book.}\label{fig:adaptedheegaarddiagramlink}
   \end{figure}
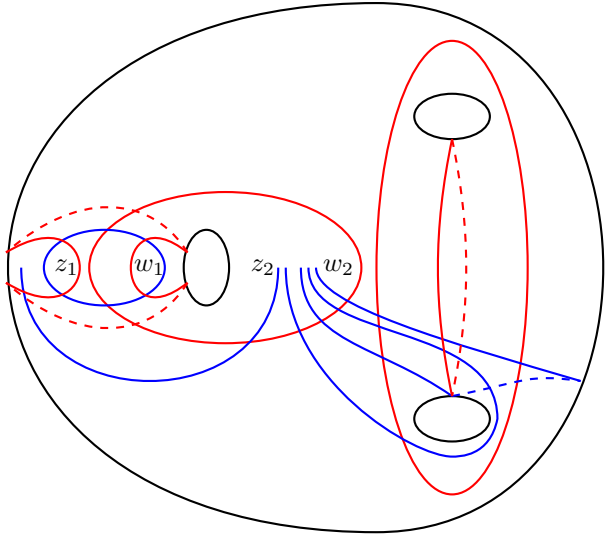

 We will write $\mathcal{H}(L)$ for the Heegaard diagram adapted to a fibered link $L$ arising from an implicit basis of arcs $\{a_i\}$ for $\Sigma$. Typically we will make the following additional assumptions about the basis of arcs:\begin{enumerate}
     \item That $a_1$ has one endpoint on $\partial_1\Sigma$ and another on $\partial_2\Sigma$,
     \item That $a_i\cap\partial_1\Sigma=\emptyset$ for $i\neq 1$. 
 \end{enumerate}
 Note that there always exists a basis of arcs that satisfies these conditions. We will work primarily with the case that $a_1$ is sent (strongly) to the left by $\phi$ at $a_1\cap\partial_1\Sigma$.  Unless $a_1$ is fixed by $\phi$, this can be arranged up to mirroring. Note that since the upper half of the Heegaard diagram $\mathcal{H}(L)$ --- i.e. the portion corresponding to $-{\Sigma}$ --- comes equipped with an orientation which disagrees with that of $\-{\Sigma}$, arcs which appear to be sent to the left are actually sent to the right.

 Finally, note that if $(S,\bm{\alpha},\bm{\beta},\bm{z},\bm{w})$ is a Heegaard diagram adapted to the binding of an open book $(\Sigma,\phi)$, then $(S,\bm{\beta},\bm{\alpha},\bm{z},\bm{w})$ is a Heegaard diagram for $-Y$, and that $-Y$ has open book decomposition $(\Sigma,\phi^{-1})$.

\subsection{Heegaard Floer homology}\label{sec:HFLintro}

We now review aspects of Heegaard Floer homology which will be relevant for us in subsequent sections.
Let $L$ be a link in a $3$-manifold $Y$ and $\Sigma$ be a surface with boundary $L$. Let $[\Sigma]$ be the homology class of $\Sigma$ in $H_2(Y,L;\Z)$. Link Floer homology is an invariant of tuples $(Y,L)$ due to Ozsv\'ath and Szab\'o~\cite{HolomorphicdiskslinkinvariantsandthemultivariableAlexanderpolynomial}. We briefly review the construction. Let $(S,\bm{\alpha},\bm{\beta},\bm{z},\bm{w})$ be a Heegaard diagram adapted to $L$. We require $(S,\bm{\alpha},\bm{\beta},\bm{z},\bm{w})$ to be \emph{admissible}; that is, for each \emph{periodic domain} --- a linear combination of components of $S\setminus(\bm{\alpha}\cup\bm{\beta})$ not containing a basepoint with boundary a linear combination of $\alpha$ and $\beta$-curves --- to have a negative coefficient. $\widehat{\CFL}(S,\bm{\alpha},\bm{\beta},\bm{z},\bm{w})$ is freely generated over the field of two elements, $\F$, by certain sets of intersection points between $\bm{\alpha}$ and $\bm{\beta}$ curves, $\bm{x}$. Each $\bm{x}$ is required to contain exactly one point from each $\alpha$-curve and exactly one point from each $\beta$-curve. The vector space $\widehat{\CFL}(S,\bm{\alpha},\bm{\beta},\bm{z},\bm{w})$ is endowed with a differential which counts Maslov index $1$ pseudo-holomorphic curves in an auxiliary symplectic manifold, $\Sym^g(S)$, where $g$ is $g(S)+|L|-1$. It turns out that the chain homotopy type of $\widehat{\CFL}(S,\bm{\alpha},\bm{\beta},\bm{z},\bm{w})$ is independent of the choice of pointed Heegaard diagram, so we denote it by $\widehat{\CFL}(Y,L)$ --- though we will still use $\widehat{\CFL}(S,\bm{\alpha},\bm{\beta},\bm{z},\bm{w})$ when we wish to emphasize that we have a specific Heegaard diagram, or specific generators from a specific Heegaard diagram, in mind.

There is a natural map $p:\Sym^g(\Sigma)\to\Sigma$. Given a map $\phi:\bm{D}\to\Sym^g(\Sigma)$, we let $D(\phi)$ denote the image of $\phi$ under $p$. Let $\{D_i\}$ be the components of the complement of the $\alpha\cup \beta$ in $\Sigma$. We let $n_{D_i}$ denote the multiplicity of the map $p\circ\phi$. These numbers can be used to compute the Maslov index of $\phi$, using a combinatorial formula due to Lipshitz~\cite{Lipshitzcylindrical}. 

For any choice of homology class of Seifert surface for $L$, $\Sigma$, $\widehat{\HFL}(Y,L)$ splits as a direct sum over $\Z$. We will denote the $i$th summand in this decomposition by $\widehat{\HFL}(Y,L,[\Sigma],i)$. This grading is called the \emph{Alexander grading} (with respect to $\Sigma$).  We shall denote it by $A_{[\Sigma]}$. One fact that we will use in later arguments is that for a generator $\bm{x}$ in a Heegaard diagram adapted to a fibered link $L$, $\mathcal{H}(L)$ as above, the Alexander grading $A_{-[\Sigma]}(\bm{x})$ can be determined by the number of intersection points of $\bm{x}$ in $\-{\Sigma}$ minus $G:=\dfrac{|L|-\chi(L)}{2}$~\cite[Lemma 4.2]{tovstopyat2018transverse}. The maximal $ {A_{[\Sigma]}}$ grading in which $\widehat{\HFL}(L)$ is non-trivial is given by $G$~\cite[Theorem 8.4, Remark 8.5]{juhasz2008floer} --- see also ~\cite[Theorem 1.1]{ni2006note} for a less general result. Here $\chi(L)$ is the maximal Euler characteristic representative of $[\Sigma]\in H_2(Y,L)$, while $|L|$ is the number of components of $L$. The maximum non-trivial Alexander grading is of rank one if and only if $L$ is fibered~\cite[Theorem 9.11]{juhasz2008floer}, see also~\cite{ni2007knot,ghiggini2008knot}.

Link Floer homology is invariant under a number of operations relevant to this paper, at least if we ignore the Maslov grading, which we will do for the duration of this paper:
\begin{equation}\label{eq:invariance}
\widehat{\HFL}(Y,L,[\Sigma],i)\cong \widehat{\HFL}(-{Y},-{L},[\Sigma],i)\cong\widehat{\HFL}({Y},{L},[-\Sigma],i)\cong \widehat{\HFL}({Y},{L},[\Sigma],-i).\end{equation}

See~\cite[Section 8]{HolomorphicdiskslinkinvariantsandthemultivariableAlexanderpolynomial}. Here $-Y$ denotes $Y$ with its orientation reversed and $-L$ denotes the image of $L$ in $-Y$. Suppressing the $3$-manifold temporarily, we also have that \begin{equation}\label{eq:inverse}
\widehat{\HFL}(L_\phi,[\Sigma],i)\cong \widehat{\HFL}(L_{\phi^{-1}},[\Sigma],-i),\end{equation} which follows from the behaviour of link Floer homology under interchanging the roles of the $\alpha$ and $\beta$-curves.

\subsection{The BRAID invariant}\label{sec:braidinvt}
A link, $L$, in a $3$-manifold, $Y$, equipped with an open book structure is \emph{braided} if it meets every page transversely with positive intersection. Two braid closures are \emph{braid isotopic} if they are isotopic via a family of braid closures. There is an invariant of braid closures (up to braid isotopy) called the \emph{BRAID invariant}~\cite{baldwin2013equivalence}, due to Baldwin, Vela-Vick, and V\'ertesi. This invariant takes value in the link Floer homology of the mirror of the underlying link in the underlying $3$-manifold with its orientation reversed and is defined as a specific generator of $\widehat{\CFL}(\mathcal{H}(-{L}))$, where here $\mathcal{H}(-{L})$ is an appropriately chosen Heegaard diagram for $-{L}\subset -Y$ --- where $-L$ is the orientation reversal of $L$, and $-Y$ is the orientation reversal of $Y$ --- similar to those discussed in Section~\ref{sec:Heegaard}.

We will specifically be interested in the BRAID invariant of a specific approximation of fibered links as a braid in the corresponding open book $(\Sigma,\phi)$. In this case the relevant Heegaard diagram is exactly that presented in Definition~\ref{def:HDadaptredtoL}, but where we reverse the role of the $\alpha$ and $\beta$ curves, as well as the $z$ and $w$-basepoints; that is, we consider $(S,\bm{\beta},\bm{\alpha},\bm{w},\bm{z})$ as opposed to $(S,\bm{\alpha},\bm{\beta},\bm{z},\bm{w})$. The generator of $\mathcal{H}(-L)$ can be described as follows. Note that the Heegaard diagram adapted to $(\Sigma,\phi)$ --- as in Definition~\ref{def:HDfor3mnfld} --- contains a canonical set of intersection points ${c_i\in a_i\cap b_i\subset \Sigma}$ for $1\leq i\leq N:=2g(\Sigma)+|\partial\Sigma|-1$. These intersection points are fixed under the isotopy used to produce the Heegaard diagram adapted to $\partial\Sigma$ in Definition~\ref{def:HDadaptredtoL}. Thus, by a mild abuse of notation, we may consider these as elements in $\alpha_i\cap\beta_i$ in the Heegaard diagram $\mathcal{H}(-L)$. To this set, we add additional generators $\widetilde{c_j}\in \widetilde{\alpha_j}\cap\widetilde{\beta}_j\subset \Sigma$ for $j=1$,$3\leq j\leq  |\partial\Sigma|$. These intersection points are exactly those directly under the corresponding $z_i$ basepoint; see Figure~\ref{fig:adaptedheegaarddiagramlink}. That is, by definition, the BRAID invariant is the class $\bm{c}(L):=\{c_1,c_2,\dots,c_N,\widetilde{c}_1,\widetilde{c}_3,\widetilde{c}_4,\dots\widetilde{c}_n\}$, where $n:=|\partial\Sigma|$.  Note that $\widetilde{c}_2$ does not appear as there are no $\widetilde{\alpha}_2$ or $\widetilde{\beta}_2$ curves.

 From the previous subsection, we know that if $L$ is fibered then $\widehat{\HFL}\big(-{L},-{Y},{-[\Sigma]},-G\big)\cong\F$.  Tovstopyat-Nelip showed, moreover, that if $L$ is fibered then $\widehat{\HFL}\big(-L,-{Y},{-[\Sigma]},-G\big)$ is generated by the BRAID invariant of the braid realised by an appropriate push off of $L$~\cite{tovstopyat2018transverse}. This generalizes earlier work of Vela-Vick~\cite{VelaVicktransverseinvariant}, who showed a similar result in the single boundary component case.

   \subsection{Maps on Heegaard Floer homology}\label{subsec:maps} Some of our proofs in Section~\ref{sec:HFL} will use pseudoholomorphic triangle-counting maps. To define such maps, one starts with a \emph{Heegaard triple-diagram}; that is a Heegaard surface, $\Sigma$, together with collections of $\alpha$, $\beta$, and $\gamma$-curves. One additionally requires that the Heegaard diagrams satisfy an admissibility assumption. To define this condition, recall that a  \emph{triply-periodic domain} is a linear combination of $\alpha$, $\beta$ and $\gamma$-curves. A pointed Heegaard triple-diagram is \emph{weakly admissible} if each non-trivial triply-periodic domain --- i.e.  a union of components of $\Sigma\setminus(\bm{\alpha}\cup\bm{\beta}\cup\bm{\gamma})$ --- can be written as a sum of doubly-periodic domains that have both positive and negative coefficients. Now recall from, say,~\cite[Section 8]{ozsvath2004holomorphic}, that given a weakly admissible Heegaard triple $(\Sigma,\bm{\alpha},\bm{\beta},\bm{\gamma})$  we have a filtered chain map: $$f:\widehat{\CFL}(\mathcal{H}_{\bm{\alpha},\bm{\beta}})\otimes\widehat{\CFL}(\mathcal{H}_{\bm{\beta},\bm{\gamma}})\to\widehat{\CFL}(\mathcal{H}_{\bm{\alpha},\bm{\gamma}}),$$ which is defined by counting pseudo-holomorphic disks subject to the boundary conditions indicated in Figure~\ref{fig:triangle1}. Here $\mathcal{H}_{\bm{\alpha},\bm{\beta}}$ denotes the Heegaard diagram obtained from $\mathcal{H}_{\bm{\alpha},\bm{\beta},\bm{\gamma}}$ by forgetting the $\gamma$-curves. Note that while the triangle counting maps defined in~\cite[Section 8]{ozsvath2004holomorphic} are for maps between closed $3$-manifolds, the same definition applies in the context of link Floer homology; see~\cite[Section 6, 7]{HolomorphicdiskslinkinvariantsandthemultivariableAlexanderpolynomial}.

\begin{figure}[htbp]
    \centering
    \begin{tikzpicture}
        \coordinate (A) at (0,0);
        \coordinate (B) at (2,0);
        \coordinate (C) at (1,{sqrt(3)});
        
        \draw[red, thick] (A) -- (B) node[midway, below] {$\T_{\bm{\alpha}}$};
        \draw[green, thick] (B) -- (C) node[midway, right] {$\T_{\bm{\gamma}}$};
        \draw[blue, thick] (C) -- (A) node[midway, left] {$\T_{\bm{\beta}}$};
        
        \node[below] at (A) {$\bm{x}$};
        \node[below] at (B) {$f_{\bm{\theta}}(\bm{x})$};
        \node[above] at (C) {$\bm{\theta}$};
    \end{tikzpicture}
    \caption{For $\bm{\theta}$ a cycle in $\widehat{\CFL}(\mathcal{H}_{\bm{\beta},\bm{\gamma}})$, $f_{\bm{\theta}}:\widehat{\CFL}(\mathcal{H}_{\bm{\alpha},\bm{\beta}})\to \widehat{\CFL}(\mathcal{H}_{\bm{\alpha},\bm{\gamma}})$ counts triangles of this form.}
    \label{fig:triangle1}
\end{figure}
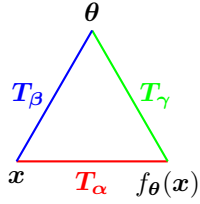

Thus, for any cycle $\bm{\theta}$ in $\widehat{\CFL}(\mathcal{H}_{\bm{\beta},\bm{\gamma}})$, we obtain a chain map $f_{\bm{\theta}}:\widehat{\CFL}(\mathcal{H}_{\bm{\alpha},\bm{\beta}})\to \widehat{\CFL}(\mathcal{H}_{\bm{\alpha},\bm{\gamma}})$ given by $\bm{x}\mapsto f( \bm{x}\otimes \bm{\theta})$. We note that because we will typically count triangles after having reversed the roles of $\beta$ and $\alpha$-curves in a given manifold, the triangles we count will look different in figures.

\section{The Single boundary component case}\label{sec:singlecomponent}

In this section, we prove Theorem~\ref{corHFKrank} in the special case that $\Sigma$ has a single boundary component:

\begin{restatable}{proposition}{singleboundary}\label{prop:HFKsingleboundarycomponent}
      Suppose $K$ is a non-trivial fibered knot with minimal genus Seifert surface $\Sigma$ in a {$3$-manifold} $Y$. Then $\rank\big(\widehat{\HFL}(K, {[\Sigma]},g(\Sigma)-1)\big)\geq 1+2\max\big\{\big\lfloor|\FDTC(K)|\rfloor-1,0\big\}$.
\end{restatable}

As noted on Baldwin's blog~\cite{Baldwinsite}, this result can be obtained using Hedden-Levine's dual knot surgery formula~\cite{hedddenlevinedual} in combination with a result of Baldwin-Ni-Sivek~\cite{baldwin2022floerveering}. We give a superficially different proof, appealing to work of Baldwin-Vela-Vick~\cite{baldwin_note_2018} and Hanselman-Rasmussen-Watson's theory of immersed curves~\cite{hanselman2016bordered,hanselman2018heegaard}.

We briefly recall the structure of the immersed curve invariant. The reader may find it helpful to refer to~\cite[Section 4.4]{hanselman2018heegaard} for further details. View $S^1$ as $\R/\Z$ where $\Z$ acts by addition. To each knot $K$ the immersed curve invariant assigns a multi-curve $\gamma$ in the cylinder $S^1\times\R\setminus(\{0\}\times(\Z+\frac{1}{2}))$. Perhaps after a small perturbation of $\gamma$, the knot Floer homology of $K_{1/n}$ --- which we set to be the knot given by the core of $1/n$ surgery on $K$ for the remainder of this paper --- can be recovered from $\gamma$ by taking the vector space freely generated by intersections of $\gamma$ and the the curve $ny=x$ in $\R\times\R$ under the natural quotient to $S^1\times\R$. Each intersection point $x$ occurs at a point with $\R$-coordinate in the range from $m-\frac{1}{2}$ to $m+\frac{1}{2}$ for some $m\in\Z$. The Alexander grading of such an intersection point $x$ is then exactly $n$~\cite[Proposition 56]{hanselman2018heegaard}. The Heegaard Floer homology of $1/n$ surgery on $K$ can be determined similarly, using a  $ny=x+\epsilon$ for an appropriate $\epsilon$ 
 --- for example $\epsilon\in\R\setminus\Q$ --- instead of the curve $ny=x$. The spectral sequence from $\widehat{\HFL}(K_{1/n})$ to $\widehat{\HF}(Y_{1/n}(K))$ can be recovered by counting bigons subject to appropriate boundary conditions --- see~\cite[Section 4.3]{hanselman2018heegaard}.

\begin{proposition}\label{prop:newknotFDTC<0}
   Let $\phi:\Sigma\to\Sigma$ be a diffeomorphism, where $\partial \Sigma$ has a single component and $\Sigma$ is not a disk, and $K_\phi$ be the binding of the corresponding open book. If $\FDTC(\phi)>0$ then:\begin{align}\label{HFK:twists-1}
      \rank\big(\widehat{\HFL}(K_{\Delta\circ\phi}, {[\Sigma]},1-g(\Sigma)) \big)=\rank(\widehat{\HFL}(K_\phi, {[\Sigma]},1-g(\Sigma)))+2.
   \end{align}
On the other hand, if $\FDTC(\phi)= 0$ then $\rank\big(\widehat{\HFL}(K_\phi, {[\Sigma]},1-g(\Sigma))\big)$ is:\begin{align}\label{HFK:twists0}
      \rank\big(\widehat{\HFL}(K_{\Delta\circ\phi}, {[\Sigma]},1-g(\Sigma))\big),\text{ or }   \rank\big(\widehat{\HFL}(K_{\Delta\circ\phi}, {[\Sigma]},1-g(\Sigma))\big)-2.
   \end{align}

\end{proposition}

For the statement of the first Lemma of this section, recall that an immersed curve  $\gamma$ can be \emph{pulled tight} (with respect to some quantity $\epsilon$, which we suppress) by isotoping $\gamma$ to a geodesic representative in $S^1\times \R$, which is usually unique. See~\cite[Section 7.1]{hanselman2016bordered} for details. 

\begin{lemma}\label{lem:immersedcurve|FDTC|<1}
    Let $\phi:\Sigma\to\Sigma$ be a self-diffeomorphism of a surface with a single boundary component of genus $g>0$ fixing $\partial \Sigma$ and let $K_\phi$ be the binding of the corresponding open book. Let $\gamma$ denote the immersed curve invariant of $K_\phi$. After pulling tight, $\gamma$ intersects $\{0\}\times(g-\frac{1}{2},\infty)$ at a unique point $c\in\{0\}\times(g-\frac{1}{2},g+\frac{1}{2})$. Moreover, the component of $\gamma$ that intersects $c$, $\gamma_c$, has trivial local system. In the region $S^1\times[g-\frac{3}{2},g+\frac{1}{2}]$, every component of $(\gamma\setminus\gamma_c)\cap S^1\times[g-\frac{3}{2},g+\frac{1}{2}]$ is contained in a neighborhood of $S^1\times\{g-\frac{3}{2}\}$. Finally:

    \begin{enumerate}
        \item\label{item:1}  If $-1<\FDTC(\phi)<0$, then $\gamma_c$ is as shown in Figure~\ref{fig:-1<FDTC<0}. More precisely, the arc emanating from $c$ to the left, $\gamma_l$, and the arc emanating from $c$ to the right, $\gamma_r$, become lines of slope in the interval $(-\infty,-1)$, after pulling tight.
    \item\label{item:2} If $\FDTC(\phi)=0$, then $\gamma_c$ is as shown in Figure~\ref{fig:-1<FDTC<0}, its mirror about the vertical line $\{0\}\times\R$, Figure~\ref{fig:FDTC=0}, or Figure~\ref{fig:FDTC=01}.
     \end{enumerate}
\end{lemma}

Observe that the $1>\FDTC(\phi)>0$ case could be recovered from the $-1<\FDTC(\phi)<0$ case by using the symmetry properties of $\gamma$. The main technical input for the proof is Baldwin-Ni-Sivek's result that knot Floer homology detects non-weakly right veering monodromies; in particular~\cite[Remark 1.4]{baldwin2022floerveering}. More precisely, the induced map on the $E_2$ page of the spectral sequence from $\widehat{\HFK}(K_\phi, {[\Sigma]},g(K))$ to $\widehat{\HFK}(K_\phi, {[\Sigma]},g(K)-1)$ is non-trivial if and only if $\phi$ is non-weakly right veering.

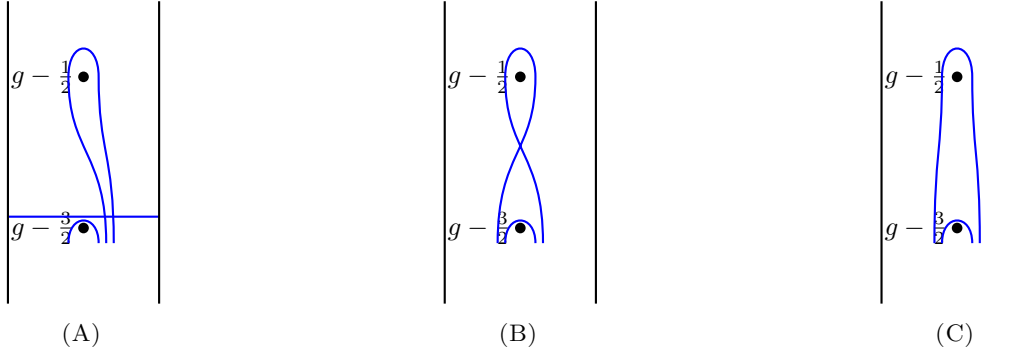
\begin{figure}[ht]\centering
\begin{subfigure}[t]{0.3\linewidth}\centering
   
\begin{tikzpicture}
    \fill[black] (0, 0) circle (2pt) node[anchor=east] {$ g-\frac{3}{2}$};
    \fill[black] (0, 2) circle (2pt) node[anchor=east] {$g-\frac{1}{2}$};

   \draw[blue, thick]  (0.3,-0.2) .. controls (0.3, 1) and (-0.2, 1) .. (-0.2, 2) .. controls (-0.2, 2.5) and (0.2, 2.5) .. (0.2, 2) .. controls (0.2, 1) and (0.4, 1)  .. (0.4, -0.2);

    \draw[blue, thick]  (-0.2,-0.2)  .. controls (-0.2, 0.2)  and  (0.2, 0.2) .. (0.2, -0.2) ;

    \draw[blue, thick]  (-1,0.15)  -- (1, 0.15) ;
   
   \draw[black, thick] (-1, -1) -- (-1, 3);
    \draw[black, thick] (1, -1) -- (1, 3);

\end{tikzpicture}\caption{}\label{fig:-1<FDTC<0}  
 \end{subfigure}
 \hfill
    \begin{subfigure}[t]{0.3\linewidth}\centering
        
 \begin{tikzpicture}
   \fill[black] (0, 0) circle (2pt) node[anchor=east] {$ g-\frac{3}{2}$};
    \fill[black] (0, 2) circle (2pt) node[anchor=east] {$g-\frac{1}{2}$};

\draw[blue, thick]  (0.3,-0.2) .. controls (0.3, 1) and (-0.2, 1).. (-0.2, 2) .. controls (-0.2, 2.5) and (0.2, 2.5) .. (0.2, 2) .. controls (0.2, 1) and (-0.3, 1) .. (-0.3, -0.2);

 \draw[blue, thick]  (-0.2,-0.2)  .. controls (-0.2, 0.2)  and  (0.2, 0.2) .. (0.2, -0.2) ;

   \draw[black, thick] (-1, -1) -- (-1, 3);
    \draw[black, thick] (1, -1) -- (1, 3);

\end{tikzpicture}\caption{}\label{fig:FDTC=0}
 \end{subfigure}
\hfill \begin{subfigure}[t]{0.3\linewidth}\centering
        
 \begin{tikzpicture}
   \fill[black] (0, 0) circle (2pt) node[anchor=east] {$ g-\frac{3}{2}$};
    \fill[black] (0, 2) circle (2pt) node[anchor=east] {$g-\frac{1}{2}$};

\draw[blue, thick]  (0.3,-0.2)  .. controls (0.3, 1) and (0.2, 1).. (0.2, 2) .. controls (0.2, 2.5) and (-0.2, 2.5) .. (-0.2, 2) .. controls (-0.2, 1) and (-0.3, 1).. (-0.3, -0.2);

 \draw[blue, thick]  (-0.2,-0.2)  .. controls (-0.2, 0.2)  and  (0.2, 0.2) .. (0.2, -0.2) ;

   \draw[black, thick] (-1, -1) -- (-1, 3);
    \draw[black, thick] (1, -1) -- (1, 3);

\end{tikzpicture}\caption{}\label{fig:FDTC=01}
 \end{subfigure}\hfill
\caption{Immersed curves for fibered knots $K_\phi$ with $-1<\FDTC(\phi)\leq0$, as in the statement of Lemma~\ref{lem:immersedcurve|FDTC|<1}.}\label{fig:immersedcurves4}
\end{figure}

\begin{proof}[Proof of Lemma~\ref{lem:immersedcurve|FDTC|<1}]
    Consider the immersed curve $\gamma$ of $K_\phi$.  Since $K_{\phi}$ is fibered, we have that; \begin{align*}\rank\big(\widehat{\HFL}(K_{\phi}, {[\Sigma]},g)\big)=1.\end{align*} 
    
    In particular, $\gamma$ intersects the vertical ray $\{0\}\times(g-\frac{1}{2},\infty)$ exactly once, at some point $c$ in the interval $\{0\}\times (g-\frac{1}{2},g+\frac{1}{2})$ and the component of the immersed curve on which $\gamma$ lies has trivial local system, proving the first part of the Lemma. The claims about the components of $\gamma\setminus\gamma_c$ follow from the fact that they cannot intersect the ray $\{0\}\times[g-\frac{1}{2},\infty)$, and that $\gamma$ is pulled tight.

    We first indicate the behavior of $\gamma$ after it leaves $c$ to the left along the arc $\gamma_l$. We will investigate $\gamma_l$'s behavior in two cases:\begin{enumerate}
        \item\label{case1} $\phi$ is non-weakly right-veering but $\Delta\circ\phi$ is weakly right veering. 
      \item\label{case2} $\phi$ is weakly right-veering but $\Delta^{-1}\circ\phi$ is non-weakly right veering.
    \end{enumerate}
    
We begin with Case~\ref{case1}. Let $Y$ be the $3$-manifold corresponding to the abstract open book $(\Sigma,\phi)$ and suppose that $\phi$ is non-weakly right veering. The map $\widehat{\HFK}(K_\phi, {[\Sigma]},g)\to\widehat{\HFK}(K_\phi, {[\Sigma]},g-1)$ induced by the spectral sequence from $\widehat{\HFK}(K_\phi)$ to $\widehat{\HF}(Y)$ is non-trivial by~\cite[Remark 1.4]{baldwin2022floerveering}. It follows that $\gamma$ contains a sub-arc, $\gamma_l$, in a small neighborhood of ${S^1\times[g+\frac{1}{2},g-\frac{3}{2}]}$ described as follows: First, $\gamma_l$ extends to the left from the intersection point $c$, intersects the line segment $\{0\}\times(g-\frac{1}{2},g-\frac{3}{2})$. Then, since $\gamma_l$ cannot intersect the $\{0\}\times(g-\frac{1}{2},\infty)$ again and $\gamma$ cannot contain components that wrap around a basepoint by a result of Hanselman-Rasmussen-Watson~\cite[P.992]{hanselman2018heegaard}, $\gamma_l$ proceeds downwards past $S^1\times\{g-\frac{3}{2}\}$, possibly after wrapping some number of times around the $S^1$ factor, along a line of slope in the range $(-\infty, 0)$. See Figure~\ref{fig:-1<FDTC}.

    \begin{figure}[ht]
\centering

 \begin{tikzpicture}

   \draw[blue, thick]  (1,1.5) .. controls (0.5,1.75) and (-0.2, 1.6) .. (-0.2, 2) .. controls (-0.2, 2.5) and (0.2, 2.5) .. (0.2, 2);
    \draw[blue, thick]  (-1,1.5) -- (1, 0.5);
    \draw[blue, thick]  (-1,0.5) -- (0.7, -0.1 );
    \draw[red, thick]  (0,3) -- (0, 0 );
     \draw[green, thick]  (-1,3) -- (1, 1 );
      \draw[green, thick]  (-1,1) -- (0, 0 );


   \draw[black, thick] (-1, -1) -- (-1, 3);
    \draw[black, thick] (1, -1) -- (1, 3);

      \fill[blue] (0, 2.4) circle (1pt) node[anchor=west] {$\textcolor{black}{c}$};
    
    \fill[black] (0, 0) circle (2pt) node[anchor=east] {$ g-\frac{3}{2}$};
    \fill[black] (0, 2) circle (2pt) node[anchor=east] {$g-\frac{1}{2}$};

\end{tikzpicture}
  
\caption{Here $\gamma$ --- shown in blue is the immersed curve of the binding of a non right-veering open book. $\gamma_l$ must form a bigon with the red curve and cannot form a bigon with the green curve. It follows that it must be of slope in the interval $(-\infty,-1)$, after pulling tight.}\label{fig:-1<FDTC}  
\end{figure}
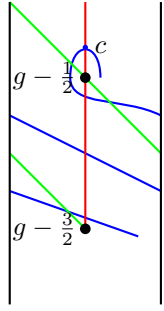

    Suppose additionally that $\Delta\circ\phi$ is weakly right-veering. Recall that $\widehat{\HFL}(K_{\Delta\circ\phi})$ can be recovered from $\gamma$ by intersecting with the line of slope $-1$ through $(0,g-\frac{1}{2})$, which we call $L_{-1}$. Since $L_{-1}$ and $\gamma$ cannot form a bigon --- by another application of~\cite[Remark 1.4]{baldwin2022floerveering} --- it follows that $\gamma_l$ must have slope in the interval $(-\infty,-1)$, as in Figure~\ref{fig:-1<FDTC<0}.

We now proceed to Case~\ref{case2}. Suppose that $\phi$ is weakly right-veering. By another application of~\cite[Remark 1.4]{baldwin2022floerveering}, the induced map $\widehat{\HFL}(K_\phi, {[\Sigma]},g)$ to $\widehat{\HFL}(K_\phi, {[\Sigma]},g-1)$ is trivial. It follows that $\gamma$ contains a sub-arc, $\gamma_l$, in a neighborhood of ${S^1\times[g-\frac{1}{2},g+\frac{1}{2}]}$ described as follows: $\gamma_l$ extends to the left from the intersection point $c$, and then, since the arc cannot intersect the $\{0\}\times(g-\frac{1}{2},\infty)$ again and cannot form a bigon with the line $\{0\}\times\R$, and the immersed curve invariant cannot contain components that wrap around a basepoint by~\cite[P.992]{hanselman2018heegaard},  $\gamma_l$ proceeds downwards past $S^1\times\{g-\frac{3}{2}\}$, possibly after wrapping some number of times around the $S^1$ factor, along a line of positive slope. See Figure~\ref{fig:FDTC>0}.

    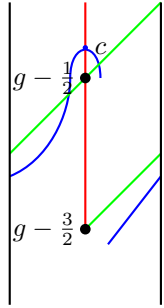
\begin{figure}[ht]
\centering

 \begin{tikzpicture}


   \draw[blue, thick]  (0.2,2) .. controls (0.2,2.5) and (-0.2,2.5) .. (-0.2, 2) .. controls (-0.2,1.5) and (-0.5,0.9) .. (-1,0.7);
    \draw[blue, thick]  (1,0.7) -- (0.3,-0.2 );
   
    \draw[red, thick]  (0,3) -- (0, 0 );
    
     \draw[green, thick]  (1,3) -- (-1, 1 );
      \draw[green, thick]  (1,1) -- (0, 0 );


   \draw[black, thick] (-1, -1) -- (-1, 3);
    \draw[black, thick] (1, -1) -- (1, 3);

      \fill[blue] (0, 2.4) circle (1pt) node[anchor=west] {$\textcolor{black}{c}$};
    
    \fill[black] (0, 0) circle (2pt) node[anchor=east] {$ g-\frac{3}{2}$};
    \fill[black] (0, 2) circle (2pt) node[anchor=east] {$g-\frac{1}{2}$};

\end{tikzpicture}
  
\caption{Here $\gamma$ --- shown in blue is the immersed curve of the binding of a right-veering open book. $\gamma_l$ cannot form a bigon with the red curve and must form a bigon with the green curve.}\label{fig:FDTC>0}  
\end{figure}

    Suppose additionally that $\Delta^{-1}\circ\phi$ is non-weakly right-veering. Recall that $\widehat{\HFL}(K_{\Delta^{-1}\circ\phi})$ can be recovered from $\gamma$ by intersecting with the line of slope $1$ through $(0,g-\frac{1}{2})$, which we call $L_{1}$. Since $L_{1}$ and $\gamma$ must form a bigon, it follows that $\gamma_l$ must have slope in the interval $(1,\infty)$.

   Having treated these Cases~\ref{case1} and~\ref{case2}, we can conclude the proof by identifying which cases can correspond to each possible value of the fractional Dehn twist coefficient. If $-1<\FDTC(\phi)<0$ then $\phi$ is non- right veering, $\Delta\circ\phi$ is strongly right veering, $\phi^{-1}$ is strongly right veering, and $\phi^{-1}\circ \Delta^{-1}$ is non- right veering. 
  
   If $\FDTC(\phi)=0$ then either:
   \begin{itemize}
       \item $\phi$ and $\phi^{-1}$ are both non-weakly right-veering while $\Delta\circ\phi$ and $\Delta\circ\phi^{-1}$ are both weakly right veering --- i.e. $\phi$ sends at least one arc strictly to the left, and at least one arc is sent strictly to the right.
        \item $\phi$ is  non-weakly right-veering while $\phi^{-1}$ is weakly right-veering --- i.e. $\phi$ sends at least one arc strictly to the left, and there is no arc which is sent strictly to the right.
          \item $\phi$ is  weakly right-veering and $\Delta^{-1}\circ\phi$ is non-weakly right-veering, while $\phi^{-1}$ is non-weakly right-veering and $\Delta\circ\phi^{-1}$ is weakly right-veering --- i.e. $\phi$ sends at least one arc  strictly to the right, and there is no arc which is sent strictly to the left.
          \item $\phi$ and $\phi^{-1}$ are both weakly right veering --- i.e. $\phi$ is the identity --- so that $\Delta^{-1}\circ\phi$ and $\Delta^{-1}\circ \phi^{-1}$ are both non-weakly right-veering.

   \end{itemize}
  
   Parts~\ref{item:1} and~\ref{item:2} of the Lemma now follow from the above analysis, which allows us to determine $\gamma_l$, as well as the symmetry properties of the immersed curve invariant and knot Floer homology, which allow us to determine $\gamma_r$ --- the component of $\gamma$ emanating from $c$ to the right --- from $\phi^{-1}$.  \end{proof}

\begin{remark}
A similar proof shows that if $-1-m<\FDTC(\phi)< -m$ if and only if the slope of $\gamma_l$ is in the range $(-\frac{1}{m+1},-\frac{1}{m})$. Thus, if $K$ is a fibered knot, the slope of the relevant part of the immersed curve --- which can alternately be interpreted as a measure of twisting in the region $S^1\times[g-\frac{1}{2},g+\frac{1}{2}]$ --- encodes information about the fractional Dehn twist of $\phi$. 
\end{remark}

\begin{remark}
Observe that Figure~\ref{fig:FDTC=01} occurs if and only if $\phi$ is the identity. Thus, the immersed curve in the region $S^1\times(g-\frac{3}{2},\infty)$ determines the entire immersed curve in this case. The entire immersed curve is computed in~\cite[Section 1.4]{hanselman2018heegaard}.
\end{remark}

\begin{proof}[Proof of Proposition~\ref{prop:newknotFDTC<0}]
Let $\phi:\Sigma\to\Sigma$ be a diffeomorphism with $\FDTC(\phi)\geq 0$. Observe that $\phi$ can be written as $(\Delta)^k\circ\phi'$ where $0\leq \FDTC(\phi')<1$. Recall that $K_\phi$ is the core of $-1/k$-surgery on the binding of the open book $\phi':\Sigma\to\Sigma$. The result now follows from Lemma~\ref{lem:immersedcurve|FDTC|<1}, observing that $\widehat{\HFL}(K_{(\Delta)^k\circ\phi'})$ and $\widehat{\HFL}(K_{(\Delta)^{k+1}\circ\phi'})$ can be computed from the intersection number of $\gamma_{\phi'}$ --- the immersed curve invariant for $K_{\phi'}$ --- with lines through $(0,g-\frac{1}{2})$ of slopes $-\dfrac{1}{k}$ and $\dfrac{-1}{k+1}$ respectively, in the region $S^1\times[g-\frac{3}{2},g-\frac{1}{2}]$. 

Note that for $k=0$, the rank stays the same if we are in the cases shown in Figure~\ref{fig:-1<FDTC<0}, or Figure~\ref{fig:FDTC=0}, but increases by two if we are in the cases shown in Figure~\ref{fig:FDTC=01}, or the reflection of the case shown in Figure~\ref{fig:-1<FDTC<0} about a vertical line. For $k>0$ the rank always increases by two.
\end{proof}

We can now prove Proposition~\ref{prop:HFKsingleboundarycomponent}. 
\begin{proof}[Proof of Proposition~\ref{prop:HFKsingleboundarycomponent}]
  Suppose $K$ is as in the statement of the Proposition. Let $\phi$ be the monodromy of $K$.  It suffices to prove the case in which $\FDTC(\phi)\geq 0$; the $\FDTC(\phi)< 0$ case follows from the symmetry properties of knot Floer homology. Observe that we can write $\phi:\Sigma\to\Sigma$ as $\Delta^n\phi'$ where $\phi':\Sigma\to\Sigma$ has $0\leq \FDTC(\phi')< 1$, and $n=\lfloor|\FDTC(\phi)|\rfloor$. Let $K'$ be the knot with monodromy $\phi'$. Observe that $\rank(\widehat{\HFL}(K', {[\Sigma]},g(\Sigma)-1))\geq 1$ by~\cite[Theorem 1.1]{baldwin_note_2018}. It then follows from at least $n-1$ applications of Equation~\ref{HFK:twists-1} from Proposition~\ref{prop:newknotFDTC<0} and at most one application of Equation~\ref{HFK:twists0} from Proposition~\ref{prop:newknotFDTC<0} that $\rank(\widehat{\HFL}(K, {[\Sigma]},g(\Sigma)-1))\geq 1+\max\{2(n-1),0\}$, as desired.\end{proof}

We can use the same techniques to quickly prove results related to work of Hedden-Mark~\cite{HeddenMark}. More specifically, Hedden-Mark showed that for any fixed $3$-manifold, $Y$, there is a uniform bound on the fractional Dehn twist coefficient of any fibered knot $K$ in $Y$~\cite[Theorem 1]{HeddenMark}.
\begin{proposition}\label{prop:versionofheddenmark}
    Suppose that $Y$ admits an open book decomposition $(\Sigma,\phi)$ where $\Sigma$ has a single boundary component. If $g(\Sigma)\geq 1$ then $$\rank(\widehat{\HF}(Y))\geq \max\{2\big\lfloor|\FDTC(\phi)|\rfloor-2,1\}.$$ If $g(\Sigma)>1$ then $${\rank(\widehat{\HF}(Y))\geq \max\{4\big\lfloor|\FDTC(\phi)|\rfloor-4},1\}.$$
\end{proposition}

\begin{proof}

We first give an immersed curve based proof that $\rank(\widehat{\HF}(Y))\geq 1$. Let $\gamma$ be the immersed curve of $\partial\Sigma$. Recall that the immersed curve of $\partial\Sigma$ contains a component that is homotopic to $\partial \Sigma$ --- i.e. a line of slope zero in our parameterization --- where here we allow our homotopy to pass through the basepoints~\cite[Corollary 6.6]{hanselman2016bordered}. It follows that the geometric intersection number of the line of slope $1/n$ with $\gamma$ is at least the geometric intersection number of the line of slope $1/n$ with the horizontal line, i.e at least one, as desired.

Since $\FDTC(\phi)=0$ when $g(\Sigma)=0$ --- so that the Proposition holds in this case --- it remains to show that if $g(\Sigma)\geq 1$ then $\rank(\widehat{\HF}(Y))\geq 2\big\lfloor|\FDTC(\phi)|\rfloor-2$ and that if $g(\Sigma)>1$ then $${\rank(\widehat{\HF}(Y))\geq 4\big\lfloor|\FDTC(\phi)|\rfloor-4}.$$ Since $\rank(\widehat{\HF}(Y))=\rank(\widehat{\HF}(-{Y}))$, it suffices to show this up to reversing the orientation of the underlying $3$-manifold. Perhaps after reversing the orientation of $Y$, we may assume that $\FDTC(\phi)\leq 0$. Set $n:=\lfloor|\FDTC(\phi)|\rfloor$. Let $K$ be the binding of $(\Sigma,\phi)$. Consider the immersed curve for $K_{-1/n}$. Since $-1<\FDTC(\Delta^n\circ\phi)\leq 0$, we can apply Lemma~\ref{lem:immersedcurve|FDTC|<1} to determine the immersed curve of $K_{-1/n}$ in the neighborhood $S^1\times(g-\frac{3}{2},\infty)$. Using the immersed curve formula for the Heegaard Floer homology of $1/n$ surgery on $K_{-1/n}$ we see that for $n\neq 0$;\begin{align*}\rank(\widehat{\HF}(Y))=\rank\Big(\widehat{\HF}(Y_{1/n}(K_{-1/n}))\Big)\geq  2(n-1)=2(\big\lfloor|\FDTC(\phi)|\big\rfloor-1).\end{align*} The $n=0$ case is vacuously true. Note that the lower bound comes from the curve shown in Figure~\ref{fig:-1<FDTC<0}. The other possible curves give the stronger rank bound $\rank(\widehat{\HF}(Y))\geq 2(\big\lfloor|\FDTC(\phi)|\big\rfloor)$.

If $g(\Sigma)>1$, then the symmetry properties of the immersed curve invariant dictate that in the neighborhood $S^1\times(-\infty,\frac{3}{2}-g)$ the immersed curve is given by rotating the 
immersed curve in $S^1\times(g-\frac{3}{2},\infty)$ by $180^\circ$. The rank bound then follows as before. 
\end{proof}

\section{The multi-boundary component case}\label{sec:HFL}

In this section, we determine the behavior of the rank of the link Floer homology of fibered links in the next to bottom Alexander grading under adding Dehn twists. 

Let $\Delta_i$ denote a positive Dehn twist about the $i$th boundary component of $\partial\Sigma$.  Let $V$ be a rank two vector space. Our main technical results are the following two propositions:

\begin{proposition}\label{thm:HFKaddtwists}
Let $\Sigma$ be a surface with $|\partial\Sigma|\geq 2$.   Suppose that $\phi:\Sigma\to\Sigma$ is a diffeomorphism of $\Sigma$ such that $\phi\circ\Delta_i^2$ sends some arc $a$ with exactly one endpoint on $\partial_i\Sigma$ weakly to the left at $\partial_i\Sigma\cap a$. Then\begin{align*}  
 \widehat{\HFL}(L_{\phi}, {[\Sigma]},1-G)\cong \widehat{\HFL}(L_{\phi\circ\Delta_i}, {[\Sigma]},1-G)\oplus V.\end{align*}
\end{proposition}

Recall here that $G=g(\Sigma)+m-1$. We have suppressed the underlying $3$-manifold in our notation for link Floer homology, and are abusing the notation by letting $\Sigma$ denote distinct Seifert surfaces for distinct knots in distinct manifolds.

\begin{proposition}\label{thm:HFKaddtwists1}
Let $\Sigma$ be a surface with $|\partial\Sigma|\geq 2$.   Suppose that $\phi:\Sigma\to\Sigma$ is a diffeomorphism of $\Sigma$ such that $\phi\circ\Delta_i$ sends some arc $a$ with exactly one endpoint on $\partial_i\Sigma$ weakly to the left at $\partial_i\Sigma\cap a$. Then either \begin{align*}  
 \widehat{\HFL}(L_{\phi}, {[\Sigma]},1-G)\cong \widehat{\HFL}(L_{\phi\circ\Delta_i}, {[\Sigma]},1-G)\oplus V,\end{align*}
\noindent or 
\begin{align*}  
 \widehat{\HFL}(L_{\phi}, {[\Sigma]},1-G)\cong \widehat{\HFL}(L_{\phi\circ\Delta_i}, {[\Sigma]},1-G).\end{align*}
\end{proposition}

Note that these statements play the same role as Proposition~\ref{prop:newknotFDTC<0};  while we have stated the results in this section in terms of isomorphisms of vector spaces, as opposed to ranks of vector spaces in the previous section, this distinction is purely cosmetic, since we are ignoring other gradings in this paper. We prove Proposition~\ref{thm:HFKaddtwists}  using the following surgery exact triangle. This triangle is essentially that given in~\cite[Theorem 8.2]{Holomorphicdisksandknotinvariants} translated into our context, where, in particular, we are not ``knotifying" our links.

\begin{lemma}\label{lem:generalexact}
    
If $\mathcal{L}$ is a link in a $3$-manifold $Z$ and $\eta$ is a framed knot in the complement of $\mathcal{L}$ then there is an exact triangle:

\begin{equation}\label{eq:generaltriangle}
\begin{tikzcd}
   \widehat{\HFL}(Z_1(\eta),\mathcal{L}) \ar[rr,"f_3^*"]&& \widehat{\HFL}(Z,\mathcal{L})\ar[dl,"f_1^*"]\\&\ar[ul,"f_2^*"] \widehat{\HFL}(Z_{0}(\eta),\mathcal{L})&
 \end{tikzcd}
\end{equation}
Moreover, $f_1^*$ is induced by a chain level count of pseudo-holomorphic triangles and if $\eta$ is in the complement of a Seifert surface for $\mathcal{L}$ then the exact triangle splits over Alexander gradings.
\end{lemma}

Here we are abusing notation by allowing $\mathcal{L}$ to denote links in three distinct manifolds. This result is well known to experts, but we were unable to find a specific reference in the literature --- for example the proof of the corresponding fact is only proven for ``knotified" links in~\cite[Theorem 8.2]{Holomorphicdisksandknotinvariants} --- so we include a proof sketch here for the sake of completeness. This essentially amounts to restating the proof of~\cite[Theorem 8.2]{Holomorphicdisksandknotinvariants}.
\begin{proof}[Proof Sketch]
A version of this surgery exact triangle was proven for 3-manifolds in~\cite[Section 9]{OSHoldisks3propapp}. The proof of the three manifold case generalizes to the link case since all of the maps used in the proof in the $3$-manifold case are filtered with respect to the filtration induced by the extra basepoints encoding the link $L$.
\end{proof}

We consider now a special case of Lemma~\ref{lem:generalexact}. We consider the Heegaard diagram $(S,\bm{\beta},\bm{\alpha},\bm{z},\bm{w})$ for $(-Y,-L)$ as constructed in Section~\ref{sec:Heegaard} from an open book decomposition $(\Sigma,\phi)$ for a manifold $Y$. We take:\begin{enumerate}
    \item $Z$ to be $-{Y}$.
    
    \item ${\eta}$ to be the framed knot in $-{Y}$ consisting of a push-off of the $1$st boundary component of $\partial\Sigma$ into the $\beta$-handlebody with framing induced by $-\Sigma$ .
    
    \item $\mathcal{L}\subset Z$ to be $-L_\phi\subset-{Y}$, the boundary of $-\Sigma$ in $-{Y}$. Consequently $\mathcal{L}\subset Z_1$ is $-L_{\phi\circ\Delta_1}\subset(-{Y})_1({\eta})$, while $\mathcal{L}\subset Z_0$ we denote by $-{L}\subset (-{Y})_0({\eta})$. 
\end{enumerate}

    To see why ${\mathcal{L} \subset Z_{1}}$ is $-L_{\phi\circ\Delta_1}$: first observe that the manifold with open book decomposition $(\phi \circ \Delta_1, \Sigma)$ is homeomorphic to $Y_{-1}(\eta)$. By reversing the orientation, this gives us that $-({Y_{-1}(\eta)})$ is homeomorphic to the orientation reversal of the open book decomposition $(\phi \circ \Delta_1, \Sigma)$ whose binding is $-{L_{\phi\circ\Delta_1}}$. But $-(Y_{-1}(\eta))$ is homeomorphic to $(-{Y})_{1}({\eta})$, giving us the desired relationship.

   Restricting to Alexander grading $1-G$ (recalling that $G=g(\Sigma)+|\partial\Sigma|-1$) the exact triangle~(\ref{eq:generaltriangle}) then reduces to:

\begin{equation}\label{eq:exactneg}
\begin{tikzcd}[column sep= 0cm]
    \widehat{\HFL}((-{Y})_{1}({\eta}),-{L_{\phi\circ\Delta_1}},{-[\Sigma]},1-G)\ar[rr,"f_3^*"]&& \widehat{\HFL}(-{Y},-{L_{\phi}},{-[\Sigma]},1-G)\ar[dl,"f_1^*"]\\& \widehat{\HFL}((-{Y})_{0}({\eta}),-{L},{-[\Sigma]},1-G)\ar[ul,"f_2^*"]&
\end{tikzcd}
\end{equation}

\noindent Here again $f_1^*$ is induced by counts of pseudo-holomorphic triangles.

Our strategy for proving Proposition~\ref{thm:HFKaddtwists} consists of two steps. First we show that $$\widehat{\HFL}(-{Y}_{0}({\eta}),-{L},{-[\Sigma]},1-G)\cong V$$ and is generated by classes that can be explicitly found in a specially chosen Heegaard diagram; see Lemma~\ref{lem:0surgery}. Then, under appropriate hypotheses on arcs as in the statement of Proposition~\ref{thm:HFKaddtwists}, we show that $f_1^*$ is surjective by finding generators of $\widehat{\HFL}(-{Y},-{L_{\phi}},{-[\Sigma]},1-G)$ that map to the two generators of $\widehat{\HFL}(-{Y}_{0}({\eta}),-{L},{-[\Sigma]},1-G)$. This implies that $f_{3}^*$ is injective and since $\im f_{3}^* = \ker f_{1}^*$, this concludes the proof. The proof of Proposition~\ref{thm:HFKaddtwists1} is similar, but we only show that $f_1^*$ is non-trivial.

\subsection{Capping off Boundary components} We first investigate $\widehat{\HFL}(-{Y_{0}(\eta)},-{L})$. $(-{Y_{0}(\eta)},-{L})$ admits a Heegaard diagram $\mathcal{H}_0(-{L}):=(S,\bm{\gamma},\bm{\alpha},\bm{z},\bm{w})$ as shown in Figure~\ref{fig:cleanedsurgery}. Specifically we may take a Heegaard diagram $(S,\bm{\beta},\bm{\alpha},\bm{z},\bm{w})$ adapted to the binding of the open book $(\Sigma,\phi^{-1})$ which is a Heegaard diagram for $-{L}$ in $-{Y}$ --- and any collection of arcs $\{a_i\}$ such that only $a_1$ intersects $\partial_1\Sigma$ and does so at exactly one point. See Section~\ref{sec:Heegaard} for more details. Note that $\beta_i$ and $\widetilde{\beta}_j$ denote the beta-curves. Recall too that $\widetilde{\beta}_i, \widetilde{\alpha}_i$ denote the necklace curves.

 \begin{figure}[ht]
\centering

   \begin{tikzpicture}


\draw[ thick]
(-1,-3.5)
.. controls (3,-3.5) and (3,3.5) ..
(-1,3.5)
.. controls (-7,3.5) and (-7,-3.5) ..
(-1,-3.5);

\draw[red, thick] (0,0) ellipse (1 and 3.0);

\draw[red, thick] (-3,0) ellipse (1.8 and 1.0);

\draw[ draw=black, thick] (0,2) ellipse (0.5 and 0.3);
\draw[draw=black, thick] (0,-2) ellipse (0.5 and 0.3);
\draw[draw=black, thick] (-3.25,0) ellipse (0.3 and 0.5);

\node at (-4,0) {$w_1$};
\node at (-5.1,0) {$z_1$};
\node at (-1.5,0) {$w_2$};
\node at (-2.5,0) {$z_2$};

\node at (-4.5,1.5) {$\gamma_1$};

\node at (-3,1.2) {$\alpha_1$};

\draw[green, thick] (-4.6,0) ellipse (0.8 and 0.5);

\draw[green, thick]
(-2,0)
.. controls (-2,-1) and (-1.0,-1.0) ..
(0,-1.7);

\draw[green, dashed, thick]
(0,-1.7) .. controls (1,-1.5) and (1.0,-1.4) ..
(1.7,-1.5);

\draw[green, thick]
(1.7,-1.5) .. controls (0,-1) and (-1.8,-0.5) ..
(-1.8,0);

\draw[green, thick]
(-3.5,0.3) .. controls (-4,1.5) and (-5,1.5) ..
(-5.3,1.2);

\draw[green, dashed, thick]
(-3.5,0.3) .. controls (-4,1) and (-5,1) ..
(-5.3,1.2);


\draw[red, thick]
(-3.5,0.2) .. controls (-4.5,1) and (-4.5,-1) ..
(-3.5,-0.2);

\draw[red, thick,dashed]
(-3.5,0.2) .. controls (-4,0.8) and (-5,0.8) ..
(-5.5,0.2);

\draw[red, thick,dashed]
(-3.5,-0.2) .. controls (-4,-0.8) and (-5,-0.8) ..
(-5.5,-0.2);

\draw[red, thick]
(-5.5,0.2) .. controls (-4.7,1) and (-4.7,-1) ..
(-5.5,-0.2);


\draw[green, thick]
(-2.2,0)
.. controls (-2.2,-1.5) and (-0.5,-2.5) ..
(0,-2.5)
.. controls (0.1,-2.5) and (0.5,-2.5) ..
(0.6,-2)
.. controls (0.6,-0.7) and (-1.9,-1) .. (-1.9,0);

\draw[red, thick, dashed]
(0,1.7)
.. controls (0.25,0.5) and (0.25,-0.5)..
(0.0,-1.7);
\draw[red, thick]
(0,1.7)
.. controls (-0.25,0.5) and (-0.25,-0.5)..
(0.0,-1.7);

\node at (-3.8,1.1) {$x$};
\filldraw[black] (-3.8,0.85) circle (2pt);

\node at (-3.9,-0.65) {$v_+$};
\filldraw[black] (-4,-0.35) circle (2pt);
\node at (-5,-0.7) {$v_-$};
\filldraw[black] (-5.1,-0.4) circle (2pt);

\end{tikzpicture}
    \caption{ $\mathcal{H}_0(-{L})$, a Heegaard diagram for $-{L}$ in the orientation reversal of $0$-surgery on a component of the binding of a multi-component fibered link. The green curves are the $\gamma$ curves and the red curves are the $\alpha$ curves.}\label{fig:cleanedsurgery}
   \end{figure}
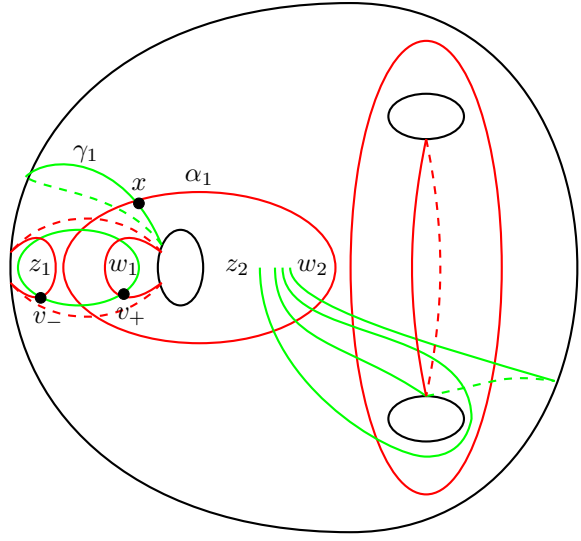

We replace $\beta_1$ with a new curve, which we denote $\gamma_1$, given by a push-off of $\partial_1\Sigma$ into $-{\Sigma}$. The other $\beta$ and $\widetilde{\beta}$ curves we leave unchanged, but we re-label them as $\gamma$ and $\widetilde{\gamma}$ curves. This yields the desired Heegaard diagram.

Let $c_i'$ denote the intersections points between the $\gamma_i$ and $\alpha_i$ curves for $i>1$ and $\widetilde{c}_i'$ denote the intersections between the $\widetilde{\gamma}_i$  and the $\widetilde{\alpha}_i$ curves that correspond to the BRAID invariant --- see Section~\ref{sec:braidinvt}. Let $v_\pm$ denote the two intersection points between $\widetilde{\gamma}_1$ and $\widetilde{\alpha}_1$ as shown in Figure~\ref{fig:cleanedsurgery}. Let $x$ denote the unique element in $\gamma_1\cap\alpha_1$. Set $\bm{v}_\pm:=\{v_\pm, x,c_2',\dots c_N',\widetilde{c}_3',\widetilde{c}_4'\dots \widetilde{c}_{n}'\}\in\widehat{\CFL}(\mathcal{H}_0(-{L}))$.  Note that $\widetilde{c}_2'$ does not appear in this expression because in the Heegaard diagram for $L$ we have chosen, the second boundary component of $\Sigma$ does not have necklace curves, so that $\widetilde{c}_2'$ is not even defined.  Here $n:=|\partial\Sigma|$ and $N:=2g(\Sigma)+n-1$. We will show that these classes are non-trivial in homology. 

\begin{remark}\label{rem:notation}
The proof of Lemma \ref{lem:0surgery} and later arguments in Section \ref{subsec:exacttriangle} will involve several different Heegaard diagrams for different parts of the exact triangle (\ref{eq:exactneg}), and at many points these diagrams will have to be superimposed. In order for later arguments to be easier to parse, we will use the prime notation $c_{i}'$ and $\widetilde{c_{i}}'$ for the intersection points corresponding to the BRAID invariant whenever we are working with any of the Heegaard diagrams associated to the $0$-surgery manifold in Lemma \ref{lem:0surgery}. While simplifying exposition further on, this will unfortunately result in different points on the same picture having the same label in one part of Lemma \ref{lem:0surgery}. We will alert the reader when that happens.
\end{remark}

   \begin{figure}[ht]\centering

        \centering
        \begin{tikzpicture}[scale=0.7]


\draw[red,thick] (0,0) circle(3);

\draw[gray,thick, dashed] (0,0) circle(1);
\draw[gray,thick,dashed] (0,0) circle(5);

\draw[thick, green] (160:1) -- (160:5);

\draw[thick, green] (-20:1) -- (-20:5);
\draw[thick, green] (50:1) -- (50:5);
\draw[thick, green] (70:1) -- (70:5);

\draw[thick, green] (175:1) -- (175:5);
\draw[thick, green] (185:1) -- (185:5);

\draw[gray] (-5,0)--(-1,0);
\draw[gray] (5,0)--(1,0);

 \node at (-4,0) {\( z_1 \)};
   \node at(-2,0) {\( w_1 \)};

    \node at(2.5,0) {\( z_2 \)};
     \node at(160:4.5) {\( \gamma_1 \)};
       \node at(185:4.5) {\( \widetilde{\gamma}_1 \)};\

         \node at(270:2.5) {\( \alpha_1 \)};


    \draw[->, cyan, thick] (50:3.2) arc[start angle=50, end angle=159, radius=3.2];
      \draw[->, cyan, thick] (70:2.8) arc[start angle=70, end angle=159, radius=2.8];

     \draw[->, cyan, thick] (-20:3.3) arc[start angle=-20, end angle=-199, radius=3.3];

            \draw[->, cyan, thick] (175:2.6) arc[start angle=175, end angle=159, radius=2.6];
             \draw[->, cyan, thick] (185:2.8) arc[start angle=185, end angle=159, radius=2.8];

\end{tikzpicture}

\caption{A neighborhood of $\alpha_1$ in $\mathcal{H}_0(-{L})$, a Heegaard diagram for $-{L}$ in the orientation reversal of $0$-surgery on a component of the binding of a multi-component fibered link. The cyan arrows indicate some of the handleslides needed to obtain $\mathcal{H}_0'(-{L})$ in the proof of Lemma~\ref{lem:0surgery}. The lower half of the annulus lies in $\Sigma$, the upper half in $-\Sigma$.}~\label{fig:stabilizedfiberedsplitunknot} 
\end{figure}
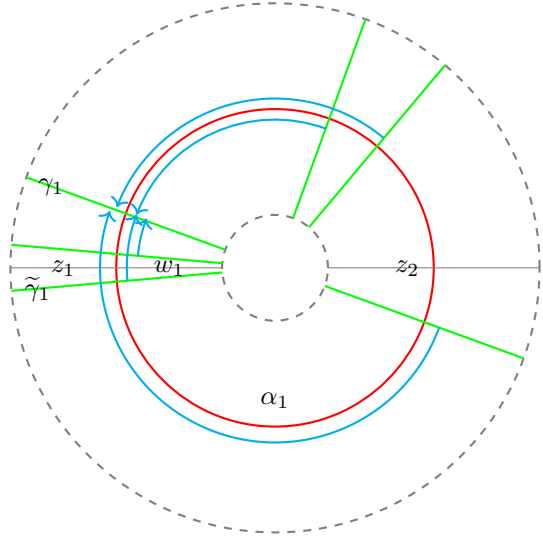

\begin{lemma}\label{lem:0surgery}
    $\widehat{\HFL}(\mathcal{H}_0(-{L}), {-[\Sigma]},1-G)\cong \F\langle\bm{v}_+,\bm{v}_-\rangle$.
\end{lemma}

 For the proof we appeal to work of Vela-Vick~\cite{VelaVicktransverseinvariant} and Tovstopyat-Nelip~\cite{tovstopyat2018transverse} showing that the transverse invariants of transverse approximations of the bindings of open books are non-trivial in link Floer homology.

\begin{proof} In this proof it will be useful to specify which link a Heegaard diagram is for: for instance, $\mathcal{H}_{0}$ will be referred to as $\mathcal{H}_{0}(-{L})$. We first produce a Heegaard diagram $\mathcal{H}'_0(-{L})=(S',\bm{\gamma'},\bm{\alpha'},\bm{z},\bm{w})$ for $-{L}$ viewed as a link in $-{Y}_0({\eta})$, as illustrated in Figure~\ref{fig:fiberedsplitunknot1}. To do so, observe that $-{L\setminus L_1}$ is fibered in $-{Y}_0({\eta})$ and $-{L_1}$ is a split unknotted component. Here $L_1$ is the component of $L$ corresponding to the basepoints $z_1$ and $w_1$. $\mathcal{H}'_0(-{L})$ is obtained from $\mathcal{H}_0(-{L})$, which is defined before the statement of the lemma, as follows: for each intersection point $\gamma_i\cap\alpha_1\cap-{\Sigma}$, perform a handleslide of $\gamma_i$ over $\gamma_1$ in a neighborhood of $\alpha_1\cap-{\Sigma}$. Recall here that $-{\Sigma}$ is the ``upper half" of $S'$. Likewise, for each intersection point $(\gamma_i\cup\widetilde{\gamma}_1)\cap\alpha_1\cap\Sigma$ perform a handleslide of $\gamma_i$ over $\gamma_1$ in a neighborhood of $\alpha_1\cap\Sigma$. See Figure~\ref{fig:stabilizedfiberedsplitunknot}. Note that at this stage we have a stabilized Heegaard diagram, $\mathcal{H}_0''(-{L})$ since $\alpha_1$ and $\gamma_1$ intersect once and do not intersect any of the other curves.  $\mathcal{H}'_0(-{L})$ is then defined as the Heegaard diagram obtained by destabilizing $\mathcal{H}_0''(-{L})$ with $\bm{\alpha'}$ and $\bm{\gamma'}$ the natural images of $\bm{\alpha}\setminus\{\alpha_1\}$ and $\bm{\gamma}\setminus\{\gamma_1  \}$ in $\mathcal{H}_{0}(-{L})$ under this procedure.

\begin{figure}[ht]\centering

   \begin{tikzpicture}

\draw[thick] (-1,0) ellipse (3 and 3);

\draw[red, thick] (0,0) ellipse (1 and 2.5);

\draw[ draw=black, thick] (0,1.5) ellipse (0.5 and 0.3);
\draw[draw=black, thick] (0,-1.5) ellipse (0.5 and 0.3);

\node at (-3,0.25) {$z_1$};
\node at (-3,-0.25) {$w_1$};
\node at (-1.3,0) {$w_2$};
\node at (-2.2,0) {$z_2$};

\node at (-2.5,-0.6) {$\tilde{\gamma}_1'$};
\node at (-3.7,-0.6) {$\tilde{\alpha}_1'$};

\draw[green, thick] (-3,0) circle (0.5);
\draw[red] (-3.1,0) circle (0.5);

\draw[green, thick]
(-1.8,0)
.. controls (-1.8,-2) and (-1.0,-1.0) ..
(0,-1.2);

\draw[green, dashed, thick]
(0,-1.2) .. controls (1,-1) and (1.0,-1) ..
(1.8,-1);

\draw[green, thick]
(1.8,-1) .. controls (0,0) and (-1.6,-1) ..
(-1.6,0);


\draw[green, thick]
(-2,0)
.. controls (-2,-2) and (-0.5,-2) ..
(0,-2)
.. controls (0.1,-2) and (0.5,-2) ..
(0.6,-1.5)
.. controls (0.6,0) and (-1.7,-1.5) .. (-1.7,0);

\draw[red, thick, dashed]
(0,1.2)
.. controls (0.25,1) and (0.25,-1)..
(0.0,-1.2);

\draw[red, thick]
(0,1.2)
.. controls (-0.25,1) and (-0.25,-1)..
(0.0,-1.2);

\node at (-3,0.8) {$v_+'$};
\filldraw[black] (-3,0.5) circle (2pt);

\node at (-3,-0.8) {$v_-'$};
\filldraw[black] (-3,-0.5) circle (2pt);

\end{tikzpicture}

\caption{A Heegaard diagrams, $\mathcal{H}_0'(-L)$, for the split sum of a fibered link and an unknot. The portions of the $\gamma$ curves in $-\Sigma$ are not shown. }\label{fig:fiberedsplitunknot1}
\end{figure}
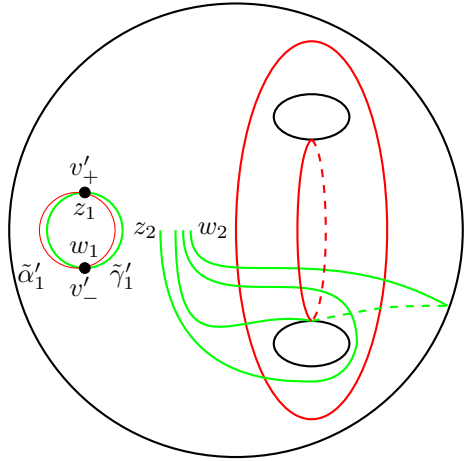

Note that we are assuming, without loss of generality, that $L_2$ is the component of $L\setminus L_1$ such that the corresponding boundary component of $\Sigma$ does not have ``necklace" curves. Let $\widetilde{\alpha}_1',\widetilde{\gamma}_1'$ be the null-homologous alpha and gamma-curves enclosing the basepoints $z_1$ and $w_1$ in $S'$, the Heegaard surface from $\mathcal{H}_0'(-{L})$, see Figure \ref{fig:fiberedsplitunknot1}. By an abuse of notation, as outlined in Remark \ref{rem:notation}, we let $c_i'$ denote the intersection points $\alpha_i'\cap\gamma_i'$ for $i> 1$ and $\widetilde{c}_i'$ denote the intersection points in $\widetilde{\alpha}_i'\cap\widetilde{\gamma}_i'$ for $i>2$. These intersections points correspond to the BRAID invariant for $L\setminus L_1$.

   Now consider only $L\setminus L_1$.  Observe that $\mathcal{H}_0'(-{L\setminus L_1}):=(S',\bm{\gamma'}\setminus\{\widetilde{\gamma}_1'\},\bm{\alpha'}\setminus\{\widetilde{\alpha}_1'\},\bm{z}\setminus\{z_1\},\bm{w}\setminus\{w_1\})$ is a Heegaard diagram adapted to $-{L\setminus L_1}$. Write $S'=\Sigma_0\cup-{\Sigma}_0$, with $\Sigma_0$ corresponding to the lower half of the Heegaard surface $S$ (in Figure~\ref{fig:fiberedsplitunknot1}), $-{\Sigma}_0$ corresponding to the upper part of the Heegaard surface $S'$ and $-{L\setminus L_1}$ isotopic to $\partial\Sigma_0=\partial(-{\Sigma}_0)$. Note that $\Sigma_0$ is a Seifert surface for $-L\setminus L_1$ in $-({Y_0}(\eta))$. Indeed, it is a minimum genus Seifert surface, since it is a page in the open book corresponding to the fibered knot $-{L\setminus L_1}$. We now proceed to find explicit generators for $\widehat{\HFL}(\mathcal{H}'_0(-{L}),[-\Sigma],1-G))$. We have two cases according to whether $\Sigma_0$ is a disk or it is not. 
   
   If $\Sigma_0$ is not a disk then it follows from~\cite[Theorem 1]{VelaVicktransverseinvariant} in the case that $L\setminus L_1$ has one component and from~\cite[Theorem 1.1]{tovstopyat2018transverse} in general that the BRAID invariant of an appropriate transverse approximation of the binding of an open book is non-trivial --- i.e. that  $\{c_2',c_3'\dots c_{N}', \widetilde{c}_3', \widetilde{c}_4'\dots \widetilde{c}_{n}'\}$ --- is non-trivial in $\widehat{\HFL}(\mathcal{H}_0'(-{L\setminus L_1}))$. Indeed, it follows from~\cite[Theorem 6.2]{cavallo2017invariant} that the BRAID invariant of such an approximation has $ A_{[\Sigma_0]}$ grading given by $\dfrac{\mathit{sl}(L\setminus L_1)+(n-1)}{2}$, where $n$ is the number of components of $L$ and $\mathit{sl}(L\setminus L_1)$ is the self-linking number of the transverse approximation of $L\setminus L_1$. The self-linking number of a transverse approximation of the binding of an open book is given by $\mathit{sl}(L\setminus L_1)=-\chi(L\setminus L_1)=2g(\Sigma_0)+|L\setminus L_1|-2$, since $\Sigma_0$ is a page of the open book where $\chi(L\setminus L_1)$ is the maximal Euler characteristic of a surface bounded by $L\setminus L_1$.  It follows that the Alexander grading of the BRAID invariant of a transverse approximation of $L\setminus L_1$, viewed as an element of $\widehat{\CFL}(\mathcal{H}_0'(-{L\setminus L_1}))$, has  $ A_{[-\Sigma_0]}$-grading given by
   \begin{align*}\dfrac{1-n-\mathit{sl}(L\setminus L_1)}{2}&=\dfrac{1-n+\chi(\Sigma_0)}{2}\\&=\dfrac{4-2n-2g(\Sigma_0)}{2}\\&=2-n-g(\Sigma)\\&=1-G.\end{align*}

   Since $-{L}$ contains $-{L_{1}}$ as a split unknotted component, one can obtain $-{L}$ by performing a connected sum of $-{L\setminus L_{1}}$ with a two-component unlink. Using the chain level K\"unneth formula~\cite[Theorem 11.1]{HolomorphicdiskslinkinvariantsandthemultivariableAlexanderpolynomial}, we can compute the link Floer homology of the connect sum of $-{L\setminus L_1}$ with a two component unlink $U_1\sqcup U_2$ bounding two disks $D_1\sqcup D_2$. Assume, without loss of generality, that the connect sum operation is between the components $-{L_2}\subset -{L\setminus L_1}$ and $U_2$. Note that $[\Sigma]=[\Sigma_0\cup D_1]\in H_1(-({Y_0}(\eta)),-{L})$, where we view $\Sigma$ as a surface in $-({Y_0}(\eta))$. We have that 
   \begin{align*}\widehat{\CFL}(\mathcal{H}'_0(-{L}),{-[\Sigma]},1-G)&\cong\widehat{\CFL}(\mathcal{H}'_0(-{L}),{-[\Sigma_0\cup D_1]},1-G)\\&\cong \widehat{\CFL}(\mathcal{H}'_0(-{L\setminus L_{1}}),{-[\Sigma_0]},1-G)\otimes \widehat{\CFL}(U_1\sqcup U_2,{-[D_1\sqcup D_2]},0)\\&\cong\widehat{\CFL}(\mathcal{H}'_0(-{L\setminus L_{1}}),{-[\Sigma_0]},1-G)\otimes V\end{align*} 
 where $V$ is a rank two vector space supported in Alexander grading zero, and $\widehat{\CFL}(\mathcal{H}'_0(-{L\setminus L_{1}}),{-[\Sigma_0]},1-G)$ is of rank one, since $L\setminus L_{1}$ is fibered and $\widehat{\CFL}(\mathcal{H}'_0(-{L\setminus L_{1}}),{-[\Sigma_0]})$ has minimal grading $1-G$. Indeed, from the proof of~\cite[Theorem 11.1]{HolomorphicdiskslinkinvariantsandthemultivariableAlexanderpolynomial}, we see that $\widehat{\HFL}(\mathcal{H}'_0(-{L}),[-\Sigma],1-G))$ is generated by $\{v_\pm',c_2',c_3',\dots c_{N}',\widetilde{c}_3'\dots \widetilde{c}_{n}'\}$. Here $v_\pm'$ are the two intersection points between $\widetilde{\gamma}_1'$ and $ \widetilde{\alpha}_1'$ as shown in Figure~\ref{fig:fiberedsplitunknot1}. In the case that $\Sigma_0$ is a disk, we still have that $\widehat{\HFL}(\mathcal{H}'_0(-{L}),[-\Sigma],1-G))$ is generated by $\{v_\pm'\}$, which can be seen by inspection.

   Having found our explicit generators for $\widehat{\HFL}(\mathcal{H}'_0(-{L}),[-\Sigma],1-G))$, we proceed on with the proof. To do so, we now go in reverse from ${\mathcal{H}_{0}'(-{L})}$ to ${\mathcal{H}_{0}(-{L})}$. Stabilize ${\mathcal{H}_0'(-{L})}$ near the basepoint $w_2$ on the second boundary component of $-{\Sigma}$ by adding new curves $\alpha_1'$ and $\gamma_1'$ as shown in Figure~\ref{fig:stabilizedfiberedsplitunknot} to obtain a Heegaard diagram ${\mathcal{H}''_0(-{L})=(S,\bm{\gamma'}\cup \{\gamma_1'\},\bm{\alpha}'\cup\{\alpha_1'\},\bm{z},\bm{w})}$ . The generators of $\widehat{\HFL}(\mathcal{H}_0''(-{L}), {[\Sigma]},1-G)$ are $${\bm{v}_\pm':=\{x,v_\pm',c_2',c_3',\dots c_{N}',\widetilde{c}_3',\dots\widetilde{c}_{n}'\}},$$ where $x$ is the unique element of $\alpha_1'\cap\gamma_1'$ and $N:=2g(\Sigma)+|\partial\Sigma|-1$. Finally, perform isotopies and handleslides as indicated by the arrows in Figure~\ref{fig:handleslide1} to obtain the Heegaard diagram $\mathcal{H}_0(-{L})$ for $-{L}$. There is an isomorphism
$$f:\widehat{\HFL}(\mathcal{H}_0''(-{L}),{-[\Sigma]},1-G)\to \widehat{\HFL}(\mathcal{H}_0(-{L}),{-[\Sigma]},1-G)$$ 

\noindent given by a sequence of maps induced by a sequence of isotopies and handleslides; i.e. counts of pseudo-holomorphic triangles; see~\cite[Section 7.2, Section 9]{ozsvath2004holomorphic}. Only the gamma curves change during these isotopies and handleslides, so to simplify exposition for the remainder of this argument we denote the set of alpha curves in both Heegaard diagrams by a general $\bm{\alpha}$. Care must be taken to guarantee that the relevant Heegaard triple diagrams are weakly admissible, but this can be checked by observing that every $\alpha,\gamma$ and $\gamma'$-curve has a region with a basepoint in it on either side.

   \begin{figure}[h!]
    \centering

    \begin{subfigure}{0.45\linewidth}
        \centering
        \begin{tikzpicture}[scale=0.5]

\draw[black,dashed] (-3,-2.5) --(-3,4.5);
\draw[black,dashed] (3,-2.5) --(3,4.5);

\draw[green,thick] (0,1) circle(2);
\draw[red,thick] (0,0) circle(2);

\draw[green, thick] (-3,4) -- (3,4);
\draw[gray] (-3,0) -- (3,0);

\draw[red, thick] (2.5,-2.5) -- (2.5,4.5);

  \filldraw[black] (1,0) circle (2pt);
   \filldraw[black] (-1,0) circle (2pt);


      \draw[->, cyan, thick] (0.1,-1) -- (0.1,4);

       \draw[->, cyan, thick] (-0.1,3) -- (-0.1,4);


         \node at (1,4.5) {\( \gamma_1'\)};

           \node at (-2,2.5) {\( \widetilde{\gamma}_1'\)};
          
           \node at (-2,-2) {\( \widetilde{\alpha}_1'\)};

            \node at (2.5,5) {\( \alpha_1'\)};

\end{tikzpicture}
        \caption{}\label{fig:handleslide1}
    \end{subfigure}
    \hfill
    \begin{subfigure}{0.45\linewidth}
        \centering
        \begin{tikzpicture}[scale=0.5]

\draw[black,dashed] (-3,-2.5) --(-3,4.5);
\draw[black,dashed] (3,-2.5) --(3,4.5);

\draw[red,thick] (0,0) circle(2);

\draw[red, thick] (2.5,-2.5) -- (2.5,4.5);

\draw[green, thick] (-3,4) -- (3,4);

\draw[gray] (-3,0) -- (3,0);

\draw[green, thick] (-3,-0.5) -- (-1,-0.5) arc[start angle=-90, end angle=90, radius=0.5] --(-3,0.5);

\draw[green, thick] (3,-0.5) -- (1,-0.5) arc[start angle=270, end angle=90, radius=0.5] --(3,0.5);

  \filldraw[black] (1,0) circle (2pt);
   \filldraw[black] (-1,0) circle (2pt);
 

\end{tikzpicture}
        \caption{}\label{subfig:postslide}
    \end{subfigure}

      \caption{Neighborhoods of $\partial_1(\Sigma)$ --- shown in gray --- in $\mathcal{H}_0''(-L)$ (Figure~\ref{fig:handleslide1}) and in $\mathcal{H}_0(-L)$ (Figure~\ref{subfig:postslide}). Note that the left and right dashed edges of each subfigure are identified. The blue arrows indicate two handleslides we use to turn $\widetilde{\gamma}_1'$ into the necklace curve $\widetilde{\gamma}_1$. The handleslide corresponding to the short arrow is performed first.}  \label{fig:choker1}
\end{figure}
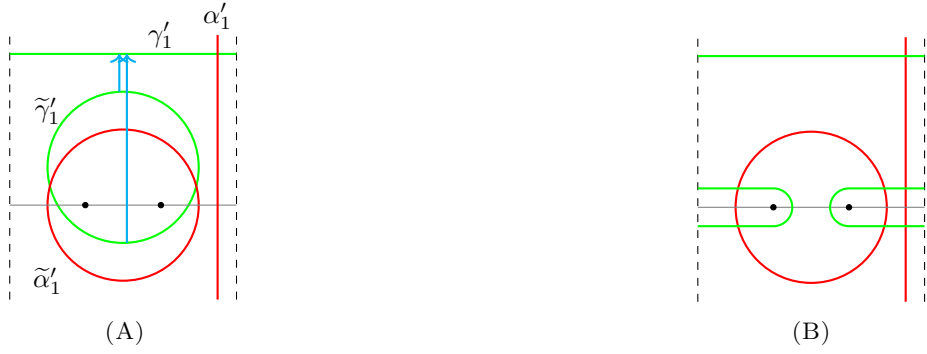

Consider a basis element $\bm{y}\in\T_{\bm{\alpha}}\cap\T_{\bm{\gamma}}$ with $\langle f(\bm{v}_\pm'),\bm{y}\rangle\neq 0$. To parse the notation $\langle-,-\rangle$, recall that if $\bm{a}$ is an element of $\T_{\bm{\alpha}}\cap\T_{\bm{\gamma}}$, then $f(\bm{a})$ is a linear combination of elements $\bm{y}\in\T_{\bm{\alpha}}\cap\T_{\bm{\gamma}}$ by definition. With this in mind, $\langle f(\bm{v}_\pm'),\bm{y}\rangle$ denotes the coefficient of $\bm{y}$ in $f(\bm{v}_\pm')$.
The convenient positioning of the basepoints near each pair of intersection points $c_i'$ for $i\geq 2$ or $\widetilde{c}_i'$ for $i\neq 1$ --- as in Figure~\ref{fig:standardtraingles}  --- implies that $c_2',c_3',\dots c_{N}',\widetilde{c}_3',\dots \widetilde{c}_{n}'\in\bm{y}$ --- see Figure~\ref{fig:standardtraingles}. In slightly more detail, since $f$ is the composition of maps induced by an operation --- either a handleslide or an isotopy --- we can write it as $f=f_1^*\circ f_2^*\circ\dots\circ f_l^*$ for some $l$. Each map $f_i^*$ is induced by a count of pseudo-holomorphic triangles in a Heegaard triple with underlying Heegaard diagram containing $\bm{\alpha}$, the pre-operation $\gamma$-curves and the post-operation $\gamma$-curves. Since the two types of $\gamma$-curves are almost identical in a neighborhood of the points playing the role of the $c_j'$ or $\widetilde{c}_j'$ intersection points --- as shown in Figure~\ref{fig:standardtraingles} --- the argument follows.
Moreover, the fact that $x$ is the unique element of $\gamma_1\cap\bm{\alpha}$ implies that $x\in\bm{y}$. Here we are abusing notation by allowing $x$ to also denote the unique element in $\gamma_1\cap\bm{\alpha}$; see Figure~\ref{fig:uniquetriangle}.

 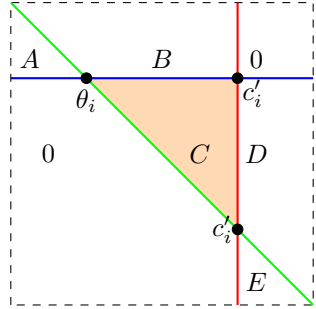
\begin{figure}[ht]
\centering

       \centering
\begin{tikzpicture}

\fill[orange!30] (1,0) -- (3,0) -- (3,-2) -- (1,0);

\draw[dashed] (0,-3) rectangle (4,1);

\draw[blue, thick] (0,0) -- (4,0);

\draw[red, thick] (3,1) -- (3,-3);

\draw[green, thick] (0,1) -- (4,-3);

     \filldraw[black] (1,0) circle (2pt);
    \node at (1,-0.3) {\( \theta_i \)};

    \filldraw[black] (3,0) circle (2pt);
    \node at (3.2,-0.2) {\( c_i' \)};

      \filldraw[black] (3,-2) circle (2pt);
    \node at (2.8,-2) {\( c_i' \)};


\node at (0.25,0.25) {\( A\)};

\node at (2,0.25) {\( B\)};

\node at (3.25,0.25) {\( 0\)};

\node at (0.5,-1) {\( 0\)};

\node at (2.5,-1) {\( C\)};

\node at (3.25,-1) {\( D\)};

\node at (3.25,-2.7) {\( E\)};

\end{tikzpicture}

    \caption{If we have domains of an admissible Heegaard diagram as shown here and $c_i'\in\bm{c}'$, $\theta_i\in\bm{\theta}$ then the shadow of any pseudo-holomorphic disk contributing to $f_{\bm{\theta}}(\bm{c}')$ has multiplicity $1$ in domain $C$ and $0$ in $A$, $B$ and $D$. Here $\bm{\theta}$ is a canonical intersection point in $\T_{\bm{\gamma}}\cap\T_{\bm{\gamma}'}$, see~\cite[Section 9]{ozsvath2004holomorphic} for details.
    }\label{fig:standardtraingles}
   \end{figure}

It remains to determine the intersection points on $\widetilde{\gamma}_1\cap \widetilde{\alpha}_1$ in $\bm{y}$. There are four such intersection points, two of which are ${v}_\pm$. Note that $\bm{v}_\pm$ are of the correct Alexander grading since there are pseudo-holomorphic triangles contributing to $f$ from $\bm{v}_\pm'\to\bm{v}_\pm$. The remaining two intersection points yield generators of the incorrect Alexander grading since they are connected to $\bm{v}_\pm$ by pseudo-holomorphic disks containing either $z_1$ or $w_1$, say by~\cite[Lemma 3.10]{HolomorphicdiskslinkinvariantsandthemultivariableAlexanderpolynomial}.  Note, moreover, that the only $\Z/2$-linear combinations of $\{v_\pm,x,c_2',c_3',\dots c_{N}',\widetilde{c}_3',\dots \widetilde{c}_{n}'\}$ are $\bm{v}_\pm$ and $\bm{v}_++\bm{v}_-$. It then follows that $\bm{v}_\pm$ are generators, as desired. 
\end{proof}

\subsection{Analysing the exact triangle}\label{subsec:exacttriangle}
 We can now proceed to the proof of Proposition~\ref{thm:HFKaddtwists}. We want to define a Heegaard triple diagram $\mathcal{H}_{\bm{\alpha},\bm{\beta},\bm{\gamma}}$ for computing the map $f_1^*$ from the exact triangle in Equation~(\ref{eq:exactneg}). Recall from the previous subsection that we already have Heegaard diagrams $\mathcal{H}$ for $(-{Y},-{L_\phi})$, and $\mathcal{H}_0(-{L})$ for $(-{Y}_0(\eta),-{L})$ with homeomorphic Heegaard surfaces, which we are denoting $S$. Taking $S$ as our Heegaard surface for $\mathcal{H}_{\bm{\alpha},\bm{\beta},\bm{\gamma}}$, we can let $\bm{\alpha}$ denote the $\alpha$-curves for both $(-Y_0(\eta),-L)$ and $(-Y_{\phi},-L)$. Likewise, $\beta$-curves will be curves in the Heegaard diagram $\mathcal{H}:(S,\bm{\beta},\bm{\alpha},\bm{w},\bm{z})$ for $(-Y_{\phi},-L)$, where $(S,\bm{\alpha},\bm{\beta},\bm{z},\bm{w})$ is adapted to $(\Sigma,\phi)$ and $\gamma$ will be the (non-alpha) curves in the Heegaard diagram $\mathcal{H}_0(-{L})$ for $(-{Y}_0(\eta),-L)$,  as in the preceding subsection. Let $c_i$ be the intersection point in $\alpha_i\cap\beta_i$ and $\widetilde{c}_j$ be the intersection point in $\widetilde{\alpha}_j\cap\widetilde{\beta}_j$ contributing to the BRAID invariant of $L$. Let $u_\pm$ indicate the two points labeled as such in Figure~\ref{fig:uniquetriangle}. Let $x$ denote the unique intersection point on $\alpha_1\cap\gamma_1$ from the Heegaard diagram $\mathcal{H}_0(-{L})$.

If $\phi$ sends an arc $a_1$ with exactly one intersection point with $\partial_1\Sigma$ weakly to the left then this will be witnessed by an intersection point $d_1\in\alpha_1\cap\beta_1\cap-{\Sigma}$. Note that that this is true even if $a_1$ is fixed by $\phi$ since when constructing our Heegaard diagram $\beta_1$ is obtained from an appropriate perturbation of $\phi(a_1)$. See Figure~\ref{fig:uniquetriangle}, where the arc appears to be going to the right since we are working in $-{\Sigma}$ rather than $\Sigma$. Consequently we have generators $\bm{u}_\pm:=\{u_\pm,d_1,c_2,\dots c_{N},\widetilde{c}_3,\dots\widetilde{c}_{n}\}$ of $\mathcal{H}(-L)$. The generators $\bm{v}_{\pm}=\{v_\pm,x,c_2',\dots c_{N}',\widetilde{c}_3',\dots\widetilde{c}_{n}'\}$ are the same as they were in the previous section; see Figure~\ref{fig:cleanedsurgery}. 

    \begin{figure}[ht]
\centering
\begin{tikzpicture}

\fill[orange!30]
  (0.5,9) arc[start angle=180, end angle=90, radius=0.5] -- 
  (5,9.5) -- 
  (5,9) -- 
  (0.5,9) -- 
  cycle;

\fill[orange!30]
 (7,2) --
  (7,3) --  (9,3) --(9,2.5) arc[start angle=0, end angle=-90, radius=0.5] -- 
  (8.5,2) -- 
  (7,2) -- 
  cycle;

\draw[dashed] (0,0) rectangle (10,10);

\draw[red, thick] (5,0) -- (5,10);

\draw[green, thick] (0,9) -- (10,9);


 \draw[blue, thick] (6.5,0) arc[start angle=0, end angle=90, radius=0.5];
 \draw[blue, thick] (6,0.5) -- (1,0.5);
  \draw[blue, thick] (1,0.5) arc[start angle=270, end angle=180, radius=0.5];
 \draw[blue, thick] (0.5,1) -- (0.5,9);
   \draw[blue, thick] (0.5,9) arc[start angle=180, end angle=90, radius=0.5];
 
  \draw[blue, thick] (1,0.5) arc[start angle=270, end angle=180, radius=0.5];
  \draw[blue, thick] (6,9.5) --  (1,9.5);
 \draw[blue, thick] (6,9.5) arc[start angle=270, end angle=360, radius=0.5];


\draw[red, thick] (0,5) -- (3,5);
 \draw[red, thick] (3,5) arc[start angle=90, end angle=-90, radius=1];
\draw[red, thick] (0,3) -- (3,3);


\draw[red, thick] (10,5) -- (7,5);
 \draw[red, thick] (7,5) arc[start angle=90, end angle=270, radius=1];
\draw[red, thick] (10,3) -- (7,3);


\draw[green, thick]  (1.5,2) -- (8.5,2);

\draw[green, thick]  (8.5,2) arc[start angle=270, end angle=360, radius=0.5];

\draw[green, thick]  (9,2.5) -- (9,7.5);

\draw[green, thick]  (9,7.5) arc[start angle=0, end angle=90, radius=0.5];
\draw[green, thick]  (1.5,8) -- (8.5,8);

\draw[green, thick]  (1.5,8) arc[start angle=90, end angle=180, radius=0.5];

\draw[green, thick]  (1,7.5) -- (1,2.5);

\draw[green, thick]  (1,2.5) arc[start angle=180, end angle=270, radius=0.5];


\draw[blue, thick]  (3,2) arc[start angle=180, end angle=270, radius=1];
\draw[blue, thick]  (4,1) -- (6,1);
\draw[blue, thick]  (6,1) arc[start angle=270, end angle=360, radius=1];

\draw[blue, thick]  (7,6) -- (7,2);

\draw[blue, thick]  (7,6) arc[start angle=0, end angle=90, radius=1];
\draw[blue, thick]  (4,7) -- (6,7);

\draw[blue, thick]  (4,7) arc[start angle=90, end angle=180, radius=1];
\draw[blue, thick]  (3,6) -- (3,2);

  \filldraw[black] (5,0.5) circle (2pt);
    \node at (4.8,0.3) {\( {c}_1 \)};

      \filldraw[black] (5,9.5) circle (2pt);
    \node at (5.2,9.3) {\( {d}_1 \)};

          \filldraw[black] (5,9) circle (2pt);
    \node at (5.2,8.7) {\( x \)};

          \filldraw[black] (0.5,9) circle (2pt);
    \node at (0.7,8.7) {\( \theta_1 \)};

          \filldraw[black] (7,2) circle (2pt);
    \node at (7.2,1.7) {\( \widetilde{\theta}_1 \)};

      \filldraw[black] (9,3) circle (2pt);
    \node at (9.3,2.8) {\( v_+ \)};
      \filldraw[black] (1,3) circle (2pt);
    \node at (1.3,2.8) {\( v_- \)};

      \filldraw[black] (7,3) circle (2pt);
    \node at (7.3,2.8) {\( u_+ \)};
      \filldraw[black] (3,3) circle (2pt);
    \node at (3.3,2.8) {\( u_- \)};


       \node at (4.5,9.8) {\( A\)};

 \node at (5.5,9.8) {\( B\)};

    \node at (7.5,9.8) {\( A\)};

   \node at (0.2,8.5) {\( C\)};
   \node at (3.5,9.2) {\( E\)};
   \node at (7.5,8.5) {\( C\)};
    
    \node at (2.5,8.5) {\( D\)};

       \node at (4.5,7.5) {\( L\)};

 \node at (5.5,7.5) {\( M\)};

     \node at (0.2,4) {\( H\)};
     \node at (2.5,4) {\( G\)};
    \node at (3.5,4) {\( w_1\)};

   \node at (6.5,4) {\( z_1\)};

 \node at (4.5,4) {\( I\)};

 \node at (5.5,4) {\( J\)};
     \node at (7.5,4) {\( N\)};
      \node at (9.5,4) {\( H\)};

   \node at (2.5,2.2) {\( F\)};
    \node at (7.5,2.2) {\( P\)};

         \node at (2.5,1.2) {\( K\)};

          \node at (4.5,1.2) {\( Q\)};

 \node at (5.5,1.2) {\( R\)};

  \node at (0.4,0.4) {\( 0\)};
  \node at (9.4,0.4) {\( 0\)};


    \node at (9,9.1) {\( \gamma_1\)};
        \node at (9,8.1) {\(\widetilde{ \gamma}_1\)};

         \node at (2.7,6.6) {\(\widetilde{ \beta}_1\)};
            \node at (2,5.4) {\(\widetilde{ \alpha}_1\)};

              \node at (4.5,5.4) {\({ \alpha}_1\)};
              
              \node at (6.6,0.5) {\({ \beta}_1\)};

\end{tikzpicture}
          \caption{A neighborhood of $\partial_1\Sigma$. The orange domains show the shadow of a pseudoholomorphic disk contributing to the map $f_1^*$ in the proof of Proposition~\ref{thm:HFKaddtwists}. The lower left and lower right domains have multiplicity zero because they contain the basepoint $z_2$. }\label{fig:uniquetriangle}
   \end{figure}
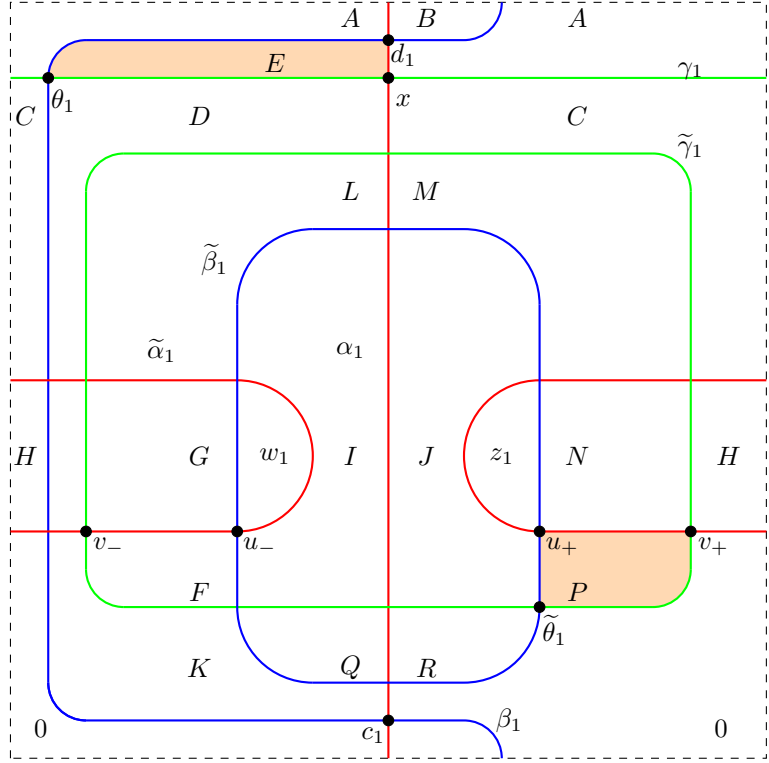

   The proof of Proposition~\ref{thm:HFKaddtwists1} is substantially easier than that of Proposition~\ref{thm:HFKaddtwists}, so we give its proof first. For this we use the following Lemma.

\begin{lemma}\label{lem:f1nontrivial}
    Suppose $\phi$ sends an arc with exactly one endpoint abutting to $\partial_1\Sigma$ weakly to the left. Then $f_1^*([\bm{u}_+]) =[\bm{v}_+]$ so that, in particular, $f_1^*$ is non-trivial.
\end{lemma}

\begin{proof}

Consider the Heegaard diagrams $\mathcal{H}=(S,\bm{\beta},\bm{\alpha},\bm{z},\bm{w})$ for $(-{Y},-{L})$ and $\mathcal{H}_0(-{L})=(S,\bm{\gamma},\bm{\alpha},\bm{z},\bm{w})$ for $(-{Y}_{0}(\eta),-{L})$. Recall that $f_1^*$ is defined by a count of appropriate pseudo-holomorphic triangles in $\Sym^g(S)$ with boundary contained in the Lagrangian submanifolds $\T_\alpha\cup\T_\beta\cup\T_\beta$, where $\alpha$, $\beta$ and $\gamma$ are viewed as (appropriate perturbations) of curves  $S$. In particular, the shadows of these triangles in $S$ are required to have vertices on $\theta_i$ for $i\geq 1$, $\widetilde{\theta_j}$ for $j\geq 3$, and  $\widetilde{\theta_1}$ --- the unique intersection point between $\widetilde{\alpha_1}$ and $\widetilde{\gamma}_1$.  We first need to check that $\mathcal{H}_{\bm{\gamma},\bm{\beta},\bm{\alpha}}$ is weakly admissible. This can be seen as follows. First note that each $\alpha$-curve has a $z$ basepoint on both sides. Ignoring the $\alpha$-curves each $\beta$-curve has a $z$ basepoint on either side. Finally then each $\gamma$ curve has a basepoint on either side.

Returning to the main proof, we observe that $\bm{u}_+$ is a cycle in $\widehat{\CFL}(\mathcal{H})$. To see this, first observe that any outgoing pseudoholomorphic disk from $\bm{u}_+$ must be constant on every intersection point $c_j$ and $\widetilde{c}_k$ for all $j$ and $k$, since two of the adjacent elementary domains have multiplicity zero. This reduces computing $\partial{\bm{u}_+}$ to counting bigons formed by $\widetilde{\beta}_1$ and $\widetilde{\alpha}_1$ with a vertex at $u_+$ but we can see that there are no such bigons by inspection of Figure~\ref{fig:uniquetriangle}, where we may ignore the $\gamma$-curves.

We now claim that $f_1^*(\bm{u}_+)=\bm{v}_+$. Consider the local behaviour of the shadow of pseudoholomorphic disk $\psi$ contributing to $f_1$ near $p\in\bm{u}_+$ with $p\neq d_1$. Due to the positions of the basepoints near $p$, we have a local picture as in Figure~\ref{fig:standardtraingles}. The positioning of the basepoints implies that the multiplicity in region $C$ is one, while the multiplicities in region $A$, $B$ $D$ and $E$ are $0$ in Figure~\ref{fig:standardtraingles}. This is the same argument as used in the proof of~\cite[Proposition 3.7]{honda2006contact}. This determines the behavior of $\partial\psi$ on all $\alpha, \beta$ and $\gamma$-curves other than $\alpha_1,\beta_1,\gamma_1$. We now check that the shadow of $\psi$ in $S$, in a neighborhood of $\partial_1(\Sigma)$, must be the regions $P$ and $E$ shown in Figure~\ref{fig:uniquetriangle}. Indeed, $P$ must be in the shadow of $\psi$ by the argument above. Consequently, the multiplicity of every region except for $A$, $B$, $E$ and $P$ in Figure~\ref{fig:uniquetriangle} must be zero. Consider the remaining component of the shadow of $\psi$. It must have a corner on $x$ since $x$ is the unique intersection point of $\gamma_1$ and any $\alpha$ curve. Note that $\partial \psi$ must go vertically downwards from $d_1$ to $x$ and immediately stop at $x$ since otherwise it would pass through regions with multiplicity zero on both sides further down in the diagram. Observe that there is an arc connecting $A$ and a region with multiplicity zero that doesn't intersect any boundary components of $\psi$ by leaving the top of Figure~\ref{fig:uniquetriangle} and coming back around the bottom. So the multiplicity of $A$ is zero, and thus so is $B$. Thus we conclude that the shadow of $\psi$ contains $E$. This determines the homotopy class of $\psi$. This homotopy class has a unique holomorphic representative by the Riemann mapping theorem, concluding the proof.\end{proof}
We can now prove the remaining proposition advertised at the start of this section.

\begin{proof}[Proof of Proposition~\ref{thm:HFKaddtwists1}]
      Since $f_1^*$ is non-trivial by Lemma~\ref{lem:f1nontrivial} and ${\rank(\widehat{\HFL}(-{Y}_0,-{L},{-[\Sigma]},1-G))=2}$ by Lemma~\ref{lem:0surgery}, it follows from the exact triangle~(\ref{eq:exactneg}) 
 that $\rank(\widehat{\HFL}(-{Y},-{L_\phi}, {-[\Sigma]},1-G))$ is either $\rank(\widehat{\HFL}((-{Y})_1,-{L_{\Delta_1\circ\phi}}, {-[\Sigma]},-G))$
or 
$\rank(\widehat{\HFL}((-{Y})_1,-{L_{\Delta_1\circ\phi}}, {-[\Sigma]},-G))+2$. The result follows by applications of symmetry properties of $\widehat{\HFL}$; see Equation~\ref{eq:invariance}.
\end{proof}

 Before proceeding to the proof of Proposition~\ref{thm:HFKaddtwists}, we require the following Lemma.

   \begin{lemma}
       Suppose that $L$ is a fibered link in $Y$ with monodromy $\phi$, where $\phi\circ\Delta_1^2$ is weakly left-veering on $\partial_1\Sigma$. There is a minimal position Heegaard diagram for $\overline{a}_1$ and $\overline{b}_1$ adapted to $L$ such that $\overline{a}_1$ and $\overline{b}_1$ intersect at least three times in an annular neighborhood of $\partial \Sigma_1$.
   \end{lemma}
Here by a minimal position Heegaard diagram for $\overline{a}_1$ and $\overline{b}_1$ adapted to $L$, we mean a Heegaard diagram as in Definition~\ref{def:HDadaptredtoL}, such that $\overline{a}_1$ and $\overline{b}_1$ do not cobound a bigon. See Figure~\ref{fig:} for an example of such a Heegaard diagram. Observe that $\overline{b}_1$ --- shown in blue --- appears to be sent to the right --- but it actually sent to the left, since the Figure shows $-\Sigma$.
   \begin{proof}
       Consider a minimal position diagram for the fibered link with monodromy $\Delta_1^2\circ\phi:\Sigma\to\Sigma$. Since $\overline{a}_1\cap\partial_1\Sigma$ is to the left of $\overline{b}_1\cap \partial_1\Sigma$, and $\overline{b}_1$ is weakly to the left of $\overline{a}_1$, there is at least one intersection point between $\overline{a}_1$ and $\overline{b}_1$ in some annular neighborhood of $\partial_1\Sigma$. Observe that one can obtain a minimal position diagram for $\phi$ by adding two left handed Dehn twists to $\beta_1$ parallel to $\partial_1\Sigma$. This adds the desired pair of intersection points.
   \end{proof}
   
   It follows from the Lemma that if $\Delta_1^2\circ \phi$ sends $a_1$ weakly to the left then there are second and third intersection points $d_2,d_3\in\alpha_1\cap\beta_1$ directly above $d_1$ in $\mathcal{H}$, as shown in Figure~\ref{fig:}. Set $\bm{u}^3_-:=\{u_-,d_3,c_2,\dots, c_N,\widetilde{c}_3,\widetilde{c}_4,\dots,\widetilde{c}_n\}$.

       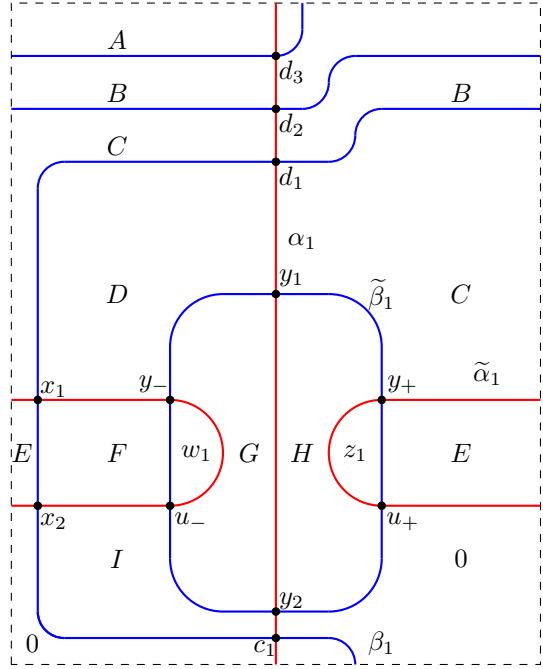
\begin{figure}[ht]
\centering
   
\begin{tikzpicture}[scale=0.7]

\draw[dashed] (0,0) rectangle (10,12.5);

  \draw[red, thick] (5,0) -- (5,12.5);


 \draw[blue, thick] (6.5,0) arc[start angle=0, end angle=90, radius=0.5];
 \draw[blue, thick] (6,0.5) -- (1,0.5);
  \draw[blue, thick] (1,0.5) arc[start angle=270, end angle=180, radius=0.5];
 \draw[blue, thick] (0.5,1) -- (0.5,9);
   \draw[blue, thick] (0.5,9) arc[start angle=180, end angle=90, radius=0.5];
 
  \draw[blue, thick] (1,0.5) arc[start angle=270, end angle=180, radius=0.5];
  \draw[blue, thick] (6,9.5) -- (1,9.5);
 \draw[blue, thick] (6,9.5) arc[start angle=270, end angle=360, radius=0.5];
 \draw[blue, thick] (6.5,10) arc[start angle=180, end angle=90, radius=0.5];
  \draw[blue, thick] (7,10.5) -- (10,10.5);
    \draw[blue, thick] (0,10.5) -- (5.5,10.5);
     \draw[blue, thick] (5.5,10.5) arc[start angle=270, end angle=360, radius=0.5];
         \draw[blue, thick] (6,11) arc[start angle=180, end angle=90, radius=0.5];
           \draw[blue, thick] (6.5,11.5) -- (10,11.5);
             \draw[blue, thick] (0,11.5) -- (5,11.5);
                \draw[blue, thick] (5,11.5) arc[start angle=270, end angle=360, radius=0.5];

                 \draw[blue, thick] (5.5,12) -- (5.5,12.5);


\draw[red, thick] (0,5) -- (3,5);
 \draw[red, thick] (3,5) arc[start angle=90, end angle=-90, radius=1];
\draw[red, thick] (0,3) -- (3,3);


\draw[red, thick] (10,5) -- (7,5);
 \draw[red, thick] (7,5) arc[start angle=90, end angle=270, radius=1];
\draw[red, thick] (10,3) -- (7,3);


\draw[blue, thick]  (3,2) arc[start angle=180, end angle=270, radius=1];
\draw[blue, thick]  (4,1) -- (6,1);
\draw[blue, thick]  (6,1) arc[start angle=270, end angle=360, radius=1];

\draw[blue, thick]  (7,6) -- (7,2);

\draw[blue, thick]  (7,6) arc[start angle=0, end angle=90, radius=1];
\draw[blue, thick]  (4,7) -- (6,7);

\draw[blue, thick]  (4,7) arc[start angle=90, end angle=180, radius=1];
\draw[blue, thick]  (3,6) -- (3,2);

  \filldraw[black] (5,0.5) circle (2pt);
    \node at (4.8,0.3) {\( {c}_1 \)};

      \filldraw[black] (5,9.5) circle (2pt);
    \node at (5.3,9.2) {\( {d}_1 \)};
    
          \filldraw[black] (5,10.5) circle (2pt);
    \node at (5.3,10.2) {\( {d}_2 \)};

      \filldraw[black] (5,11.5) circle (2pt);
    \node at (5.3,11.2) {\( {d}_3 \)};

  \filldraw[black] (3,5) circle (2pt);
    \node at (2.7,5.3) {\(y_- \)};
    
     \filldraw[black] (7,5) circle (2pt);
    \node at (7.4,5.3) {\(y_+ \)};

      \filldraw[black] (0.5,3) circle (2pt);
    \node at (0.8,2.7) {\( x_2 \)};
     \filldraw[black] (0.5,5) circle (2pt);
    \node at (0.8,5.2) {\( x_1 \)};

      \filldraw[black] (5,1) circle (2pt);
    \node at (5.3,1.2) {\( y_2 \)};
     \filldraw[black] (5,7) circle (2pt);
    \node at (5.3,7.3) {\( y_1 \)};

      \filldraw[black] (3,3) circle (2pt);
    \node at (3.4,2.7) {\( u_- \)};

 \filldraw[black] (7,3) circle (2pt);
    \node at (7.4,2.7) {\( u_+\)};


       \node at (2,11.8) {\( A\)};

       \node at (2,10.8) {\( B\)};
         
         \node at (2,9.8) {\( C\)};

 \node at (8.5,10.8) {\( B\)};

 \node at (2,7) {\( D\)};
 
    \node at (8.5,7) {\( C\)};

     \node at (0.2,4) {\( E\)};
      \node at (2,4) {\( F\)};
   
    \node at (3.5,4) {\( w_1\)};

   \node at (6.5,4) {\( z_1\)};

 \node at (4.5,4) {\( G\)};

 \node at (5.5,4) {\( H\)};
     \node at (8.5,4) {\( E\)};

      \node at (2,2) {\( I\)};

  \node at (0.4,0.4) {\( 0\)};
  \node at (8.5,2) {\( 0\)};


    \node at (7,0.4) {\( \beta_1\)};

     \node at (7,7) {\( \widetilde{\beta}_1\)};

         \node at (5.5,8) {\( {\alpha}_1\)};

               \node at (9,5.5) {\( \widetilde{\alpha}_1\)};
   
\end{tikzpicture}
    \caption{A neighborhood of $\partial_1\Sigma$ in the Heegaard diagram adapted to $(\Sigma,\phi)$ when $\Delta_1^{2}\circ\phi$ sends the arc $a_1$ weakly to the left. Here $\Delta_1$ is a right-handed Dehn twist. The multiplicities of the bottom left and right domains are both zero because they contain the basepoint $w_2$.}
    \label{fig:}
   \end{figure}
\begin{lemma}\label{lem:cycle}
    Let $\phi:\Sigma\to\Sigma$ be a diffeomorphism of a surface with $|\partial\Sigma|>1$ such that $\phi\circ\Delta_1^2$ sends an arc with exactly one endpoint on $\partial_\Sigma$ weakly to the left at that endpoint. The generator $\bm{u}_-^3$ is a cycle.
\end{lemma}

For the proof we determine the homology classes of disks $\psi$ contributing to the differential. We use the following key facts. Firstly, the homology class of a disk is determined by its multiplicity in each of the elementary domains of the Heegaard diagram --- i.e. connected components of $\Sigma\setminus(\bm{\alpha}\cup\bm{\beta})$. Moreover, if $D$ is the domain of a curve $\psi\in\pi_2(\bm{x},\bm{y})$ then $\partial D$ is an oriented multi-curve which connects $\bm{x}$ to $\bm{y}$ that passes from elements of $\bm{x}$ to elements of $\bm{y}$ along $\beta$-curves and from elements of $\bm{y}$ to elements of $\bm{x}$ along $\alpha$ curves. In particular, following $\partial D$ from an element of $\bm{x}$ to an element of $\bm{y}$ along a $\beta$ curve, we must have larger multiplicity on the left than the right, and following $\partial D$ along an $\alpha$ curve from an element of $\bm{y}$ to an element of $\bm{x}$ we must have larger multiplicity on the left than the right. Likewise the multiplicities of two adjacent domains agree unless they are separated by a component of $\partial D$. We also use that the domain of any pseudo-holomorphic disk has non-negative multiplicities. If $\partial \psi$ is constant on some point $x\in\alpha\cap{\beta}$ then the multiplicities of domains separated by an arc in ${\alpha}$ or ${\beta}$ are the same. Finally, if $x\in\bm{x}\cap\bm{y}$ then the multiplicities of domains separated by an arc in ${\alpha}\cup{\beta}$ are the same, where $\alpha$ and $\beta$ are the curves with $x\in\alpha\cap\beta$. This follows from the fact that each $\alpha$ and $\beta$-curve have basepoints on either side; see Figure~\ref{fig:standardtraingles}, where one may ignore the $\gamma$-curve.

\begin{proof}
    Let $\bm{y}$ be an element of $\T_{\bm{\alpha}}\cap\T_{\bm{\beta}}$ such that $\langle\partial \bm{u}_-^3,\bm{y}\rangle\neq 0$. Let $\psi$ be a pseudo-holomorphic disk contributing to the differential.  Consider the positioning of the basepoints near $c_i$ for $i\neq 1$, and $\widetilde{c}_i$ for $i>2$. We readily see that $c_i\in\bm{y}$ for $i\neq 1$, and $\widetilde{c}_i\in\bm{y}$ for $i>2$.

    It thus suffices to determine the behaviour of $\psi$ in an appropriate neighborhood of $\alpha_1\cup\beta_1\cup\widetilde{\alpha}_1\cup\widetilde{\beta}_1$. The relevant portion of the Heegaard diagram is contained in a neighborhood of $\partial_1\Sigma$, which we have shown in Figure~\ref{fig:}. We proceed by cases. Either 
    \begin{enumerate}
        \item 
        
        $\bm{y}$ contains elements in $\alpha_1\cap\widetilde{\beta}_1$ and $\widetilde{\alpha}_1\cap\beta_1$, or
        
       \item $\bm{y}$ contains elements in $\alpha_1\cap\beta_1$ and $\widetilde{\alpha}_1\cap\widetilde{\beta}_1$. 
    \end{enumerate}

\textbf{Case 1.} Suppose $\bm{y}$ contains elements in $\alpha_1\cap\widetilde{\beta}_1$ and $\widetilde{\alpha}_1\cap\beta_1$. There are exactly two pairs of intersection points that yield a generator of the correct Alexander grading per~\cite[Lemma 4.2]{tovstopyat2018transverse}, as discussed in Section~\ref{sec:HFLintro}; we have that either $x_1,y_2\in\bm{y}$ or $x_2,y_1\in\bm{y}$, where $x_1,x_2\in\widetilde{\alpha}_1\cap\beta_1$ and $y_1,y_2\in\widetilde{\beta}_1\cap\alpha_1$ are as shown in Figure~\ref{fig:}.

Suppose $x_1,y_2\in\bm{y}$. Consider the behavior of $\partial\psi$ near $\beta_1$. Observe that $\partial\psi$ must exit from $d_3$ to the left --- as viewed in Figure~\ref{fig:} --- and pass to $x_1$, stopping there on the first time of passing. We can then see that $A=0,B=1$, $C=2$, $D=3$, $E=F$ and $I=0$. Consider the behavior of $\partial\psi$ near $\widetilde{\alpha}_1$. The positioning of basepoint $z_1$ and $w_1$ imply that $\partial\psi$ must exit $x_1$ to the right, as shown in Figure~\ref{fig:}, and stop at $u_-$ at the first time of passing. It follows that $E=0$, and $F=I$, so $E=F=I=0$. Consider the behavior of $\partial\psi$ near $\alpha_1$. Observe that $\partial\psi$ must exit upward from $y_2$ --- as viewed in Figure~\ref{fig:} --- and pass to $d_3$, stopping there on the first time of passing, so we see that $G=H+1$. Now consider the behavior of $\partial\psi$ near $\widetilde{\beta}_1$. Observe that $\partial\psi$ must exit from $u_-$ and travel counterclockwise to $y_2$, stopping at the first time of passing. Thus we see that $G=1$, but also that $G=D=3$, a contradiction.

Suppose $x_2,y_1\in\bm{y}$.  Consider the behavior of $\partial\psi$ near $\beta_1$. Observe that $\partial\psi$ must exit from $d_3$ to the left --- as viewed in Figure~\ref{fig:} --- and pass to $x_2$, stopping at the first opportunity. We thus see that $A=0,B=1$, $C=2$, $D=3$, $F=E+1$ and $I=0$. Consider the behavior of $\partial\psi$ near $\widetilde{\alpha}_1$. Observe that $\partial\psi$ must exit $x_2$ to the right --- as viewed in Figure~\ref{fig:} ---  stopping at $u_-$ on the first opportunity. We then see that $F=1$ and $E=C=0$, a contradiction. 

 \textbf{Case 2.}   Suppose $\bm{y}$ contains an element $z$ in $\alpha_1\cap\beta_1$ and an element $w$ in $\widetilde{\alpha}_1\cap\widetilde{\beta}_1$. We have two subcases; either the intersection point $z$ is in ${\Sigma}$ (Case 2.a)) or it is not (Case 2.b)).

  \textbf{Case 2.a)}  We turn now to the case $z\in\Sigma$; i.e. $z=c_1$. In turn we have that $w=y_-$ (Case 2.a)(i)) or $y_+$ (Case 2.b)(ii)), the intersection points as shown in Figure~\ref{fig:}.
    
   \textbf{Case 2.a) (i)}  Suppose $w=y_-$. Consider the behavior of $\partial\psi$ near $\widetilde{\alpha}_1$. Observe that $\partial\psi$ must exit $y_-$ to the right along $\widetilde{\alpha}_{1}$ and stop at $u_-$ on the first time of passing. In particular it follows that $E=C=0$. Consider in turn the behavior of $\partial\psi$ near $\beta_1$. Observe that $\partial\psi$ must exit $d_3$ to the left and stop at $c_1$ on the first time of passing. In particular we see that $C=B+1$, so that $B=-1$, a contradiction.
    
      \textbf{Case 2.a) (ii)} Suppose $w=y_+$. Consider the behaviour of $\partial\psi$ in a neighborhood of $\beta_1$. Observe that $\partial\psi$ must exit $d_3$ to the left along $\beta_1$ and stop at $c_1$ at the first time of passing. In particular $A=0$, $B=1$, $C=2$, $D=3$ and $F=E+1$, $I=1$. Consider the behavior of $\partial\psi$ in a neighborhood of $\alpha_1$. Observe that $\partial\psi$ must exit $c_1$ upwards along $\alpha_1$ and stop at $d_3$ at the first time of passing. In particular $G=H+1$. Consider the behaviour of $\partial\psi$ in a neighborhood of $\widetilde{\alpha}_1$. Observe that $\partial\psi$ must exit $y_+$ to the right and stop at $u_-$ at the first time of passing. From this it follows that $E=0$ and $F=D-1$ and $F=I$. But $D=3$ and $I=1$, so this is a contradiction. Having dealt with the final case, we have concluded the proof of Case 2.a.
 
  \textbf{Case 2.b)} If $z$ is in $-{\Sigma}$ (the upper half of the Heegaard diagram as viewed in, say, Figure~\ref{fig:}) then $w\in\Sigma$ as else $\bm{y}$ is not of the correct Alexander grading, again by~\cite[Lemma 4.2]{tovstopyat2018transverse}. Thus $w=u_\pm$. We first claim that the $w=u_-$ case is impossible. To see this, observe that $u_-\in \bm{y}\cap\bm{u}_-^3$, and the positioning of the basepoints in neighborhoods of $\widetilde{\alpha}_1$ and $\widetilde{\beta}_1$ would imply that $\psi$ is constant on $u_-$ and consequently that the differential would be counting bigons between $\alpha_1$ and $\beta_1$, which do not occur since we have taken our curves to be in minimal position.
  
  We now claim that $d_3\not\in \bm{y}$. Suppose otherwise. Then the positioning of the basepoints near $c_1$ imply that $\psi$ is constant on $d_3$, and the differential counts bigons cobounded by $\widetilde{\alpha}_1$ and $\widetilde{\beta}_1$ which do not contain basepoints. There are no such bigons, so we have a contradiction.
    
  Traveling along $\partial\psi$ in the reverse of the direction dictated by the orientation of $\psi$, observe that $-\partial\psi$ can exit $u_-$ to the right or to the left along $\widetilde{\alpha}_1$, but in either case $-\partial\psi$ must stop at $u_+$ on the first time of passing. From the positioning of the basepoints in the regions adjacent to $u_+$, we see that in the first case $-\partial\psi$ must exit $u_+$ downwards along $\widetilde{\beta}_1$ and stop at $u_-$ on the first time of passing. Similarly, in the second case, $-\partial\psi$ must exit $u_+$ upwards and stop at $u_-$ on the first time of passing. Consequently we have that the homology classes of disks contributing to $\langle\partial \bm{u}_-^3,\bm{y}\rangle$ come in pairs; we either have that $G=H=1$ and $E=F=0$ or $E=F=1$ and $G=H=0$. In either case we have that $I=0$ and $D=C=1$.

  We have two more sub-cases according to whether $\partial \psi$ exits $d_3$ along $\beta_1$ to the left or to the right, as viewed in Figure~\ref{fig:}.
  
   \textbf{Case 2.b) (i)}  If $\partial \psi$ exits $d_3$ to the left along $\beta_1$ --- as viewed in Figure~\ref{fig:} --- then it must terminate at $d_1$ or $d_2$ and return directly to $d_3$ along $\alpha_1$. If it terminates at $d_1$ then we have that $A=0$, $B=1$ and $C=D=2$, a contradiction.
   
   Thus it terminates at $d_2$ and $A=0$, while $B=C=1$. In particular, $\pi_2(\bm{u_-^3},\bm{y})$ is non-empty; it contains two homotopy classes of curves that correspond to annuli in $S\times S$ that project to disks in $\Sym^2(S)$ under the group action that interchanges the coordinates. These two classes of annuli are lifts of the two immersed annuli we can see in $S$  --- namely the domains $\{B,C,D,G,H\}$ and $\{B,C,D,E,F\}$ --- to $S\times S$ under the projection map from $S\times S$ to the first coordinate.  See~\cite[Lemma 3.6]{ozsvath2004holomorphic} and~\cite[Lemma 4.1]{lipshitz2006cylindrical} for more general versions of this argument. 

   Observe that there is a $(-1,1)$-parameterized family of complex structures on these annuli, with complex structures inherited from $S$. The parameter $t\in(-1,1)$ measures the portion of a cut from $d_2$ along  $\beta_1$ to the relevant point on $\widetilde{\alpha}_1$ in the case that $t\leq 0$ or a cut from $d_2$ along  $\alpha_1$ to the relevant point on $\widetilde{\beta}_1$ in the case that $t\geq 0$. Observe also that these cuts are from the obtuse right angled corner at $d_2$ along non-separating curves.

   Annuli that double branched cover disks that contribute to the differential are exactly those that admit an involution which: 1) interchanges the boundary components and 2) satisfies that the conformal angle swept out by the $\alpha$ segments of each boundary component agree~\cite[Proof of Lemma 9.3]{ozsvath2004holomorphic}. If there exists such an involution then there is exactly one pseudo-holomorphic representative. Ozsv\'ath and Szab\'o showed that there are an odd number of values of the parameter $t$ above for which this condition holds~\cite[Proof of Lemma 3.4]{OSHoldisks3propapp}. Thus the contributions of pseudo-holomorphic curves with domains given by this pair of homotopy classes of annuli---namely those with domains $\{B,C,D,G,H\}$ and $\{B,C,D,E,F\}$---cancel.

       \textbf{Case 2.b) (ii)}  If $\partial \psi$ exits $d_3$ to the right along $\beta_1$, as viewed in Figure~\ref{fig:}, then wherever it terminates it must do so before arriving in a neighborhood of $c_1$ due to the positioning of the basepoints. Likewise $\partial \psi$ must arrive at $d_3$ along $\alpha_1$ from above, as viewed in Figure~\ref{fig:}. In particular it follows that the multiplicities of $A, B,C$ and $D$ are one, while the elementary domain northeast of $d_3$ is multiplicity $2$. Recall that in Lipshitz's cylindrical reformulation of Heegaard Floer homology, counts of pseudo-homolomorphic disks are equivalent to counts of pseudo-holomorphic curves satisfying certain conditions~\cite{lipshitz2006cylindrical}. Let $u:R\to S\times\R\times[0,1]$ be such a curve with domain satisfying the conditions we have deduced thus far. Recalling that our pseudo-holomorphic curves are constant on $N-2$ of the $c_i$ intersection points, Lipshitz's local formula for the Maslov index~\cite[Equation 5, Equation 6]{lipshitz2006cylindrical} implies that \begin{align*}
    1&=N-\chi(R)-(N-2)+2e(D)\\&=2+\chi(R)-\frac{k}{2}+\frac{l}{2}\\&=4-2g(R)-|\partial R|-\frac{k}{2}+\frac{l}{2}.
       \end{align*}

   \noindent Here $k$ is the number of acute right angled corners, $l$ is the number of obtuse right angled corners, $N$ is the genus of $S$, $D$ is the domain of the holomorphic curve and $e(D)$ is its Euler measure. See~\cite[Section 4]{lipshitz2006cylindrical} for details. We already have that there are at least three acute corners --- namely at $u_-,u_+$ and $d_3$ --- and that $|\partial R|=2$. It follows that
  \begin{equation*}
 1\leq \frac{1}{2}-2g(R)+\frac{l}{2}.
 \end{equation*}

\noindent On the other hand $l\leq 1$, so in fact we have that $l=1$ and $g(R)=0$, so that $R$ is an annulus.
Let $\partial_\pm A$ denote the images of the two boundary components of $R$ in $S$. Without loss of generality we take $\partial_-A$ to be the component supported in $\widetilde{\beta}_1\cup\widetilde{\alpha}_1$. Observe that $\partial_+A$ is homotopic to $\partial_-A$ in $S$, with a homotopy induced by a $[0,1]$-parameterization of the obvious foliation of the annulus $R$ by circles. Since we are assuming that $\beta_1$ and $\alpha_1$ do not form any bigons it follows from~\cite[Theorem 2.7]{MR804478} that $\partial_+A$ is embedded in $S$. In particular, there is a path from any point on $\partial_+A$ to any basepoint in $S$ that does not intersect $\partial A$. It follows that, of the pairs of domains separated by an arc component of $\partial_+A$ in $\alpha_1$ or $\beta_1$, at least one has multiplicity zero. This contradicts the fact that we deduced earlier that $B=2$ and $C=1$, excluding this case. \end{proof}

    Proposition~\ref{thm:HFKaddtwists} will now be a consequence of the following Lemma.

\begin{lemma}\label{lem:surj}
    Suppose $\phi:\Sigma\to\Sigma$ is a diffeomorphism of a surface $\Sigma$ with $|\partial\Sigma|>1$ such that $\phi\circ\Delta_1^2$ sends some arc $a$ with exactly one endpoint on $\partial_1\Sigma$ weakly to the left at $\partial_1\Sigma\cap a$.
\end{lemma}

We will proceed by a diagrammatic analysis similar to that used in the proof of Lemma~\ref{lem:f1nontrivial}. We will refer frequently to Figure~\ref{fig:fdtc<-1}, which plays a role similar to that Figure~\ref{fig:uniquetriangle} played in the proof of Lemma~\ref{lem:f1nontrivial}.

\begin{proof}[Proof of Lemma~\ref{lem:surj}]
    Since  $f_1^*([\bm{u}_+]) =[\bm{v}_+]$ by Lemma~\ref{lem:f1nontrivial}, it is enough to show that $[\bm{v}_-]$ is also in the image of $f_1^*$. For this it suffices to show that $f_1(\bm{u}_-^3)=\bm{v}_-$ or $f_1(\bm{u}_-^3)=\bm{v}_++\bm{v}_-$.
    
    First, we show that $f_1(\bm{u}_-^3)$ can only be a linear combination of $\bm{v}_+$ and $\bm{v}_-$. This follows from the positions of the basepoints near $c_i, c_i'$ for $i>1$, $\widetilde{c}_j, \widetilde{c}_j'$ for $j> 2$  --- see Figure~\ref{fig:standardtraingles} --- which implies that if $\bm{x}$ is an element of $\T_{\bm{\alpha}}\cap\T_{\bm{\beta}}$ with $\langle f_1(\bm{u}_-^3),\bm{x}\rangle\neq 0$ then $c_i'\in\bm{x}$ for $i>1$, $\widetilde{c}_j'\in\bm{x}$ for $j>2$. Now, observe that $x\in \bm{x}$, as $x$ is the only intersection point lying on $\gamma_1$. Finally, observe that $v_\pm$ are the only intersection points in $\widetilde{\alpha}_1$ for which $\bm{x}$ would be of the correct Alexander grading. Recall here that the Alexander grading must be preserved by $f_1$, and can be computed via a count of how many intersection points are in the lower part of the Heegaard surface (see Section~\ref{sec:HFLintro} or~\cite{lipshitz2016heegaard} for instance), so these counts for $\bm{u}_-^3$ and $\bm{x}$ must match.
    
    It thus suffices to show that there is a single pseudo-holomorphic triangle contributing to $\langle f_1(\bm{u}_-^3),\bm{v}_-\rangle$. Let $\psi$ be a homotopy class of disk from $\bm{u}_-^3$ to $\bm{v}_-$. As noted above, the positioning of the basepoints near the $c_i$ intersection points for $i>1$ and the $\widetilde{c}_j'$ intersection points for $j>2$ imply that the shadow of $\psi$ is locally as in Figure~\ref{fig:standardtraingles}. It will thus suffice to determine  the part of the Heegaard triple diagram $\mathcal{H}_{\bm{\gamma},\bm{\beta},\bm{\alpha}}$ shown in Figure~\ref{fig:fdtc<-1}.

           \begin{figure}[ht]
\centering
   
\begin{tikzpicture}

\draw[dashed] (0,0) rectangle (10,12.5);

\draw[green, thick] (0,9) -- (10,9);


  \fill[orange!20]
  (0,11.5)  -- 
  (5,11.5) -- 
(5,10.5) -- (5.5,10.5) arc[start angle=270, end angle=360, radius=0.5] arc[start angle=180, end angle=90, radius=0.5]--
(10,11.5)  -- (10,9)
 --(0,9);


  \fill[orange!50]
  (0,9)  -- 
  (10,9) -- 
(10,10.5) --
(7,10.5) arc[start angle=90, end angle=180, radius=0.5] -- (6.5,10) arc[start angle=0, end angle=-90, radius=0.5] --
  (5,9.5) -- (5,10.5)--
  (0,10.5) -- 
  (0,9.5);


  \fill[orange!80]
  (0.5,9) arc[start angle=180, end angle=90, radius=0.5] -- 
  (5,9.5) -- 
  (5,9) -- 
  (0.5,9) -- 
  cycle;


   \fill[orange!20]
  (1.5,2)  -- 
  (8.5,2) arc[start angle=-90, end angle=0, radius=0.5]
  -- (9,2.5) -- (9,7.5)  arc[start angle=0, end angle=90, radius=0.5] -- (1.5,8) arc[start angle=90, end angle=180, radius=0.5]  --(1,2.5) arc[start angle=180, end angle=270, radius=0.5];

     \fill[white]
  (3,2)  -- (7,2) --(7,6) arc[start angle=0, end angle=90, radius=1]-- (4,7) arc[start angle=90, end angle=180, radius=1];

   \fill[white]
  (0,0)  -- (0,3) --(3,3) -- (3,0)--(0,0);

  \draw[red, thick] (5,0) -- (5,12.5);


 \draw[blue, thick] (6.5,0) arc[start angle=0, end angle=90, radius=0.5];
 \draw[blue, thick] (6,0.5) -- (1,0.5);
  \draw[blue, thick] (1,0.5) arc[start angle=270, end angle=180, radius=0.5];
 \draw[blue, thick] (0.5,1) -- (0.5,9);
   \draw[blue, thick] (0.5,9) arc[start angle=180, end angle=90, radius=0.5];
 
  \draw[blue, thick] (1,0.5) arc[start angle=270, end angle=180, radius=0.5];
  \draw[blue, thick] (6,9.5) -- (1,9.5);
 \draw[blue, thick] (6,9.5) arc[start angle=270, end angle=360, radius=0.5];
 \draw[blue, thick] (6.5,10) arc[start angle=180, end angle=90, radius=0.5];
  \draw[blue, thick] (7,10.5) -- (10,10.5);
    \draw[blue, thick] (0,10.5) -- (5.5,10.5);
     \draw[blue, thick] (5.5,10.5) arc[start angle=270, end angle=360, radius=0.5];

      \draw[blue, thick] (6.5,10) arc[start angle=180, end angle=90, radius=0.5];
  \draw[blue, thick] (7,10.5) -- (10,10.5);
    \draw[blue, thick] (0,10.5) -- (5.5,10.5);
     \draw[blue, thick] (5.5,10.5) arc[start angle=270, end angle=360, radius=0.5];
         \draw[blue, thick] (6,11) arc[start angle=180, end angle=90, radius=0.5];
           \draw[blue, thick] (6.5,11.5) -- (10,11.5);
             \draw[blue, thick] (0,11.5) -- (5,11.5);
                \draw[blue, thick] (5,11.5) arc[start angle=270, end angle=360, radius=0.5];

                 \draw[blue, thick] (5.5,12) -- (5.5,12.5);


\draw[red, thick] (0,5) -- (3,5);
 \draw[red, thick] (3,5) arc[start angle=90, end angle=-90, radius=1];
\draw[red, thick] (0,3) -- (3,3);


\draw[red, thick] (10,5) -- (7,5);
 \draw[red, thick] (7,5) arc[start angle=90, end angle=270, radius=1];
\draw[red, thick] (10,3) -- (7,3);


\draw[green, thick]  (1.5,2) -- (8.5,2);

\draw[green, thick]  (8.5,2) arc[start angle=270, end angle=360, radius=0.5];

\draw[green, thick]  (9,2.5) -- (9,7.5);

\draw[green, thick]  (9,7.5) arc[start angle=0, end angle=90, radius=0.5];
\draw[green, thick]  (1.5,8) -- (8.5,8);

\draw[green, thick]  (1.5,8) arc[start angle=90, end angle=180, radius=0.5];

\draw[green, thick]  (1,7.5) -- (1,2.5);

\draw[green, thick]  (1,2.5) arc[start angle=180, end angle=270, radius=0.5];


\draw[blue, thick]  (3,2) arc[start angle=180, end angle=270, radius=1];
\draw[blue, thick]  (4,1) -- (6,1);
\draw[blue, thick]  (6,1) arc[start angle=270, end angle=360, radius=1];

\draw[blue, thick]  (7,6) -- (7,2);

\draw[blue, thick]  (7,6) arc[start angle=0, end angle=90, radius=1];
\draw[blue, thick]  (4,7) -- (6,7);

\draw[blue, thick]  (4,7) arc[start angle=90, end angle=180, radius=1];
\draw[blue, thick]  (3,6) -- (3,2);

  \filldraw[black] (5,0.5) circle (2pt);
    \node at (4.8,0.3) {\( {c}_1 \)};

      \filldraw[black] (5,9.5) circle (2pt);
    \node at (5.2,9.3) {\( {d}_1 \)};
    
          \filldraw[black] (5,10.5) circle (2pt);
    \node at (5.2,10.2) {\( {d}_2 \)};

         \filldraw[black] (5,11.5) circle (2pt);
    \node at (5.2,11.2) {\( {d}_3 \)};

          \filldraw[black] (5,9) circle (2pt);
    \node at (5.2,8.7) {\( x \)};

          \filldraw[black] (0.5,9) circle (2pt);
    \node at (0.7,8.7) {\( \theta_1 \)};

  \filldraw[black] (3,5) circle (2pt);
    \node at (2.7,5.3) {\(y_- \)};
    
     \filldraw[black] (7,5) circle (2pt);
    \node at (7.3,5.3) {\(y_+ \)};

      \filldraw[black] (1,3) circle (2pt);
    \node at (1.3,2.7) {\( v_- \)};

      \filldraw[black] (0.5,3) circle (2pt);
    \node at (0.2,2.7) {\( x_2 \)};
     \filldraw[black] (0.5,5) circle (2pt);
    \node at (0.2,5.2) {\( x_1 \)};

      \filldraw[black] (5,1) circle (2pt);
    \node at (5.2,0.8) {\( y_2 \)};
     \filldraw[black] (5,7) circle (2pt);
    \node at (5.2,7.2) {\( y_1 \)};

      \filldraw[black] (3,3) circle (2pt);
    \node at (3.3,2.7) {\( u_- \)};

 \filldraw[black] (7,3) circle (2pt);
    \node at (7.3,2.7) {\( u_+\)};

          \filldraw[black] (9,3) circle (2pt);
    \node at (9.3,2.7) {\( v_+ \)};

       \filldraw[black] (7,2) circle (2pt);
    \node at (7.2,1.7) {\( \widetilde{\theta}_1 \)};


       \node at (2.5,10.7) {\( A\)};

       \node at (2.5,9.8) {\( B\)};

 \node at (5.5,9.8) {\( A\)};

 \node at (2.5,9.2) {\( D\)};
 
    \node at (8,9.8) {\( B\)};

       \node at (8,11.8) {\( U\)};

   \node at (0.2,8.5) {\( E\)};
     \node at (2.5,8.5) {\( F\)};

       \node at (8,8.5) {\( E\)};

       \node at (2.5,7.5) {\( T\)};

 \node at (8,7.5) {\( S\)};

     \node at (0.2,4) {\( G\)};
      \node at (0.7,4) {\( H\)};
     \node at (2.5,4) {\( I\)};
    \node at (3.5,4) {\( w_1\)};

   \node at (6.5,4) {\( z_1\)};

 \node at (4.5,4) {\( J\)};

 \node at (5.5,4) {\( K\)};
     \node at (8,4) {\( R\)};
      \node at (9.5,4) {\( G\)};

   \node at (2.5,2.5) {\( P\)};
    \node at (8,2.5) {\( Q\)};

          \node at (4.5,1.2) {\( M\)};

 \node at (5.5,1.2) {\( N\)};

    \node at (2.5,1.2) {\( L\)};

  \node at (0.4,0.4) {\( 0\)};
  \node at (9.4,0.4) {\( 0\)};


    \node at (9,9.2) {\( \gamma_1\)};
        \node at (9,8.1) {\(\widetilde{ \gamma}_1\)};

         \node at (2.7,6.6) {\(\widetilde{ \beta}_1\)};
            \node at (2,5.4) {\(\widetilde{ \alpha}_1\)};

              \node at (4.5,5.4) {\({ \alpha}_1\)};
              
              \node at (6.6,0.5) {\({ \beta}_1\)};
   
\end{tikzpicture}
    \caption{A neighborhood of $\partial_1\Sigma$ in the Heegaard triple for $f$ when $\Delta_1^2\circ\phi$ sends the arc $a_1$ weakly to the left. Here $\Delta_1$ is a right-handed Dehn twist. The Heegaard diagram $\mathcal{H}_0(-L)$ is obtained by forgetting the $\beta$-curves, which are shown in blue. The Heegaard diagram $\mathcal{H}(-L)$ is obtained by forgetting the $\gamma$-curves, which are shown in blue. }\label{fig:fdtc<-1}
   \end{figure}
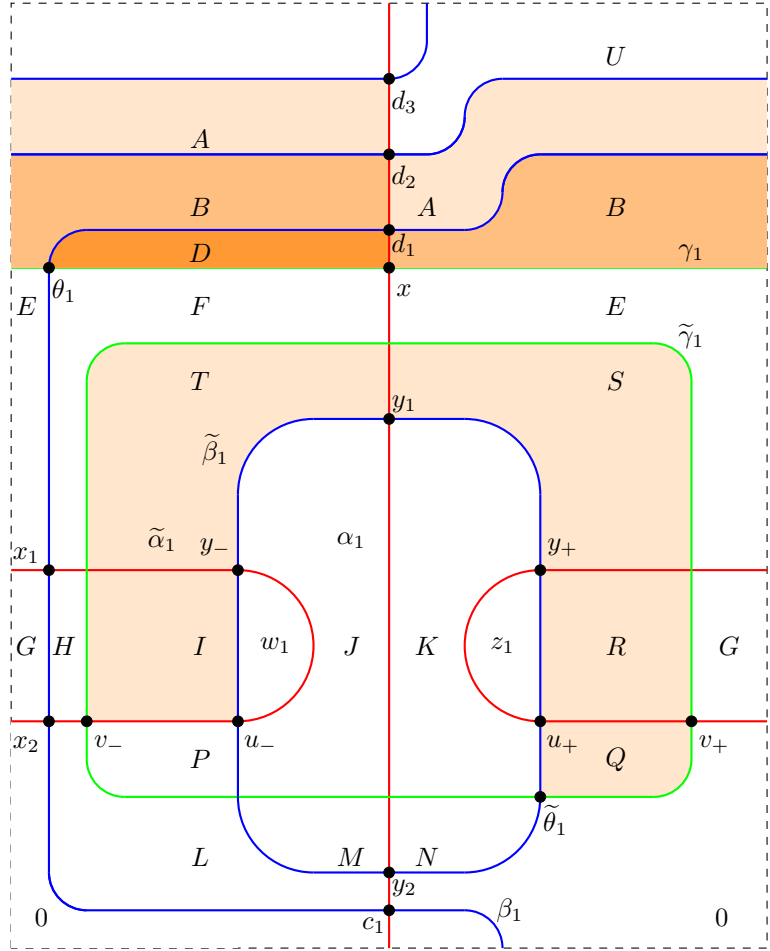

Consider the behavior of $\psi$ near $\widetilde{\alpha}_1$.  Observe that $\partial\psi$ passes from $v_-$ to $u_-$ along $\widetilde{\alpha}_1$. Note that $\partial\psi$ cannot go clockwise from $v_-$ to $u_-$ along $\widetilde{\alpha}_1$, since this would require that $G<0$, violating the principle of positivity of domains. Thus, it goes counterclockwise. However, $\partial\psi$ cannot pass $u_-$, as else we would require that $J<0$, violating the principle of positivity of domains once again. It follows that $I=P+1$. Since $\partial\psi$ does not pass along any other segment of $\widetilde{\alpha}_1$ other than the one between $I$ and {$P$}, it follows that the multiplicities of the regions connected by all other segments of $\widetilde{\alpha}_1$ are equal. In particular we have that; $H=L=F$, $G=E=0$, $I=P+1=T$,  $J=0$, $K=0$ and $Q=R=S$.

   Now consider the behavior of $\psi$ near $\alpha_1$.  Observe that $\partial\psi$ passes from $x$ to $d_3$ along $\alpha_1$. Since $J=K=0$, we see that $\partial\psi$ must pass vertically from $x$ to $d_3$  and stop on the first instance it reaches $d_3$. It follows that $A=U+1$, $B=A+1$, and $D=B+1$. As in the previous paragraph, since $\partial\psi$ avoids all other segments of $\alpha_1$, it follows too that $L=0$, $E=F$, $T=S$, and $M=N$. Combining these identities with the fact that $E=0$, we see that $F=H=L=0$.  
   
  Consider the behaviour of $\psi$ near $\beta_1$. Note that $\partial\psi$ must go from $d_3$ to $\theta_1$ along $\beta_1$. Since there are zeros on the right and left of $\beta_1$ at the bottom of Figure~\ref{fig:fdtc<-1}, it follows that $\partial\psi$ must exit $d_3$ to the left and stop on the first instance it reaches $\theta_1$. It follows that $A=1$, $B=2$, $C=3$, and $D=3$, while every region in $-\Sigma$ --- with the exception of the domains shown in Figure~\ref{fig:fdtc<-1} --- has the same multiplicity, $0$.
   
Now consider the behaviour of $\psi$ near $\widetilde{\beta}_1$. Observe that $\partial\psi$ passes from $u_-$ to $\widetilde{\theta}_1$. Since $J=0$, there are zeros on the left and right of the segment of $\widetilde{\beta}_1$ that passes counterclockwise --- as in Figure~\ref{fig:fdtc<-1} --- from $u_-$ to $\widetilde{\theta}_1$, $\partial\psi$ must pass clockwise from $u_-$ to $\widetilde{\theta}_1$ and stop on the first instance it reaches $\widetilde{\theta}_1$. As a result we have that $M=0$ and $N=0$. It also follows from the above identities that $I=T=S=R=Q=1$.

This entirely determines the shadow of $\psi$; namely it has $I=T=S=R=Q=A=1,B=2, D=3$ and multiplicity zero in all other regions. Observe that $\psi$ has a unique pseudo-holomorphic representative by the Riemann mapping theorem. The desired result follows.\end{proof}

    We conclude this subsection with the proof of Proposition~\ref{thm:HFKaddtwists}.
\begin{proof}[Proof of Proposition~\ref{thm:HFKaddtwists}]
   Since $\rank(\widehat{\HFL}(-{Y}_0,-{L},{-[\Sigma]},1-G))=2$ by Lemma~\ref{lem:0surgery}, and $f_1^*$ is surjective by Lemma~\ref{lem:surj}, the result follows from the exact triangle~(\ref{eq:exactneg}). 
\end{proof}

\subsection{Rank Bounds for Link Floer Homology}

We conclude this section by deducing the lower bounds on ${\rank(\widehat{\HFL}(L_\phi, {[\Sigma]},1-G))}$ promised in the introduction. To do so, we first give bounds on $${\rank(\widehat{\HFL}(L_\phi, {[\Sigma]},1-G))}$$ where $\phi$ has $-1<\FDTC(\phi)<1$ and then extract general bounds via applications of Proposition~\ref{thm:HFKaddtwists},  and Proposition~\ref{thm:HFKaddtwists1}. We begin with two purely geometric lemmas. The first is a standard fact we include for the sake of completeness.

\begin{lemma}\label{lem:nonidentityinterestingbasis}
    Suppose that for each arc $a$ connecting some component $\partial_i\Sigma$ to a component $\partial_j\Sigma$ with $j\neq i$ the map $\phi:\Sigma\to \Sigma$ neither sends $a$ strongly to the left nor strongly to the right at $\partial_i\Sigma\cap a$. Then $\phi$ is isotopic to the identity, so that in particular $\FDTC(\phi)=0$.
\end{lemma}

\begin{proof}
    Consider a basis consisting of arcs for $\Sigma$, every element of which has boundary components on distinct components of $\partial \Sigma$. One can isotope $\phi$ to a diffeomorphism which fixes this basis of arcs. The result now follows from an application of Alexander's trick.
\end{proof}

We will also use the following Lemma:

\begin{lemma}\label{lem:arcs}
Let $\Sigma$ be a surface with non-empty boundary and suppose $\phi:\Sigma\to\Sigma$ is not the identity homeomorphism. Then there is a basis of arcs for $\Sigma$ such that for every boundary component of $\Sigma$, $\partial_i\Sigma$, at least one of the arcs with a boundary component on $\partial_i\Sigma$ is not isotopic rel boundary to its image under $\phi$.
\end{lemma}

Note of course that any basis of arcs for a surface $\Sigma$ contains an arc with a boundary component on $\partial_i\Sigma$, for every $i$.

\begin{proof}

  We construct such a basis of arcs $\mathcal{B}$ recursively. To that end, order the boundary components of $\Sigma$, $(\partial_i\Sigma)$. Pick a basis of arcs $\mathcal{A}_1$ for $\Sigma$, each element of which has exactly one boundary component on $\partial_1\Sigma$ (or, in the special case that $\Sigma$ has only one boundary component, pick any basis of arcs). Observe that if all of these arcs are fixed up to isotopy by $\phi$ then $\phi$ is isotopic to the identity, by Alexander's trick. Thus at least one arc is not fixed up to isotopy, by assumption. Add such an arc to $\mathcal{B}$.

    We now proceed recursively through the remaining boundary components which do not intersect any of the arcs already added to $\mathcal{B}$. We proceed as follows: fix such a boundary component $\partial_j\Sigma$. By the change of coordinates principle (see~\cite[Chapter 1.3]{farb_primer_2012}), there is a diffeomorphism $\psi$ of $\Sigma$ fixing $\partial \Sigma$ point-wise that sends the arcs in $\mathcal{B}$ to a standard set of arcs $B$ --- as shown in red in the upper half of Figure~\ref{fig:adaptedheegaarddiagram} in the genus 2, two boundary component case, for example --- that go between the boundary components in the most straightforward manner possible. Now pick a basis of arcs $\mathcal{A}_{j}$ for $\Sigma$, each element of which has exactly one boundary component on $\partial_j\Sigma$ such that no arc in $\mathcal{A}_{j}$ intersects any of the arcs in $B$. This is possible to do explicitly since $B$ is standard.  Observe that since $\phi$ is not the identity, neither is $\psi \phi \psi^{-1}$, and hence there must be an arc $a$ in $\mathcal{A}_{j}$  that is not fixed by $\psi \phi \psi^{-1}$ up to isotopy. Since
    $$(\psi \phi \psi^{-1})(a) \neq a,$$
    we have that
    $$(\phi)(\psi^{-1}(a)) \neq \psi^{-1}a$$
    so $\phi$ does not fix $\psi^{-1}\alpha$ up to isotopy. Add $\psi^{-1}a$ to $\mathcal{B}$, which takes care of $\partial_j \Sigma$. For all other boundary components proceed similarly.\end{proof}

    With this lemma at hand we can give the following general lower bound on the rank of link Floer homology:

\begin{lemma}\label{lem:atleasthalf}

 Suppose that $\Sigma$ is a surface with $m$ boundary components that isn't a disk. Let $\phi:\Sigma\to\Sigma$ be a non-identity diffeomorphism which fixes $\partial\Sigma$ pointwise. Then:

    \begin{align*}
    \rank(\widehat{\HFL}(\partial\Sigma,[\Sigma],1-G))\geq \Big\lceil\frac{m}{2}\Big\rceil.\end{align*}
\end{lemma}

This is a generalization of~\cite[Theorem 1.1]{baldwin_note_2018}. The proof is inspired by a proof of a similar result due to the first author and Dey under slightly stronger hypotheses~\cite[Theorem 4.1]{binns2022rank}. Here, unlike in~\cite[Theorem 4.1]{binns2022rank}, we do not require that the manifold with open book decomposition $(\Sigma,\phi)$ be a rational homology sphere.

\begin{proof}[Proof of Lemma~\ref{lem:atleasthalf}]

Let $\phi,\Sigma, m$ be as in the statement of the theorem and set $L:=\partial \Sigma$. By symmetry properties of link Floer homology, it suffices to show that $\widehat{\HFL}(-{L},[-\Sigma],1-G)\geq\Big\lceil\frac{m}{2}\Big\rceil$.

   Pick a basis of arcs as in the statement of Lemma~\ref{lem:arcs}. Let $L_i$ denote the $i$th component of $L:=\partial\Sigma$. Let $\{a_j\}$ denote a collection of arcs from that basis, at most one per component of $L$, with a single endpoint on $L_j$ that are sent  (strongly) to the left at that point. Up to reversing orientation we may assume that $|\{a_j\}|\geq \lceil\frac{m}{2}\rceil$. Obtain a Heegaard diagram for $-L$ as in Section~\ref{sec:Heegaard}. 

    Recall that for each $i$ there is a spectral sequence from $\widehat{\CFL}(-{L})$ to $\widehat{\HFL}(-{L\setminus -L_i})\otimes V$ with $E_2$ page $\widehat{\HFL}(-{L})$ where $V$ is a vector space of rank $2$ supported in Alexander grading zero~\cite{HolomorphicdiskslinkinvariantsandthemultivariableAlexanderpolynomial}. Note here we are using the fact that $[\mu_i]\neq 0\in H_1(Y\setminus\nu(L),\partial(Y\setminus\nu(L));\Z)$ (since it has algebraic intersection number one with the embedded surface $\Sigma$) to guarantee that  the differential counting pseudo-holomorphic disks that are allowed to cross basepoints is filtered. Here $\mu_i$ is a meridian of $\partial_i\Sigma$. The induced differential on the $E_2$-page is a map $$\partial_i^*:\widehat{\HFL}(-{L},[-\Sigma],1-G)\to\widehat{\HFL}(-{L},[-\Sigma],-G).$$ Observe that the BRAID invariant, $\bm{c}$, of (an appropriate transverse approximation) of $L$ is non-zero and generates $\widehat{\HFL}(-{L},[-\Sigma],-G)$ by~\cite{tovstopyat2018transverse}. Consider the generators $\bm{d}_j$ that are given by taking the BRAID invariant and replacing the intersection point $c_j$ on the $\alpha$-curve obtained from $a_j$ --- the (strongly) left veering arc through $\partial_j\Sigma$ --- with $d_j$ the intersection point as shown in Figure~\ref{fig:uniquetriangle} in the $j=1$ case. Note that since we have chosen basis arcs that are strongly sent to the left for each of the $|\{a_j\}|$ components, and by the convenient positioning of the basepoints near the $c_k$ intersection points with $k\neq j$, and $\widetilde{c}_k$ generators for $k\neq 2$, we have that $\bm{d}_j$ represents a class in $\widehat{\HFL}(-{L},[-\Sigma],1-G)$. We also have that $\partial_j^*(\bm{d}_j)=\bm{c}$. This follows from the positioning of the basepoints near the $c_k$ and $\widetilde{c}_k$ intersection points as above, together with the fact that there is a bigon contributing to $\partial^*_j$ bounded by $\alpha_j$ and $\beta_j$, and containing $w_j$, coming from the fact that $a_j$ is sent weakly to the left at its endpoint on $\partial_j\Sigma$.
    Likewise, $\partial_k^*(\bm{d}_j)=0$ for $k\neq j$, since there is no similar such bigon. It follows that the generators $\bm{d}_j$ are linearly independent and non-trivial, so in turn that $\rank(\widehat{\HFL}(-{L},[\Sigma],1-G))\geq \lceil\frac{m}{2}\rceil$, as desired.\end{proof}

    We have an even stronger result in the identity case:

\begin{lemma}\label{lem:identitycase}
     Suppose $\Sigma$ is a surface with $m$ boundary components and let $\mathbbm{1}:\Sigma\to\Sigma$ be the identity. Then:
    \begin{align*}
    \rank(\widehat{\HFL}(\partial\Sigma,[\Sigma],1-G))\geq 2(g(\Sigma)+|\partial\Sigma|-1).\end{align*}

\end{lemma}

\begin{proof}

We continue with the same framework and notation as in the proof of Lemma~\ref{lem:atleasthalf}. In particular we consider the Heegaard diagram $\mathcal{H}$ shown in Figure~\ref{fig:identity}.  We claim that the generators $\bm{d}_i$ (and $\widetilde{\bm{d}}_i$) consisting of the intersection point $d_i$ (respectively $\widetilde{d}_i$), along with the remaining $c_j$ and $\widetilde{c}_j$ generators form a summand of the Heegaard Floer chain complex of rank $2g(\Sigma)+|\partial\Sigma|-1$ with vanishing differential. 
To see that the summand is of rank $2g(\Sigma)+|\partial\Sigma|-1$ note that there are $2g(\Sigma)+|\partial\Sigma|-1$ generators $\bm{d}_i$, and $|\partial\Sigma|-1$ generators $\widetilde{\bm{d}}_i$
The fact that the differential is trivial can be seen by noting that the Heegaard diagram is \emph{nice} in the sense of~\cite{SarkarWang2010}, so that the differential only counts bigons and rectangles. We claim that there are no such domains starting or ending at the two generators, from which the result follows. For the claim observe that the positioning of the basepoints near the $c_i$ imply that any disk contributing to $\partial \bm{d_i}$ fixes each $c_j$ intersection point. It can then be readily checked that $\partial\bm{d}_i=0$ from the positioning of the basepoints near the $d_i$ intersection points.

Suppose now that there is a generator $\bm{x}$ and a disk $\phi\in \pi_2(\bm{x},\bm{d}_i)$. Observe that $d_i\in\bm{x}$ by the positioning of the basepoints near $d_i$. It is also clear that any such $\phi$ cannot have a bigon domain. To rule out the case of rectangular domains, $\mathcal{D}$, consider the arc component, $b$, of $\partial\mathcal{D}$ that approaches some $c_j\in\bm{d}_i$ along a $\beta$-curve (an identical argument can be used to rule out the $\widetilde{c}_j$ cases). Suppose $b$ starts at some point $x\in\bm{x}$. Observe that $x\in \Sigma$ (the lower half of the Heegaard diagram as shown in~\ref{fig:identity}) in order for $\bm{x}$ to be of the correct Alexander grading. Indeed, one can see from the positioning of the basepoints near the $d_k$ and $\widetilde{d_k}$ intersection points that $b$ is entirely contained in $\Sigma$. Since the angle at $x$ is required to be acute the arc component of $\partial\mathcal{D}$ along an $\alpha$-curve, that terminates at $x$ must come from the direction of $-\Sigma$, as opposed to from the direction of the $c_l$ (or $\widetilde{c_l}$) intersection point that the $\alpha$ curve contains. Moreover, this component must be entirely supported in $\Sigma$ for the same reason as argued earlier, namely the multiplicities of the domains near the $d_k$ and $\widetilde{d_k}$ intersection points. This contradicts the requirement that this arc begins at an element $c_l$ (or $\widetilde{c_l}$), concluding the proof.
\end{proof}

    \begin{figure}[ht]
\centering

   \begin{tikzpicture}


\draw[ thick]
(-1,-3.5)
.. controls (3,-3.5) and (3,3.5) ..
(-1,3.5)
.. controls (-7.5,3.5) and (-7.5,-3.5) ..
(-1,-3.5);

\draw[red, thick] (0,0) ellipse (1 and 3.0);

\draw[red, thick] (-3,0) ellipse (1.8 and 1.0);

\draw[ draw=black, thick] (0,2) ellipse (0.5 and 0.3);
\draw[draw=black, thick] (0,-2) ellipse (0.5 and 0.3);
\draw[draw=black, thick] (-3.25,0) ellipse (0.3 and 0.5);

\node at (-4,0) {$w_1$};
\node at (-5.1,0) {$z_1$};
\node at (-1.5,0) {$w_2$};
\node at (-2.5,0) {$z_2$};

\draw[blue, thick] (-4.6,0) ellipse (0.8 and 0.5);

\draw[blue, thick]
(0.2,1.75)
.. controls  (0.2,1.0) and (-2,1)  ..
 (-2,0)
.. controls (-2,-1) and (0.2,-1.0) ..
(0.2,-1.75);

\draw[blue, dashed, thick]
(0.2,-1.75) .. controls (0.2,-1.5) and (1.0,-1.4) ..
(1.7,-1.5);

\draw[blue, dashed, thick]
(0.2,1.75) .. controls (0.2,1.5) and (1.0,1.4) ..
(1.7,1.5);

\draw[blue, thick]
(1.7,-1.5) .. controls (0,-1) and (-1.8,-0.5) ..
(-1.8,0) .. controls (-1.8,0.5)  and  (0,1)..
(1.7,1.5)
;

\draw[blue, thick]
(-5.7,0) .. controls (-5.7,-2) and (-2.3,-2) ..
(-2.3,0) .. controls (-2.3,2) and (-5.7,2)  ..
(-5.7,0)

;


\draw[red, thick]
(-3.5,0.2) .. controls (-4.5,1) and (-4.5,-1) ..
(-3.5,-0.2);

\draw[red, thick,dashed]
(-3.5,0.2) .. controls (-4,1) and (-5,1) ..
(-5.9,0.2);

\draw[red, thick,dashed]
(-3.5,-0.2) .. controls (-4,-1) and (-5,-1) ..
(-5.9,-0.2);

\draw[red, thick]
(-5.9,0.2) .. controls (-4.6,1) and (-4.6,-1) ..
(-5.9,-0.2);


\draw[blue, thick]
(-2.2,0)
.. controls (-2.2,-1.5) and (-0.5,-2.5) ..
(0,-2.5)
.. controls (0.1,-2.5) and (0.5,-2.5) ..
(0.6,-2)
.. controls (0.6,-0.7) and (-1.9,-1) .. (-1.9,0)

.. controls  (-1.9,1)  and (0.6,0.7).. (0.6,2)
.. controls  (0.5,2.5)  and (0.1,2.5).. (0,2.5)

.. controls (-0.5,2.5) and  (-2.2,1.5) .. (-2.2,0)
;

\draw[red, thick, dashed]
(0,1.7)
.. controls (0.25,0.5) and (0.25,-0.5)..
(0.0,-1.7);
\draw[red, thick]
(0,1.7)
.. controls (-0.25,0.5) and (-0.25,-0.5)..
(0.0,-1.7);

\node at (-2.5,1.3) {$d_1$};
\filldraw[black] (-2.7,1) circle (2pt);
\node at (-2.5,-1.3) {$c_1$};
\filldraw[black] (-2.7,-1) circle (2pt);

\node at (-0.9,2.6) {$d_2$};
\filldraw[black] (-0.65,2.3) circle (2pt);
\node at (-0.9,-2.6) {$c_2$};
\filldraw[black] (-0.65,-2.3) circle (2pt);

\node at (-3.8,0.6) {$\widetilde{d}_1$};
\filldraw[black] (-4.1,0.4) circle (2pt);

\node at (-3.8,-0.6) {$\widetilde{c}_1$};
\filldraw[black] (-4.1,-0.4) circle (2pt);

\node at (-0.35,1.5) {${d}_3$};
\filldraw[black] (-0.1,1.35) circle (2pt);
\node at (-0.35,-1.5) {${c}_3$};
\filldraw[black] (-0.1,-1.35) circle (2pt);


\node at (-4,-1.8) {$\beta_1$};
\node at (-1.6,-1.8) {$\beta_2$};
\node at (-0.6,-3) {$\alpha_2$};
\node at (-3.6,-1.1) {$\alpha_1$};

\end{tikzpicture}
 
    \caption{A Heegaard diagram $\mathcal{H}$ adapted to the identity diffeomorphism on surfaces with two boundary components in the sense of Section~\ref{sec:braidinvt}. $d_i$ are intersection points as shown. }\label{fig:identity}
   \end{figure}
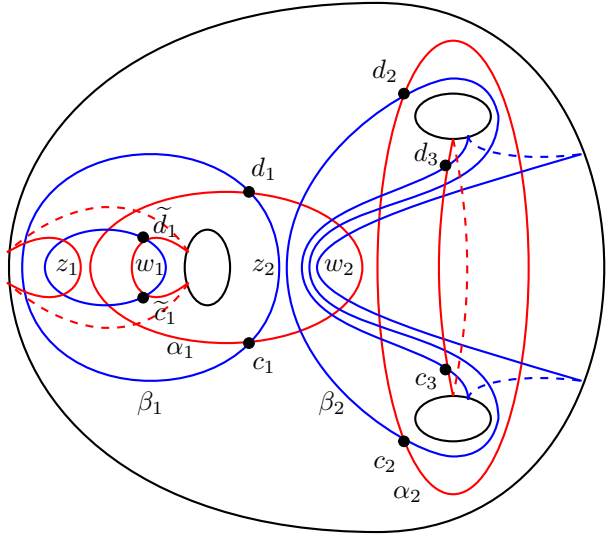

We are now ready to prove our main theorem concerning link Floer homology, which we restate here for the reader's convenience.

\corHFKrank*

Here $G:=1-g(\Sigma)-|\partial\Sigma|$, the minimum grading in which $\widehat{\HFL}$ is non-trivial.

\begin{proof}
    Observe that $\phi$ can be written as $\prod_{1\leq i\leq m}\Delta_i^{k_i}\phi'$ where $\Delta_i$ is a full right handed Dehn twist about the $i$th boundary component of $\Sigma$, $k_i$ is given by $\big\lfloor\FDTC(\phi,\partial_i\Sigma)\big\rfloor$ if $\FDTC(\phi,\partial_i\Sigma)\geq0$ or $\Big\lceil\FDTC(\phi,\partial_i\Sigma)\Big\rceil$ if $\FDTC(\phi,\partial_i\Sigma)\leq0$. Then $0\leq\FDTC(\phi',\partial_i\Sigma)<1$ for all $i$ such that $\FDTC(\phi,\partial_{i}\Sigma) \geq 0$ and $-1 < \FDTC(\phi',\partial_{i}\Sigma) \leq 0$ for all $i$ such that $\FDTC(\phi,\partial_{i}\Sigma) \leq 0$.
    
The proof is an iterative process. Begin with $\phi$ and the first boundary component. If $\FDTC(\phi,\partial_1\Sigma)\leq0$, apply at least $(|k_1|-1)$ iterated applications of Proposition~\ref{thm:HFKaddtwists} to $\phi$, and then at most one application of Proposition~\ref{thm:HFKaddtwists1}. If $\FDTC(\phi,\partial_i\Sigma)\geq0$, instead apply these Propositions to $\phi^{-1}$. Continue this process until, by at least $\sum_{i=1}^m(|k_i|-1)$ iterated applications of Proposition~\ref{thm:HFKaddtwists}, at most one application of Proposition~\ref{thm:HFKaddtwists1} per boundary component, and an appropriate number of applications of Equation~\ref{eq:invariance}, you have arrived at $\phi'$. Let $L'$ be the boundary of the open book $\phi':\Sigma\to \Sigma$.  Now observe that $\rank(\widehat{\HFL}(L',[\Sigma],G-1))\geq \Big\lceil\frac{m}{2}\Big\rceil$ by Proposition~\ref{lem:atleasthalf} when $\phi'$ is not the identity and Lemma~\ref{lem:identitycase} in the case that $\phi'$ is the identity, giving the desired result. 
\end{proof}

We can give the following improvement of Theorem~\ref{corHFKrank} in the special case that $\phi$ is a product of boundary Dehn twists.

\begin{lemma}\label{prop:HFKintegerFDTC}
Let $\Sigma$ be a surface with $m$ boundary components that is neither a disk nor an annulus. If $\bm{\Delta}:=\prod\Delta_i^{n_i}$ where $\Delta_i$ is a boundary Dehn twist about the $i$th boundary component of $\Sigma$. Then:  \begin{align*}
    \rank\big(\widehat{\HFL}(\partial\Sigma, {[\Sigma]},G-1)\big)\geq 2(g(\Sigma)+|\partial\Sigma|-1)+\underset{1\leq i\leq m}{\sum}2\max\{|n_i|-1,0\}.\end{align*}

\end{lemma}

\begin{proof}
   In the $m>1$ case, this follows exactly as in the proof of Theorem~\ref{corHFKrank}, except that now we can always apply Lemma~\ref{lem:identitycase} for the base case $\phi'=\mathbbm{1}_\Sigma$ --- following the notation from the proof of Theorem~\ref{corHFKrank}. The $m=1$ case follows as in the $m>1$ case except that we apply Proposition~\ref{prop:newknotFDTC<0} instead of Proposition~\ref{thm:HFKaddtwists} and Proposition~\ref{thm:HFKaddtwists1}.
\end{proof}

\begin{remark}We note in passing that if one additionally assumes that the three-manifold with open book decomposition $(\Sigma,\phi)$ in the statement of Theorem~\ref{corHFKrank} is a rational homology sphere, then the ``$\Big\lceil\frac{m}{2}\Big\rceil$" can be improved to an $m$ by applying~\cite[Theorem 4.1]{binns2022rank} as opposed to Lemma~\ref{lem:atleasthalf} in the proof.
\end{remark}

\begin{remark}\label{rmk:nociels}
    
 Suppose that for a non-identity diffeomorphism $\phi:\Sigma\to\Sigma$ there is  a basis of arcs $\{a_j\}$ with the property that for all $k$ at least one $a_j$ abutting to $\partial\Sigma_k$ is sent strictly to the left by $\phi'$ (the diffeomorphism constructed from $\phi$ by composing with appropriate boundary twists in the proof of Theorem~\ref{corHFKrank}). Then the ``$\Big\lceil\frac{m}{2}\Big\rceil$" in the statement of Lemma~\ref{lem:atleasthalf} can be replaced by an $m$ and likewise for Theorem~\ref{corHFKrank}. In Lemma~\ref{lem:HFLbraidlift}, we will show this can be achieved for open books arising as the double branched covers of braids. We will apply this result to obtain Theorem~\ref{con:akh}, our rank bound in annular Khovanov homology.
\end{remark}

We conclude this subsection by discussing the annulus case. Conceptually this is no different from the higher genera, two boundary component cases.  In particular note that the technical lemmas at the beginning of this section apply for annuli. However, since the fractional Dehn twist coefficients of self-diffeomorphisms of annuli behave slightly differently than for higher genera --- in that a single Dehn twist about the annulus's core increases the FDTC by one on both boundary components (see \cite{ito2018fractional} Remark 4.3) --- the statement looks slightly different. Recall that every diffeomorphism of the annulus is some power of a positive Dehn twist, $\Delta_1=\Delta_2$, about the annulus' core.
\begin{proposition}\label{prop:annuluscase}
    Let $L_n$ denote the binding of an annular open book $A$ with monodromy $(\Delta_1)^n$. Then 
 \begin{align}\label{eq:annulus}
   \widehat{\HFL}(L_n,[A],i)&\cong \begin{cases}\F&\text{if $ i=\pm 1$}\\ \F^{\max\{2|n|,2\}}&\text{if $ i=0$}\\
        0&\text{otherwise.}\end{cases}
    \end{align}

\end{proposition}
\noindent Recall here that $\F$ is the field with two elements.

\begin{proof}
    Consider first the $n=1$ case. This is the positive Hopf link, which is well known to have link Floer homology of the desired form. The $n>1$ cases then follow from repeated applications of Proposition~\ref{thm:HFKaddtwists}. The $n\leq -1$ cases follow from the symmetry properties of link Floer homology.

    For the $n=0$ case, observe that$$\widehat{\HFL}(L_{0},[\Sigma],i)\cong\begin{cases} \F&\text{for $|i|=1$}\\0&\text{for $|i|>1$}.\end{cases}$$

     since $L_0$ is fibered and bounds an annulus. Moreover, Proposition~\ref{thm:HFKaddtwists1} implies that; \begin{align*}   
   2=\rank(\widehat{\HFL}(L_1, [A],0))\geq \rank(\widehat{\HFL}(L_0,[A],0)).\end{align*}

\noindent The inequality above is tight by Lemma~\ref{lem:identitycase}, concluding the proof. \end{proof}

Let $\partial_{\pm}A$ be the boundary components of $A$. Since $\FDTC((\Delta_1)^n,\partial_\pm A)=n$, Equation~(\ref{eq:annulus}) can be restated as
    \begin{align}\label{eq:annulus2}
  \widehat{\HFL}(L_n,[A],i)\cong \begin{cases}\F&\text{if $ i=\pm 1$}\\ \F^{\max\{2|\FDTC((\Delta_1)^n,\partial_\pm A)|,2\}}&\text{if $ i=0$}\\
        0&\text{otherwise.}\end{cases}
    \end{align}

\section{Annular Khovanov homology}\label{sec:akh}

In this section we prove our rank bound for annular Khovanov homology in terms of the fractional Dehn twist coefficient, Theorem~\ref{con:akh}.
\subsection{A brief review of annular Khovanov homology}

Annular Khovanov homology is an invariant of links in the thickened annulus due to Asaeda, Przytyki, and Sikora~\cite{asaeda_categorification_2004} that generalizes Khovanov's original invariant for links in $\R^3$~\cite{khovanov2000categorification}. It assigns to each such link  a finitely generated module, which we denote $\AKh(L)$. We will take coefficients in $\Z/2$ unless otherwise stated. $\AKh(L)$ comes equipped with a $\Z$-valued \emph{annular grading}. The annular gradings in which $\AKh(L)$ is supported are of the same parity. We denote the annular grading $m$ portion of annular Khovanov homology by $\AKh(L,m)$. If $\widehat{\beta}$ is an $n$-braid closure then a straightforward computation shows that the maximum annular grading in which $\AKh(\widehat{\beta})$ is non-trivial is $n$, and indeed that the rank there is one. In this section we will be interested in the next to maximum grading in which $\AKh(\widehat{\beta})$ is non-trivial for an $n>1$-braid closure $\widehat{\beta}$, namely $n-2$.

We can alternatively view annular links as links in the complement of an unknot $A$ in $S^3$. (If $L$ is the closure of a braid, then $A$ is the braid axis.) Let $\Sigma(L)$ denote the double branched cover of $S^{3}$ branched over $L$, and let $\widetilde{A}$ denote the lift of $A$ to $\Sigma(L)$. Let $\widetilde{D}$ denote the lift of a disk $D$ bounded by $A$. There is a spectral sequence from the annular Khovanov homology of an annular link $L$ to (a version of) the knot Floer homology of $\widetilde{A}$ relative to the homology class of $\widetilde{D}$. 
 Roberts~\cite{roberts_knot_2013} proved this in the odd index case and Grigsby and Wehrli extended this result to the general case~\cite[Theorem 2.1]{grigsby_khovanov_2010}. Note that while \cite[Theorem 2.1]{grigsby_khovanov_2010} refers to sutured Floer homology rather than knot Floer homology, but \cite[Proposition 2.24]{grigsby_khovanov_2010} demonstrates that these are, up to some simple modifications, equivalent in the setting at hand. 
 
Let $\beta$ be a braid of index $n$. If $n=1$ --- i.e. for the unique $1$-stranded braid --- we have that $\rank(\AKh(\widehat{\beta};\Z))=2$, with support in annular gradings  $\pm1$. A consequence of the spectral sequence described above is that if $n>1$ is odd then (see~\cite[Theorem 2.1, Proposition 2.24]{grigsby_khovanov_2010}):
\begin{equation}\label{eq:oddindex}
\rank\big(\AKh(\widehat{\beta},n-2)\big)\geq \rank\Big(\widehat{\HFL}\Big(\widetilde{A},\Sigma(\widehat{\beta}),[\widetilde{D}],\frac{n-3}{2}\Big)\Big)+1,
\end{equation}

\noindent while if $n$ is even index then:
\begin{equation}\label{eq:evenindex}
\rank\big(\AKh(\widehat{\beta},n-2)\big)\geq \rank\big(\widehat{\HFL}(\widetilde{A},\Sigma(\widehat{\beta}),[ \widetilde{D}],\frac{n-2}{2})\big).
\end{equation}

\noindent The reason these rank bounds depend on the parity of the braid index is clarified in~\cite[Proposition 2.24 ]{grigsby_khovanov_2010} . 

\subsection{The rank bound}

Our goal in this section is to prove the following:

\akh*

For the even $n$ case we need a slightly stronger version of Theorem~\ref{corHFKrank} under slightly stronger hypotheses. 
\begin{lemma}\label{lem:HFLbraidlift}
    Suppose $\widetilde{A}$ is the lift of the braid axis to the double branched cover of a $2n$-braid given by $(D,\beta)$ with $n\geq 1$. Then : \begin{align}\label{eq6.11}\rank(\widehat{\HFL}(\widetilde{A},\Sigma(\widehat{\beta}),[\widetilde{D}],1-G(\widetilde{D})))\geq 2+\max\{0,4\lfloor|\FDTC(\beta)|\rfloor-4\}.\end{align}

\end{lemma}

 This lemma follows from arguments similar to those applied in the proofs of Lemma~\ref{lem:atleasthalf} and Theorem~\ref{corHFKrank}. 

\begin{proof}
Suppose first that $\beta$ is not the identity. Pick a basis of arcs $a_i$ for $D$ where at least one of them is sent strongly to the left (perhaps after reversing orientation). This basis lifts to a basis of arcs for $\Sigma$ in the double branched cover containing an arc which is sent strongly to the left at both endpoints. In particular, in the notation of Lemma~\ref{lem:atleasthalf}, we have that $\rank(\widehat{\HFL}(\partial\Sigma,[\Sigma],G-1))\geq 2$. This rank bound also holds true in the case that $\beta$ is the identity, by Lemma~\ref{lem:identitycase}.

We can now apply the rest of the argument in the proof of Theorem~\ref{corHFKrank} to deduce that
 \begin{align*}
    \rank\big(\widehat{\HFL}(\widetilde{A}, {[\Sigma(\widehat{\beta})]},G(\widetilde{D})-1))\big)\geq 2+\underset{i\in \{1,2\}}{\sum}2\max\big\{\big(\big\lfloor|\FDTC(\phi,\widetilde{A}_i)|\big\rfloor-1\big),0\big\}.\end{align*}
 
 \noindent Here $\widetilde{A}_i$ are components of the lift of $A$, and $\phi$ is the monodromy of the open book with page $\widetilde{D}$. Since $n>1$ we have that $\FDTC(\phi,\widetilde{A}_i)=\FDTC(\beta)$ for each $i$, by~\cite[Theorem 4.2]{ito2018fractional}. \end{proof}

With this Lemma in hand, we can proceed to the proof of the main Theorem of this section.

\begin{proof}[Proof of Theorem~\ref{con:akh}]
 Suppose $\beta$ is an $n$-braid as in the statement of the Theorem. We have two cases according to whether $n$ is odd or even.
 
 If $n$ is even, then $\widetilde{A}$ has two components $\widetilde{A}_1$ and $\widetilde{A}_2$, and --- provided that the Euler characteristic of the branched cover of the disk is strictly negative, i.e. $n>2$ --- the fractional Dehn twist coefficients of each component is exactly the fractional Dehn twist coefficient of $\beta$~\cite[Theorem 4.2]{ito2018fractional}. Thus Equation~(\ref{eq:evenindex}) and Lemma~\ref{lem:HFLbraidlift} imply that:
\begin{align*}\rank(\AKh(\widehat{\beta}, n-2))&\geq \rank\big(\widehat{\HFL}(\widetilde{A},\Sigma(\widehat{\beta}),[\widetilde{D}],\dfrac{n-2}{2})\big)\\
&\geq \max\{2,4\big\lfloor|\FDTC(\beta)|\big\rfloor-2\},
\end{align*}
as desired.

The $n=2$ case can be treated similarly. Here the fractional Dehn twist of $\widetilde{A}_1$ and $\widetilde{A}_2$ are both twice the fractional Dehn twist coefficient of $\beta$, so Equation~(\ref{eq:evenindex}) and Proposition~\ref{prop:annuluscase} imply that:
\begin{align*}
\rank(\AKh(\widehat{\beta}, 0))&\geq \rank(\widehat{\HFL}(\widetilde{A},\Sigma(\widehat{\beta}),[ \widetilde{D}],0)) \\&=\max\{2,4|\FDTC(\beta)|\}\\
&\geq \max\{2,4\big\lfloor|\FDTC(\beta)|\big\rfloor-2\}.
\end{align*}

The result in the odd index case holds by a similar argument. The fractional Dehn twist of the open book corresponding to the double branched cover of an odd-index braid is half that of the braid~\cite[Theorem 4.2]{ito2018fractional}.  Thus Equation~(\ref{eq:oddindex}) and Proposition~\ref{prop:HFKsingleboundarycomponent} imply that:
\begin{align*}
\rank(\AKh(\widehat{\beta}, n-2))&\geq 1+ \rank\Big(\widehat{\HFL}(\widetilde{A},[ \widetilde{D}],\dfrac{n-2}{2})\Big)\\
&\geq 2+2\max\{\big\lfloor|\FDTC(\widetilde{A})|\rfloor-1,0\}\\& =\max\Big\{2\Big\lfloor\Big|\frac{\FDTC(\beta)}{2}\Big|\Big\rfloor,2\Big\},\end{align*}concluding the proof.\end{proof}

Note that the $n=2$ case in the proof is somewhat redundant; $\AKh(-,\Z/2)$ can be computed by hand for $2$-braids following, say, the computation of $\AKh(-,\C)$ in~\cite[Section 9.3]{grigsby_annular_2018}. We conclude this section with a stronger bound for certain special braids:

\begin{proposition}\label{prop:braidfulltwist}
  Let $\Delta_m$ denote the $m>1$-braid given by a full right handed Dehn twist. If $m$ is even then $\rank(\AKh(\widehat{\Delta^n_m},m-2))\geq m+\max\{0,4 |n|-4\}$.
\end{proposition}

\begin{proof}
Observe that the lift of the axis to the double branched cover of $\widehat{\Delta^n_m}$ is a fibered two component link of genus $\frac{m-2}{2}$. The result now follows directly from Lemma~\ref{prop:HFKintegerFDTC} and Equation~\ref{eq:evenindex}.
\end{proof}

\section{Examples and Questions}\label{sec:examples}

    In this section we provide a few groups of examples that may be of interest and pose some questions. We do not have examples where any of our rank bounds from Theorems~\ref{con:akh}, \ref{prop:HFKsingleboundarycomponent}, and \ref{corHFKrank} are tight, but we are limited substantially by the computational complexity of the corresponding calculations, so our search for such examples can by no means be considered exhaustive. We insted provide families of examples for both annular Khovanov homology and knot and link Floer homology that show that our rank bounds can be arbitrarily bad. 
   
    \subsection{The full twist in the braid group}

    Recall that $2$-stranded braids are given by $\sigma_1^n$ for some $n$. Our lower bound on $\rank(\AKh(\widehat{\sigma_1^{2m}},0))$ from Theorem~\ref{con:akh}, $\max\{2,4|m|-2\}$ is not tight, as the rank is in fact $\max\{2,4|m|\}$. To see this, one could give a complete computation of $\AKh(\widehat{\sigma_1^{m}};\Z/2)$ using a computation similar to that given by Grigsby-Licata-Wehrli in the complex coefficient case~\cite[Section 9.3]{grigsby_annular_2018}. Alternatively, $\rank(\AKh(\widehat{\sigma_1^{m}},0;\Z/2))))$ can be computed by combining Proposition~\ref{prop:annuluscase} and Equation~\ref{eq:evenindex}.

    For full twists on more strands, the bound continues to be non-tight, and the gap appears to widen. Using Hunt, Keese, Licata and Morrison's computer program (see \cite{hunt2015computing}) we find that the full twist on $3$ strands has annular Khovanov homology of rank $7$ (with complex coefficients) in annular gradings $\pm 1$. On the other hand, the lower bound from Theorem~\ref{con:akh} evaluates to $2$.

\subsection{Split sums of braids}
    
    The rank bound in annular Khovanov homology can be arbitrarily bad. Let $\beta_1$ be an $n$-braid and $\beta_2$ be an $m$-braid. The split sum of two braids has index $n+m$. Moreover, by the K\"unneth formula for the split sum of annular links and the fact that braids have annular Khovanov homology of rank one in the maximal non-trivial $k$-gradings we have that  \begin{align*}
        \rank(\AKh(\widehat{\beta_1\sqcup\beta}_2;\Z,k=n+m-2))=\rank( \AKh(\widehat{\beta}_1),k=n-2)+\rank(\AKh(\widehat{\beta}_2;\Z,  m-2)).
    \end{align*}

    \noindent Thus for appropriate choices of $\beta_1$ and $\beta_2$ we have that $\rank(\AKh(\widehat{\beta_1\sqcup\beta}_2;\Z, n+m-2))$ can be made arbitrarily large. On the other hand $\beta_1\sqcup\beta_2$ has fractional Dehn twist coefficient zero. This is because the properly embedded arc separating the first $n$ punctures from the remaining $m$ punctures is fixed by $\beta_1\sqcup\beta_2$.

    \subsection{Knot and Link Floer homology} The rank bounds in Theorem~\ref{corHFKrank} can also be arbitrarily bad. Consider, for example, the $n$-fold connect sums of the figure eight knot, $\#^{n}(K)$. The K\"unneth formula for knot Floer homology implies that $\rank(\widehat{\HFL}(\#^{n}(K),G-1))=3n$, but $\FDTC(\#^{n}(K))=0$ since the arc on the fiber surface cutting across the connected sum band is fixed by the monodromy. Similar examples hold for link Floer homology by taking connected sums on each link component and applying the K\"unneth formula. 
    More generally, note that for fibered knots in $S^3$, the rank bound is trivial since Gabai-Ortell proved that in this case $|\FDTC(\phi)|<\frac{1}{2}$~\cite{gabai1989essential}, while as seen above $\rank(\widehat{\HFL}(K,G-1))$ can be arbitrarily large in general.

\subsection{Questions} We suspect better general rank bounds are possible for both homology theories. We note that annular Khovanov homology is an invariant that is particularly well suited to studying braids, so it is natural to ask:

\begin{question}
    Is there an intrinsic proof of the annular Khovanov homology lower rank bound that does not require using link Floer homology?
\end{question}

In the absence of a proof intrinsic to Khovanov homology, we can still ask for a more natural proof as follows:

\begin{question}
    There is another invariant of annular links, \emph{annular instanton Floer homology} and a spectral sequence from the annular Khovanov homology to annular instanton Floer homology, both due to Xie~\cite{xie2021instantons}. Is there a proof of the annular Khovanov homology rank bound via annular Xie's annular instanton Floer homology? Or, more precisely, can one show that the rank of the next to top annular grading of annular instanton Floer homology of a braid closure is bounded below by a function of the fractional Dehn twist coefficient of the braid?
\end{question}

We conclude with:

\begin{question}
    Let $\widehat{\beta}$ be a braid closure in an abstract open book $(\Sigma,\phi)$. Is there a version of the Heegaard Floer homology results in this paper for the next to top grading of $\widehat{\HFL}(\widehat{\beta}\cup\partial\Sigma,[\Sigma])$?
\end{question}

\begin{question}
 Can the strategy applied in Section~\ref{sec:HFL} be used to generalize Hedden and Mark's result that there is a uniform bound on the fractional Dehn twist coefficient of fibered knots in any fixed $3$-manifold~\cite{HeddenMark} to the case of fibered links?
\end{question}

\bibliographystyle{alpha}
\bibliography{bibliography}

\end{document}